\documentclass[11pt, oneside, psamsfonts]{amsart}

\newif\ifPDF
\ifx\pdfoutput\undefined\PDFfalse
\else \ifnum \pdfoutput > 0 \PDFtrue
        \else \PDFfalse
        \fi
\fi

\usepackage[centertags]{amsmath}
\usepackage{amsfonts}
\usepackage{mathrsfs}
\usepackage{textcomp}
\usepackage{amssymb}
\usepackage{amsthm}
\usepackage{newlfont}
\usepackage[all]{xy}
\usepackage{bbm}
\usepackage{amsmath}
\usepackage{amscd}

\ifPDF
  \usepackage[pdftex]{color, graphicx}
  \usepackage[pdftex, bookmarks, colorlinks]{hyperref}
  \hypersetup{colorlinks=false}

\else
  \usepackage{color}
  \usepackage[dvips]{graphicx}
  \usepackage[dvips]{hyperref}
\fi

\usepackage[scale=0.8]{geometry}

\usepackage[pagewise, mathlines, displaymath]{lineno}

\newtheorem{thm}{Theorem}[section]
\newtheorem{cor}[thm]{Corollary}
\newtheorem{lem}[thm]{Lemma}

\newtheorem{prop}[thm]{Proposition}
\theoremstyle{definition}
\newtheorem{defn}[thm]{Definition}
\theoremstyle{definition}
\newtheorem{rem}[thm]{Remark}

\numberwithin{equation}{section}

\newcommand{\Comp}{\mathbb C}

\newcommand{\eps}{\varepsilon}

\newcommand{\aff}{\mathrm{Aff}}

\newcommand{\red}{\textcolor{red}}

\newcommand{\green}{\color{green}}

\newcommand{\cpc}{completely positive contractive map}
\newcommand{\morp}{contractive completely positive linear map}

\newcommand{\hm}{homomorphism}
\newcommand{\dt}{\delta}
\newcommand{\ep}{\varepsilon}

\newcommand{\td}{\tilde}

\newcommand{\id}{\operatorname{id}}

\newcommand{\la}{\langle}
\newcommand{\ra}{\rangle}
\newcommand{\andeqn}{\,\,\,{\text{and}}\,\,}
\newcommand{\rforal}{\,\,\,{\text{for\,\,\,all}\,\,\,}}
\newcommand{\CA}{C*-algebra}
\newcommand{\SCA}{sub-C*-algebra}
\newcommand{\af}{{\alpha}}
\newcommand{\bt}{{\beta}}

\newcommand{\N}{\mathbb{N}}
\newcommand{\Z}{\mathbb{Z}}
\newcommand{\Q}{\mathbb{Q}}
\newcommand{\R}{\mathbb{R}}
\newcommand{\C}{\mathbb{C}}
\newcommand{\T}{\mathbb{T}}

\newcommand{\bbf}{\boldsymbol}
\newcommand{\diag}{{\text{diag}}}

\newcommand{\wilog}{without loss of generality}
\newcommand{\Wlog}{Without loss of generality}

\newcommand{\beq}{\begin{eqnarray}}
\newcommand{\eneq}{\end{eqnarray}}
\newcommand{\tforal}{\,\,\,\text{for\,\,\,all}\,\,\,}
\newcommand{\tand}{\,\,\,\text{and}\,\,\,}

\newcommand{\p}{\mathfrak{p}}
\newcommand{\q}{\mathfrak{q}}
\newcommand{\Aff}{\mathrm{Aff}}

\title[Simple C*-algebras with trivial K-theory]
{The classification of simple separable KK-contractible C*-algebras with finite nuclear dimension}

\author{George A. Elliott}
\address{{\tiny{Department of Mathematics, University of Toronto, Toronto, Ontario, Canada~\ M5S 2E4}}}
\email{elliott@math.toronto.edu}

\author{Guihua Gong}
\address{{\tiny{School of Mathematics, Jilin University, Changchun, Jilin, 130012, P.R.~China, 
and
\newline \indent
Department of Mathematics, University of Puerto Rico, San Juan, PR 00936, USA }}
}
\email{ghgong@gmail.com}

\author{Huaxin Lin}
\address{
{\tiny{Department of Mathematics, East China Normal University, Shanghai, 200062, P.R.~China, 
\newline\indent  {\it and current}:
\newline\indent Department of Mathematics, University of Oregon, Eugene, OR 97403, USA}}}

\email{hlin@uoregon.edu}

\author{Zhuang Niu}
\address{{\tiny{Department of Mathematics, University of Wyoming, Laramie, WY 82071, USA}}}
\email{zniu@uwyo.edu}

\begin{document}

\begin{abstract}
The class of simple separable KK-contractible 
(KK-equivalent to $\{0\}$)  
\CA s 
which have finite nuclear dimension is shown to be classified by the Elliott invariant. In particular, the class of  \CA s $A\otimes \mathcal W$ is classifiable, where $A$ is a simple separable \CA\, with finite nuclear dimension and $\mathcal W$ is the simple inductive limit of Razak algebras with unique trace, which is bounded (see \cite{Razak-W} and \cite{Jacelon-W}).
\end{abstract}

\maketitle

\section{Introduction}

The classification of unital simple separable \CA s with finite nuclear dimension  which satisfy the UCT has been completed (see, for example, \cite{KP0}, \cite{Ph1}, \cite{GLN-TAS}, 
\cite{EGLN-DR}, and  \cite{TWW-QD}).
As is well known, the case that there exists a non-zero projection in the stabilization of the algebra follows.
In the remaining case, that the algebra {{$A$}} is stably 
projectionless
(i.e., if the algebra is finite, the case $\mathrm{K}_0(A)_+=\{0\}$),
a number of classification results are known (see \cite{Razak-W}, \cite{Tsang-W}, \cite{Robert-Cu}). 

In this paper we consider the general (axiomatically determined) case assuming 
 trivial K-theory. {{Recall that a \CA\, $A$ is said to be KK-contractible if it is 
 KK-equivalent to $\{0\}.$ 
 In the presence of the UCT, it is equivalent to say that $\mathrm{K}_i(A)=\{0\},$ $i=0,1.$}}
 From the order structure of the $\mathrm{K}_0$-group,  
  one sees that the case of stably projectionless simple \CA s is very different from 
 the unital case. In particular,  the proofs in this paper  do not depend on the unital results---and require rather different techniques.

We obtain the following classification theorem:

\theoremstyle{plain}
\newtheorem*{Theorem}{Theorem}
\begin{Theorem}[Theorem \ref{clas-thm}]
The class of KK-contractible stably projectionless simple separable  \CA s with finite nuclear dimension is classified by the invariant $(\widetilde{\mathrm{T}}(A), { \Sigma_A})$. Any \CA\ $A$ in this class is a simple inductive limit of Razak algebras.
\end{Theorem}

Here, ${\widetilde{\mathrm T}}(A)$ is the cone of lower semicontinuous traces finite on the 
Pedersen ideal $\mathrm{Ped}(A)$ of $A$, with the topology of pointwise convergence (on $\mathrm{Ped}(A)$), and $\Sigma_A$ is the norm function (the lower semicontinuous extended positive real-valued function on ${\widetilde{\mathrm T}}(A)$ defined by $\Sigma_A(\tau)=\sup\{\tau(a): a\in \mathrm{Ped}(A)_+,  \|a\|\leq 1\}$).

Consider the \CA\, $\mathcal W$, the (unique) simple inductive limit of Razak algebras with a unique 
trace (up to a multiple), which is furthermore bounded 
(see \cite{Razak-W} and \cite{Jacelon-W}; $\mathcal W$ is also sometime  called the Razak-Jacelon algebra).
We will show that ${\cal W}$ is the unique separable simple 
\CA\,  with a unique tracial state and with finite nuclear dimension which is KK-contractible.
Hence $A\otimes \mathcal W$ is KK-contractible for any amenable \CA\, $A$ (see Lemma \ref{KK-Kun}). Thus, if $A$ has finite nuclear dimension,  so that the 
\CA\, $A\otimes\mathcal W$ has finite nuclear dimension as well (see Proposition 2.3(ii) of \cite{WZ-ndim}), then 
$A\otimes\mathcal W$ is classifiable (whether it is finite---Theorem \ref{clas-thm}---or infinite---in which case by \cite{Ph1} it must be $\mathcal O_2\otimes \mathcal K$).

\theoremstyle{plain}
\newtheorem*{Corollary}{Corollary}
\begin{Corollary}[Corollary \ref{cor-WW}]
Let $A$ be a simple separable  
\CA\, with finite nuclear dimension. Then the \CA\, $A\otimes \mathcal W$ is classifiable. In particular, $\mathcal W\otimes \mathcal W \cong \mathcal W$.
\end{Corollary}

{\bf Acknowledgements.}   The greater  part of this  research
was done 
while
the second and third named  authors were at the Research Center for Operator Algebras in East China Normal University
in the summer of 2016. Much of the revision from the initial work was done 
while
the 
last three authors were there in the summer of 2017. 
All  authors acknowledge the support of the Research Center
which is partially supported by Shanghai Key Laboratory of PMMP,  The  Science and Technology Commission of Shanghai Municipality (STCSM), grant \#13dz2260400 and a NNSF grant (11531003). The first named author was partially supported by NSERC of Canada.
The second named author was partially supported by the NNSF of China (Grant \#11531003), 
the third named author was partially supported by NSF grants (DMS \#1361431 and  \#1665183), and the fourth named author was partially supported by a Simons Collaboration Grant (Grant \#317222).

\section{The reduction class ${\cal R}$, the tracially approximate point--line class $\mathcal D$, and model algebras}

Let $A$ be a \CA\,.  Denote by $\mathrm{Ped}(A)$ the Pedersen ideal.
Denote by ${\widetilde{\mathrm T}}(A)$ the topological cone of   
lower semicontinuous  positive traces defined (i.e., finite) on $\mathrm{Ped}(A)$,  with the topology 
of pointwise convergence (on the elements of $\mathrm{Ped}(A)$).
Denote by $\mathrm{T}(A)$ the set of all tracial states of $A.$  Denote by $\overline{\mathrm{T}(A)}^{\mathrm{w}}$ the  weak* closure 
of $\mathrm{T}(A)$ in the space of all positive linear functionals on $A.$ 
Let $X$ be a topological convex set, or a topological cone. Denote by $\Aff_+(X)$ the cone of all continuous  positive real-valued affine functions $f$ on $X$ 
which vanish at zero and only at that point, together with zero function. Following \cite{Robert-Cu}, let us denote  
by
$\text{LAff}_+(X)$ the 
cone
of all 
lower semicontinuous affine functions with values in $[0, \infty]$ on $X$ 
which are limits of increasing sequences of functions in $\Aff_+(X).$
We are mostly interested in the case that $X=\widetilde{\mathrm{T}}(A).$
Let $\Sigma_A\in \text{LAff}_+({\widetilde{\mathrm T}}(A))$ denote the (possibly infinite) norm function: 
$\Sigma_A(\tau)=\sup\{\tau(a): a\in \mathrm{Ped}(A)_+, \|a\|\leq 1\}.$
We shall refer to $\Sigma_A$ as the scale of $A$.

{{For $\ep>0,$ let $f_\ep\in C_0((0, \infty))_+$ (throughout the paper) such that
$f(t)=0$ if $t\in (0, \ep/2),$ $f(t)=1$ if $t\in [\ep,\infty)$ and linear in $[\ep/2, \ep).$}}

{{Let $a\in A_+,$ for each $\tau\in {\widetilde{\mathrm T}}(A),$
define ${\mathrm d}_\tau(a)=\lim_{\ep\to 0}  \tau(f_\ep(a)).$ If $e\in A_+$ is a strictly positive element,
then $\Sigma_A(\tau)={\mathrm d}_\tau(e)$ for all $\tau\in {\widetilde{\mathrm T}}(A)$
(independent of the choice of $e$).}}
If $S\subseteq 
{\widetilde{\mathrm T}}(A)\setminus \{0\}$ is a convex subset, 
denote by $\text{LAff}_{0+}(S)$ the cone $\{f|_S: f\in \text{LAff}_+(\widetilde{\mathrm{T}}(A))\}$
of restrictions to $S$ of the functions in ${\widetilde{\mathrm T}}(A)\setminus \{0\}$.
If $S$ is bounded, denote by $\text{LAff}_{b,0+}(S)$ the subset of $\text{LAff}_{0+}(S)$ consisting of those
functions bounded on $S$.
In 
the
case that ${\mathrm{T}}(A)$ is compact, let us  
denote the cone $\text{LAff}_{0+}(\mathrm{T}(A))$ just by $\text{LAff}_{+}(\mathrm{T}(A))$.

\begin{defn}
A simple \CA\ $A$ will be said to be in the reduction class, denoted by $\mathcal R$, if $A$ is separable, has continuous scale (\cite{Lncs1} and \cite{eglnp}), and $\mathrm{T}(A)\not=\O$. 
For any non-zero exact Jiang-Su stable 
separable
simple \CA\, $A,$ by Lemma 6.5 of \cite{ESR-Cuntz} 
(combined with Theorem 1.2 of \cite{Rob-0}; see Remark 5.2 of \cite{eglnp}), there is a non-zero hereditary \SCA\, $A_0\subseteq A$ 
such that $A_0$ has continuous scale---and so, if $\mathrm{T}(A)\not=\O$, belongs to the class $\mathcal R$.
In particular, as $A$ is separable and simple, {it follows from Brown's theorem (\cite{Br1}) that }
 $A\otimes {\cal K}\cong A_0\otimes {\cal K}.$
We will use the fact that 
$\mathrm{T}(A)$ is a compact base for ${\widetilde{\mathrm T}}(A)$ when $A$ 
belongs to the class $\mathcal R$ (see Theorem 5.3 of \cite{eglnp}
and Theorem 3.3 of \cite{Lncs1}). 
(By Theorem 5.3 of \cite{eglnp}, when $A$ is as above,  
with $\mathrm{T}(A)\not=\O$, and $A=\mathrm{Ped}(A)$, 
these two properties are in fact equivalent.)

\end{defn}

\begin{defn}[\cite{point-line}]
Let $E$ and $F$ be finite dimensional \CA s, and let $\phi_0, \phi_1: E \to F$ be homomorphisms (not necessarily unital). The \CA\,
$$A(E, F, \phi_0, \phi_1)=\{(e, f)\in E \oplus \mathrm{C}([0, 1], F): f(0)=\phi_0(e),\ f(1)=\phi_1(e)\}$$
will be called an Elliott-Thomsen algebra or a point--line algebra. (See \cite{point-line}. These algebras are the one-dimensional case of the non-commutative CW-complexes studied in \cite{LPcw}.) The class of point--line algebras will be denoted by $\mathcal C$.
\end{defn}

\begin{defn}[\cite{Razak-W}]\label{DRazak}
Let $k, n\in\mathbb{N}$. Consider the homomorphisms $\phi_0, \phi_1: \mathrm{M}_k(\Comp) \to \mathrm{M}_{k(n+1)}(\Comp)$ defined by 
$$\phi_0(a)=a\otimes \diag(1_n,0_k)=\mathrm{diag}(\underbrace{a, ..., a}_n, 0_k) \quad\mathrm{and}\quad \phi_1(a) = a\otimes 1_{n+1}=\mathrm{diag}(\underbrace{a, ..., a}_{n+1}).$$ The \CA\,
\beq\label{ddraz} 
R(k, n)=A(\mathrm{M}_k(\Comp), \mathrm{M}_{k(n+1)}(\Comp), \phi_0, \phi_1) \in \mathcal C
\eneq
will be called a Razak algebra.  Let $e\in R(k,n)$ be a strictly positive element.
(It is easy to check that
\beq\label{Razaklambda}
\lambda_s(R(k,n))=\inf\{d_\tau(e): \tau\in \mathrm{T}(R(k,n))\}={n\over{n+1}}
\,\,\,---\text{see\,\,\, \ref{DfS}\,\,\, below}).
\eneq

Let us also call a direct sum of such \CA s a Razak algebra, and denote this class of \CA s by $\mathcal R_{\textrm{az}}$. 
\end{defn}

\begin{defn}\label{Dbuild1}
Denote by ${\cal C}_{0}$ the class of all \CA s $A$  in ${\cal C}$ which
satisfy the following conditions:
(1) $\textrm{K}_1(A)=\{0\},$
(2) $\textrm{K}_0(A)_+=\{0\}$, and 
(3) $0\not\in \overline{\textrm{T}(A)}^\mathrm{w}.$
(What (2) says is that the \CA s  in  ${\cal C}_0$ are stably 
projectionless. What (3) says is that the spectrum of $A$ is compact.)

Denote by ${\cal C}_0^0$ the subclass of \CA s in ${\cal C}_0$ with $\textrm{K}_0(A)=\{0\}.$ 
Then  every Razak algebra is in  ${\cal C}_0^0.$ 
Let ${\cal C}_0'$ denote the class of all full hereditary \SCA s of \CA s in ${\cal C}_0$
and let ${{\cal C}_0^{0}}'$ denote the class of all full hereditary \SCA s of \CA s in ${\cal C}_0^0.$

\end{defn}

In what follows, for $r>0,$ we will use $f_r$ to denote the continuous positive function 
defined on $[0,\infty)$ by $f_r(t)=0$ if $t\in [0, r/2],$ $f_r(t)=1$ if $t\in [r, 1]$, and $f_r$ is
linear on $(r/2, r).$

\begin{defn}[see  8.1 and 8.11 of \cite{eglnp}]\label{DDD}
Recall that a simple \CA\,
is said to be in the class ${\cal D}$ (or ${\cal D}_0),$  if the following conditions hold:
There are a strictly positive element $e\in A$ with $\|e\|\le 1$ 
and a real number $1>{\mathfrak{f}}_e>0,$ 
such that for any $\ep>0,$  any 
finite subset ${\cal F}\subseteq A$, and any $a\in A_+\setminus \{0\},$  there are ${\cal F}$-$\ep$-multiplicative \cpc s $\phi: A\to A_0$ and  $\psi: A\to D$  for orthogonal 
\SCA s $A_0,\,D\subseteq A$ with $D\in {\cal C}_0$ (or ${{\cal C}_0^{0}}$), satisfying
\beq\label{DNtr1div-1++}
&&\|x- (\phi(x) + \psi(x))\|<\ep\rforal x\in {\cal F},\\ \label{DNtrdiv-2}
&&\phi(e)\lesssim a,\quad\mathrm{and} \\ \label{DNtrdiv-4}
&&\tau(f_{1/4}(\psi(e)))\ge \mathfrak{f}_e\rforal \tau\in \textrm{T}(D). 
\eneq
In fact $\mathfrak{f}_e$ can be chosen 
to be $\inf\{\tau(f_{1/4}(e)): \tau\in \textrm{T}(A)\}/2$ (see 9.2 of \cite{eglnp}).
{{Note that, if $A\in {\cal D}$ is a separable \CA\, and $B$ is a hereditary \SCA\, of $A,$ then
$B\in {\cal D}$ (see  8.6 of \cite{eglnp}).}} 
We refer to \cite{eglnp} for a detailed discussion of the definition of the class ${\cal D}.$

\end{defn}
\begin{defn}\label{DM0}
 Let us denote by ${\cal M}_0$ the class of simple separable \CA s which are inductive limits
of sequences of \CA s in ${\cal C}_0^0$, with respect to maps which are injective and 
take strictly positive elements to strictly positive elements.
This class is closed under tensoring with full matrix 
algebras (as the class of Razak algebras is), and hence is
closed under tensoring with
any unital simple AF algebra (as the tensor product of a map between
two Razak algebras and a unital map between two finite-dimensional algebras
is injective and preserves strictly positive elements if both maps
have these properties).



\end{defn}

\begin{defn}\label{DW}
Recall that the \CA\, ${\cal W}$ is the simple inductive limit of a sequence of Razak algebras
with injective connecting maps (\eqref{ddraz}---see \cite{Razak-W},
\cite{Tsang-W}, and \cite{Jacelon-W}) with a unique trace, which is bounded. By Theorem 1.1 of \cite{Razak-W},
it is the unique such C*-algebra. (Unique meaning in the Razak limit class, with unique trace which is bounded;
in particular it follows that this algebra---indeed any simple such limit of Razak algebras---is isomorphic to its
tensor product with a full matrix algebra.)
The C*-algebra ${\cal W}$ belongs to the class ${\cal M}_0$
by Lemma 3.3 of \cite{Jacelon-W}.

Furthermore,
${\cal W}$
has continuous scale (and so belongs to the class ${\cal R}$), by the first part of Proposition 5.4 of \cite{eglnp}
(the required property of strict comparison holds by {Theorem 4.6 of \cite{Toms-ASH}}). 
Hence by Theorem 3.3 of \cite{Lncs1}, ${\cal W}$ is
algebraically simple. Hence by the remark in Definition 9.5 of \cite{eglnp},  $\mathcal W\in {\cal D}_0$.
\end{defn}

{\begin{thm}\label{R2}
For any non-empty metrizable Choquet simplex $\Delta,$
there exists a 
non-unital simple \CA\, $A\in 
{\cal R}$  (Definition 2.1)
such that 
$A=\lim_{n\to\infty}(B_n, \imath_n)$
 where each $B_n$ is a finite direct sum
of copies of ${\cal W}$ and  each $\imath_n$ 
preserves strictly positive elements
(takes strictly positive elements into strictly positive elements),
every trace of $A$ is bounded, and
$$
(\mathrm{K}_0(A),  \mathrm{K}_1(A), \mathrm{T}(A))=(\{0\}, \{0\}, \Delta).
$$

Moreover, $A$ may be chosen so that $A\in {\cal M}_0$ 
%
%
%
(Definition 2.6), 
and $A\in {\cal D}_0$ 
(Definition 2.5).
\end{thm}

\begin{proof}
By 3.10 of \cite{Btrace}, there exists 
a unital  simple AF algebra  $D$ with $\mathrm{T}(D)=\Delta.$
As we shall now show, the
C*-algebra $A=D\otimes {\cal W}$ 
has the desired properties.
By \ref{DW}, $A\in {\cal M}_0.$
%
Since ${\mathrm{M}}_n({\cal W})\cong {\cal W}$ (see \ref{DW}), it follows easily that $A=\lim_{n\to\infty}(B_n,\iota_n),$ where each $B_n$ is 
a finite direct sum of copies of ${\cal W}.$  

Let $e\in {\cal W}$ be a strictly positive element. Since $W$ is algebraically simple (see \ref{DW}),
$e\in {{\mathrm{Ped}}(\mathcal W)}.$  By the definition of {the} Pedersen ideal, $1\otimes e\in {\mathrm{Ped}}(A).$ 
It follows from Proposition 5.6.2 of \cite{Pbook}
 that ${\mathrm{Ped}}(A)=A$ as the hereditary  \SCA\, generated by $1\otimes e$ is $A$ itself. 
 In other words, $A$ is algebraically simple.
{{By (the end of) Definition 9.5 of \cite{eglnp}, $A\in {\cal D}_0.$}}
Consequently, all traces of $A$ are bounded, {{and by  Theorem 9.4 of \cite{eglnp},
$A$ has strict comparison for positive elements.}}
 It is clear that $\mathrm{K}_0(A) =  \mathrm{K}_1(A) = \{0\}.$
Note that the natural affine map from the simplex $\Delta = \mathrm{T}(D)$ to
$\mathrm{T}(A)$, consisting of tensoring with the unique tracial state
of ${\cal W}$, is weak* continuous and bijective and therefore a homeomorphism.
It remains to show that $A$ has continuous scale (and so belongs to the class ${\cal R}$).
 This follows from the established facts that $A$ is algebraically simple and 
${\mathrm{T}}(A)$ is compact and 
Theorem
5.3 of \cite{eglnp}.
 \end{proof}

\begin{cor}[{\cite{Razak-W}, \cite{Tsang-W}}]
\label{Ctsang}
Let ${\widetilde T}$ be a non-zero topological cone with a 
compact 
base 
which is a metrizable Choquet simplex
and let $\gamma: {\widetilde T}\to (0, \infty]$ be a lower semicontinuous affine function,
%
zero at $0\in\widetilde{T}$, but only there.
There exists
a 
simple \CA\, $A$ which is an inductive limit {{of
Razak}} 
algebras
such that
$$
({\widetilde{\mathrm T}}(A), \Sigma_A)= ({\widetilde T}, \gamma).
$$

Moreover, $A$ may be chosen to be an inductive limit of finite direct sums of copies of ${\cal W}.$ 
\end{cor}

\begin{proof}
By Theorem 5.1 of \cite{Tsang-W}, there is a simple \CA\, $B$ which is an inductive limit 
of Razak algebras  such that ${\widetilde{\mathrm T}}(B)={\widetilde T}$ and 
the lower semicontinuous function $\omega(\tau)=\|\tau\|$ (allow values in $[0,\infty]$) is equal to $\gamma.$
Let $e\in B_+$ be a strictly positive element of $B$ with $\|e\|=1.$ Then  
$d_\tau(e)=\|\tau\|$ for each $\tau\in {\widetilde{\mathrm T}}(A).$  Thus, $({\widetilde{\mathrm T}}(B), \Sigma_B)= ({\widetilde T}, \gamma).$

To see the last part of the theorem, let  $b\in {\mathrm{Ped}}(B)_+\setminus \{0\}$ with $\|b\|=1.$
Define $\Delta=\{\tau\in {\widetilde{\mathrm T}}(B): \tau(b)=1\}.$ Since $B$ is separable, by  Proposition 2.6 of \cite{T-0-Z}, $\Delta$ is a base for ${\widetilde T}$ 
and it is a metrizable Choquet simplex.   Moreover $0\not\in \Delta.$
Choose a unital simple 
AF algebra $C$ with $\mathrm{T}(C)=\Delta$ (\cite{Btrace}).
Since $\Delta$ is compact, $\inf\{\gamma(\tau): \tau\in \Delta\}>0.$ 
It follows from Corollary I.1.4 of \cite{Alf} that there is an increasing sequence of continuous affine 
functions $f_n$ converging to $\gamma$ on $\Delta.$  By a compactness argument, we may assume 
that each $f_n\in \Aff_+(\Delta).$  (In other words, $\gamma\in {\mathrm{LAff}}_+({\tilde T}).$)
Since $\rho_C({\mathrm{K}}_0(C))$ is dense in $\Aff(\Delta)$ (see III.3.4 of \cite{BH}),  {there is}  
 an element  $a\in (C\otimes {\cal K})_+$ such that ${\mathrm d}_\tau(a)=\gamma(\tau)$ for all $\tau\in \Delta=\mathrm{T}(C)$
(see, for example,  Theorem 15.2 of   \cite{Gd}  and also  the proof of III 3.3 of \cite{BH}).  Set $\overline{a(C\otimes {\cal K})a}=C_1;$ the hereditary \SCA\,  $C_1$ is also AF. Note that  the topological cone ${\tilde{\mathrm{T}}}(A),$  being completely determined by the compact base  $\Delta$ (which does not contain zero), is isomorphic to the cone ${\tilde{\mathrm{T}}}$  which also has $\Delta$ as a base.
The tensor product $A=C_1\otimes {\cal W}$ has the desired properties.

\end{proof}

\section{A stable uniqueness theorem}\label{section-stable-uniq}

The following lemma, concerning extensions with non-unital quotient, is 
a consequence of, and in fact
equivalent to, {the second part of Corollary 16 of \cite{EK-BDF} and Theorem 2.1 of \cite{Gabe-BDF}}, in the case of
a trivial extension (which is all 
that 
we need---this restriction can
easily be removed, in the nuclear setting, by working with Choi-Effros liftings). The analogous, purely unital
setting---both quotient and extension unital---is dealt with in
Theorem 6 of \cite{EK-BDF}. As pointed out in \cite{Gabe-BDF},
the mixed case, unital quotient but non-unital extension, while
discussed in Section 16 of \cite{EK-BDF}, is not correctly
dealt with there, and a corrected statement of the 
first part of Corollary 16 of \cite{EK-BDF}
was given in Theorem 2.3 of \cite{Gabe-BDF}. Closely related earlier results are contained in \cite{DE-classification} and \cite{Lnsuniq}.

\begin{lem}\label{abs-no-1}
Let $A$ and $B$ be 
C*-algebras with $B$ stable and $A$ {separable and} non-unital. Let $\pi: A\to \mathrm{M}(B)$ be a 
faithful
homomorphism such that the 
composition with the quotient map to
$\mathrm{M}(B)/B$ 
is also faithful
and the induced (trivial) extension
is purely large
(in the sense of \cite{EK-BDF}).
Then, for any nuclear homomorphism $\sigma: A\to \mathrm{M}(B)$, 
there is a sequence  $(u_n)$ in $\mathrm{M}(\mathrm{M}_2(B))$ with $u_n^*u_n=1_{{\mathrm{M}(B)}}\otimes e_{11}$ and $u_nu_n^*=1_{\mathrm{M}(\mathrm{M}_2(B))}$
such that
\begin{enumerate}
\item $\pi(a) - u_n^*(\sigma(a)\oplus\pi(a))u_n\in B$, $n=1, 2, ...$, $a\in A$, and
\item $\lim_{n\to\infty}(\pi(a) - u_n^*(\sigma(a)\oplus\pi(a))u_n) = 0$, $a\in A$.
\end{enumerate}
\end{lem}
\begin{proof}
This follows immediately from { the second part of Corollary 16 of \cite{EK-BDF} and Theorem 2.1 of \cite{Gabe-BDF}}, in the case
of a trivial extension, with the ideal
of that theorem taken to be the C*-algebra direct sum of a countable infinity of
copies of the present ideal, $B$, and the (trivial) extension to be that induced by the
infinite repetition of the map $\pi$ into the Cartesian product of copies of the
multiplier algebra $\mathrm{M}(B)$. (This ostensibly special case of 2.1 of \cite{Gabe-BDF} is
interesting in that it is in fact a stronger result, in the case of trivial extensions---this observation is also
valid in the case of a general (non-trivial) extension, in the nuclear setting---again, on considering Choi-Effros liftings.)
\end{proof}

Let $A$ and $B$ be \CA s, let $\gamma: A \to B$ be a homomorphism, and consider the ampliated homomorphism $$\gamma_\infty:=\gamma\oplus \gamma\oplus\cdots: A \to\mathrm M(\mathcal K\otimes B),$$
where $\mathcal K$ is the algebra of compact operators on a separable infinite-dimensional Hilbert space, and $\mathrm M(\mathcal K\otimes B)$ is the multiplier algebra.

\begin{lem}\label{absorbing}
With $A$ and $B$ and $\gamma$ and $\gamma_\infty$ as above,
assume
that $A$ and $B$ are separable and $\gamma$ is faithful, and that $A$ is not unital.
If $\gamma: A\to B$ is full, i.e., 
if $\overline{B\gamma(a)B}=B$, $a\in A\setminus\{0\}$, then for any nuclear homomorphism 
$\sigma: A\to \mathrm M(\mathcal K\otimes B)$,
there is a sequence $(u_n)$ in ${\mathrm{M}}(\mathrm{M}_2(\mathcal K\otimes B))$
with $u_n^*u_n=1_{\mathrm{M}(\mathcal K\otimes B)}\otimes e_{11}$ and $u_nu_n^*=1_{\mathrm{M}_2(\mathrm{M}(\mathcal K\otimes B))}$
such that
\begin{enumerate}
\item $\gamma_\infty(a) - u_n^*(\sigma(a)\oplus\gamma_\infty(a))u_n\in \mathcal K\otimes B$, $n=1, 2, ...$, $a\in A$, and
\item $\lim_{n\to\infty}(\gamma_\infty(a) - u_n^*(\sigma(a)\oplus\gamma_\infty(a))u_n) = 0$, $a\in A$.
\end{enumerate}

\end{lem}

\begin{proof}
The lemma follows from Lemma \ref{abs-no-1}
immediately once one checks that the extension $\gamma_\infty$ is purely large.

To see $\gamma_\infty$ is purely large, let 
 $B_s=B\otimes {\cal K}.$ 
Let $c\in \gamma_{\infty}(A)+B_s$ be a non-zero element which is not in $B_s.$ One may write $c=\gamma_{\infty}(a)+b$ for some $a\not=0$ and $b\in B_s.$ Let us consider $\overline{cBc^*}.$ Replacing $c$ by $cc^*,$  one may assume that $c\ge 0.$ Therefore
one may assume that $a\ge 0.$ 
It is  clear that 
$\overline{\gamma_{\infty}(a)B_s\gamma_{\infty}(a)}\cong \overline{\gamma(a)B\gamma(a)}\otimes {\cal K}.$ 
Since $\overline{B\gamma(a)B}=B,$  $\overline{B\overline{\gamma_{\infty}(a)B_s\gamma_{\infty}(a)}B}=B.$
Thus $B\subset \overline{\gamma_{\infty}(a)B_s\gamma_{\infty}(a)}.$ 
It follows 
$\overline{\gamma_{\infty}(a)B_s\gamma_{\infty}(a)}=B_s.$ In other words, $\overline{\gamma_{\infty}(a)B_s\gamma_{\infty}(a)}$ is full in $B_s.$ In what follows, we assume $\|c\|\le 1$ and $\|a\|\le 1.$

We now follows the proof of Theorem 17 (iii) of \cite{EK-BDF}. Since there are some typos there, 
we will add some details concerning the current situation. 

 We first show that $\overline{cB_sc}$ is full. 
Put  $c_1=\gamma_\infty(a)=\gamma(a)\otimes 1$ and  choose $u_n=1\otimes v_n,$ where $v_n\in M({\cal K})$ and $(v_n)$ is   a sequence of unitaries  corresponding to some permutations of 
an orthonormal basis such that
$\lim_{n\to\infty}\|b_1u_nb_2\|=0$ for any $b_1, b_2\in B_s.$
Note  that $u_nc_1^{1/2}=c_1^{1/2}u_n$ as $c_1=\gamma_\infty(a),$ and 
$u_ncu_n^*\to c_1$   strictly in $M(B_s)$ exactly as on the page 405 of \cite{EK-BDF}.
Hence
\beq\nonumber
\hspace{0.2in}u_n(c(u_n^*b'u_n)c)u_n^*=(u_ncu_n^*)b'(u_ncu_n^*)
\to c_1 b' c_1\rforal b'\in B
\eneq
(converges in norm).
Since $\overline{c_1B_sc_1}=\overline{\gamma_\infty(a)B_s\gamma_\infty(a)}$ is full 
in $B_s,$ It follows that $\overline{cB_sc}$ is full in $B_s.$

Put 
$c'=c^{1/2}c_1c^{1/2}$ and $c''=c_1^{1/2} cc_1^{1/2}.$
Put $x=c^{1/2}c_1^{1/2}.$ 
Then $xx^*=c'$ and $x^*x=c''.$ 
Let $C_1:=\overline{c'B_sc'}$ and $C_2:=\overline{c''B_sc''}.$ 
(Note, since $c''=\gamma_\infty(a^2)+b'$ for some $b'\in B_s,$
$C_2$ is also full in $B_s.$)
We will show that 
$C_2:=\overline{c''B_sc''}$  is stable. 
Since $C_1\cong C_2,$ $C_1$ is then  also stable.
(Note also, since the  (closed) ideal generated by $\overline{xx^*B_sxx^*}$ contains 
that of $\overline{x^*xB_sx^*x}=C_2,$ $C_1$ is also full.) 
Since $0\le c'\le c^{1/2},$ this implies that $\overline{cB_sc}$ contains 
a stable \SCA\, $C_1.$  In other words, $\gamma_\infty$ is purely large.

%

To show that $C_2$ is stable, 
we write $c''=c_1^2+b_1$ for some $b_1\in B_s.$ 
We will verify condition (b) of Proposition 2.2 of \cite{HR} which 
by Proposition 2.2 and Theorem 2.1 of \cite{HR} is equivalent to the stability of a $\sigma$-unital 
\CA. 
Fix an element $a \in C_2$  with $0\le a_1\le 1$ and $\ep>0.$    Since $C_2\subset \overline{c_1B_sc_1},$ 
one may choose an integer 
$k\ge 4$ such that
\beq
\|(c'')^{1/2k}a_1^{1/2}-a_1^{1/2}\|<\ep/8 \andeqn \|(c_1^{1/k}a_1^{1/2}-a_1^{1/2}\|<\ep/8.
\eneq
Put $d=c_1^{1/k}$ and $d_1=(c'')^{1/2k}.$ Then $d-d_1\in B_s.$
Since $d=\gamma(a)^{1/k}\otimes 1,$ $u_nd=du_n.$
Recall that $\lim_{n\to\infty}\|b_1u_nb_2\|=0$ for any $b_1, b_2\in B_s.$
Hence, there is an integer $n_1\ge 1$ such that, for all $n\ge n_1,$ 
\beq\label{20208-3-1}
&&du_na_1^{1/2}\approx_{\ep/8} d_1u_na_1^{1/2}=u_nd_1a_1^{1/2}\approx_{\ep/8} u_na_1^{1/2}, \andeqn\\\label{20208-3-2}
&& a_1^{1/2}du_na_1^{1/2}\approx_{\ep/4} 0.
\eneq
Put $y_n=du_na_1^{1/2}\in C_2.$ Then  (see \eqref{20208-3-1} and \eqref{20208-3-2})
\beq
&&y_n^*y_n=a_1^{1/2}u_n^*ddu_na_1^{1/2}\approx_{\ep/4} (a_1^{1/2}u_nd)a_1^{1/2}\approx_{\ep/4} a\andeqn\\
&&(y_n^*y_n)(y_ny_n^*)=y_n^*(du_na_1^{1/2}du_na_1^{1/2})y_n=(y_ndu_n)(a_1^{1/2} du_na_1^{1/2})y_n\approx_{\ep/4} 0.
\eneq
By  2.2 (b) of \cite{HR}, $C_2$ is stable.  
As mentioned above, it follows that $C_1$ is stable and is full in $B_s.$
This shows that the extension $\gamma_{\infty}$ is purely large.
%
%
%
%
\end{proof}

{\begin{rem}
One may prove directly that the map $\gamma_\infty$ absorbs any $\sigma$ as stated in Lemma \ref{absorbing}
without using the notion of purely large.
\end{rem}
}

\begin{thm}{{{\rm{(Theorem 4.2 of \cite{DE-KK-Asy})}}}}\label{no-1-uniq-hom-0}
Let $A$ be a separable \CA\, without unit, and let $B$ be a separable \CA\,. Let $\gamma: A\to B$ be a full homomorphism. 

Let $\phi, \psi: A\to B$ be nuclear homomorphisms with $[\phi]=[\psi]$ in $\mathrm{KK}_\mathrm{nuc}(A, B)$. Then for any finite set $\mathcal F\subseteq A$ and $\eps>0$, there exist an integer $n$ and a unitary $u\in {\widetilde{{\mathrm{M}_{n+1}(B)}}}$ such that 
$$\|u^*(\phi(a)\oplus(\underbrace{\gamma(a)\oplus\cdots\oplus{\gamma}(a)}_n)u - \psi(a)\oplus(\underbrace{\gamma(a)\oplus\cdots\oplus\gamma(a)}_n)\| < \eps,\quad a\in\mathcal F.$$
\end{thm}

\begin{proof}

Since $[\phi]=[\psi]$ in $\mathrm{KK}_\mathrm{nuc}(A, B)$, one has that $[\phi, \psi, 1]=0$ in $\mathrm{KK}_\mathrm{nuc}(A, B)$ in the sense of \cite{DE-KK-Asy}.
Set $\mathrm{M}(\mathcal K(H)\otimes {B})=D,$ $\mathrm{M}({\cal K}(\Comp\oplus H)\otimes B)=D_1$, $\mathrm{M}(\mathcal K(H\oplus H)\otimes {B})=D_2$, 
{{where $H=l^2.$}}

Consider the projection $e_n=f_n\otimes 1_{\widetilde{B}}\in\mathrm{M}(\mathcal K(H)\otimes B),$ $n=1,2,...$, where $f_n$ is the projection onto the first $n$ basis elements of $H$.

Consider the unital maps $\Phi^\sim, \Psi^\sim: {\widetilde A}\to \mathrm{M}(\mathcal K(H)\otimes {B})$ defined by 
\beq
\Phi^\sim(a)=\phi(a)\oplus\gamma'_{\infty}(a)\quad\mathrm{and}\quad
\Psi^\sim(a)=\psi(a)\oplus\gamma'_{\infty}(a)\rforal a\in A,
\eneq
where $\gamma'_{\infty}(a)$ 
is considered to be a 
map from $A$ to $(1_{D}-e_1)\mathrm{M}({\cal K}(H)\otimes B)(1_{D}-e_1).$ 
One checks $[\Phi^\sim, \Psi^\sim, 1]=0$ in $\mathrm{KK}_{\mathrm{nuc}}({\widetilde A}, B).$
 By Proposition 3.6 of \cite{DE-KK-Asy}, there is a unital strictly nuclear representation $\sigma: {\widetilde A} \to \mathrm{M}(\mathcal K(H)\otimes {B})$ and a continuous path of unitaries $u: [0, \infty) \to \mathrm U(\mathcal K(H{\oplus} H)\otimes B + \Comp 1_{D_2})$ such that for any $a\in A$,  
\begin{enumerate}
\item[] $\lim_{t\to\infty} \|u_t(\Phi^\sim(a)\oplus\sigma(a){)}u_t^* - (\Psi^\sim(a) \oplus \sigma(a)) \| =0$, and
\item[] $u_t(\Phi^\sim(a)\oplus\sigma(a))u_t^* - (\Psi^\sim(a) \oplus \sigma(a)) \in\mathcal K(H\oplus H) \otimes B.$
\end{enumerate}

In particular, there is a sequence of unitaries $(u_n)$ in $\mathrm U(\mathcal K(H\oplus H)\otimes{B} + \Comp 1_{D_2})$ such that
\begin{equation}\label{uni-1}
\lim_{n\to\infty} \|u_n(\phi(a)\oplus \gamma_{\infty}'(a)\oplus\sigma(a))u_n^* - (\psi(a) \oplus \gamma_{\infty}'(a)\oplus\sigma(a)) \| =0,\quad a\in \widetilde{A}.
\end{equation}

Since $\gamma$ is full, by Lemma \ref{absorbing}, the map $\gamma_\infty$ is (non-unital) nuclearly absorbing.
Therefore $\gamma'_\infty\oplus \sigma\sim \gamma_{\infty}$; that is, there is a sequence of isometries $(v_n) $ in 
${\mathrm M}(\mathcal K(H_1\oplus H)\otimes B)$, with $v_nv_n^*=1_{D}$,
such that,  for any $a\in A$,
$$\lim_{n\to\infty} \| \gamma'_\infty(a)\oplus \sigma(a) - {v_n^*} \gamma_{\infty}(a) v_n\|=0,\quad \mathrm{and}$$
$$\gamma'_\infty(a)\oplus \sigma(a)- v_n^* \gamma_{\infty}(a) v_n \in \mathcal K(H_1\oplus H)\otimes {B},$$
where $H_1=(1_{D}-e_1)H.$

Consider the unitaries $w_n=(e_1\oplus v_{n})u_{n}(e_1\oplus v^*_n)$ in $\mathrm{M}(\mathcal K(\Comp\oplus H)\otimes B)$, in fact in $\mathcal K(\Comp\oplus H)\otimes{B} + \Comp 1_{D_1}$.
For any {contraction} $a\in A$,
\begin{eqnarray*}
&& \| w_n(\phi(a)\oplus\gamma_{\infty}(a))w^*_n - \psi(a)\oplus\gamma_{\infty}(a) \| \\
& = & \| (e_1\oplus v_n)u_n(e_1\oplus v^*_n)(\phi(a)\oplus\gamma_{\infty}(a))(e_1\oplus v_n)u^*_n(e_1\oplus v^*_n) - \psi(a)\oplus\gamma_{\infty}(a) \| \\
&\approx & \| (e_1\oplus v_n)u_n(\phi(a)\oplus(\gamma_{\infty}'(a)\oplus \sigma(a)))u^*_n(e_1\oplus v^*_n) - \psi(a)\oplus\gamma_{\infty}(a) \| \\
& \approx & \| (e_1\oplus v_n)(\psi(a)\oplus(\gamma'_\infty(a)\oplus \sigma(a)))(e_1\oplus v^*_n) - \psi(a)\oplus\gamma_{\infty}(a) \| \\
& \approx & \| \psi(a)\oplus\gamma_{\infty}(a) - \psi(a)\oplus\gamma_{\infty}(a) \| = 0.
\end{eqnarray*}
That is,  there is 
a sequence of unitaries $(w_k)$ in $\mathrm U(\mathcal K(\Comp\oplus H)\otimes{B} + \Comp 1_{D_1})$ such that
$$\lim_{k\to\infty} \|w_k(\phi(a)\oplus\gamma_{\infty}(a))w_k^* - (\psi(a) \oplus \gamma_{\infty}(a)) \| =0,\quad a\in A.$$ 

Since $w_k \in \mathcal K({\Comp}\oplus H)\otimes B + \Comp 1_{D_1}$, one has that $[w_k, e_n]\to 0$, as $n\to\infty$. Then, for sufficiently large $k$, and then sufficiently large $n$, the element $e_nw_ke_n$ of ${\mathrm{M}_{n}(B)}+\Comp 1_{n}$ can be perturbed to a unitary $u$ verifying the conclusion of the theorem.
\end{proof}

\begin{rem}
The unital version of \ref{no-1-uniq-hom-0} can be found in {4.2} of \cite{DE-KK-Asy} (see an earlier version  in \cite{Lnsuniq}). 
A different approach could also be found in an earlier version of this paper (see \cite{GLpp1}).
\end{rem}

\begin{prop}[Proposition 2.1 of \cite{Bla-WEP}]\label{lem-of-bla}
Let $A$ be a separable \CA\, (with or without unit). Then there is a countable subset $S$ of $A$ such that if $J$ is any ideal of $A$, then $S\cap J$ is dense in $J$.
\end{prop}

\begin{lem}\label{sep-cond}
Let $D$ be a \CA\,. Let $A\subseteq D$ be a separable sub-\CA\, such that 
$$
\overline{DaD}=D\tforal a\in A\setminus\{0\},
$$ 
and let $B\subseteq D$ be another separable sub-\CA\,. Then, there is a separable sub-\CA\, $C$ of $D$ such that
\beq
A, B \subseteq C\tand
\overline{CaC}=C\tforal a\in A\setminus\{0\}
\eneq
(i.e., such that the inclusion map $A\to C$ is full).
\end{lem}
\begin{proof}
The proof follows an idea of Blackadar.
Applying Proposition \ref{lem-of-bla}, one obtains a  countable set $$\{a_0, a_1, a_2, ...\}\subseteq A$$ such that
$\{a_0, a_1, a_2, ...\}\cap J$ is dense in $J$ for any ideal $J$ of $A$. We may assume that $$a_0=0,\quad \textrm{so that}\quad a_j\neq 0,\quad j=1, 2, ...\ .$$

Set $$C_1=\textrm{C*}(A \cup B)\subseteq D.$$ It is clear that $C_1$ is separable. Pick a dense set $\{c_1, c_2, ...\}$ in $C_1$. Since $\overline{Da_jD}=D$, $j=1, 2, ...$, for any $\eps>0$ and any $c_i$, there are finitely non-zero sequences $x_{c_i, a_j, \eps, 1}, x_{c_i, a_j, \eps, 2}, ...$ and $y_{c_i, a_j, \eps, 1}, y_{c_i, a_j, \eps, 2}, ...$ in $D$ such that 
$$ \| c_i - (x_{c_i, a_j, \eps, 1} a_j y_{c_i, a_j, \eps, 1}+  x_{c_i, a_j, \eps, 2}a_jy_{c_i, a_j, \eps, 2}+ \cdots )\|<\eps.$$

Set $$C_2=\textrm{C*}(C_1, x_{c_i, a_j, \frac{1}{n}, k}, y_{c_i, a_j, \frac{1}{n}, k} : i, j, n, k=1, 2, ... ).$$
Then
$$\overline{C_2a_jC_2} \supseteq C_1,\quad j=1, 2, ...\ .$$
Repeating the construction above, one obtains a sequence of separable \CA s $$C_1\subseteq C_2\subseteq \cdots \subseteq C_n \subseteq \cdots\subseteq D$$
such that 
$$\overline{C_{n+1}a_jC_{n+1}} \supseteq C_n,\quad j=1, 2, ...,\ n=1, 2, ...\ .$$
Setting $\overline{\bigcup_{n=1}^\infty C_n} = C$, one has 
$$\overline{Ca_jC} = C,\quad j=1, 2, ...\ .$$

Then the separable sub-\CA\, $C$ satisfies the requirements of the lemma. Indeed, let $a \in A\setminus\{0\}$. Consider the ideal $J:=\overline{CaC} \cap A$. Since $a\in J$, one has $J\neq\{0\}$. By Proposition \ref{lem-of-bla}, one has that $\{a_0, a_1, a_2, ...\}\cap J$ is dense in $J$, and in particular, the ideal $J$ contains some $a_j\neq 0$. Since 
$C=\overline{Ca_jC} \subseteq \overline{CJC}=\overline{CaC}$, one has $\overline{CaC}=C$, as desired.
\end{proof}

\begin{rem}
If $A$ is simple, then, in the proof above, one only needs to pick one non-zero element of $A$ and does not need Proposition \ref{lem-of-bla}. 
\end{rem}


\begin{lem}\label{Lsep1}
Let $B$ be a 
$\sigma$-unital \CA\, and let $A$ be a separable amenable
\CA\, which is a \SCA\,  of  $B.$ 
Let $h_1, h_2: A\to B$ be \hm s such that
$[h_1]=[h_2]$ in $\mathrm{KK}(A,B)$ {{  (which we  
regard as  $\mathrm{KK}^1(A, SB)$)}}.
There exists a separable \SCA\, $C\subseteq B$ such that
$A, h_1(A), h_2(A)\subseteq C$ and $[h_1]=[h_2]$ in $\mathrm{KK}(A, C).$
If the inclusion of $A$ in $B$ is full (in other words, $\overline{BaB}=B$ for any $0\neq a\in A$), then $C$ may be chosen such that 
the inclusion of $A$ in $C$ is full.
\end{lem}

\begin{proof}
Consider the extensions $\tau_1, \tau_2: A\to \mathrm{M}(\mathrm{S}B)/\mathrm{S}B$ given by the mapping tori
\beq\label{Lsep-n1}
M_{h_i}=\{(f,a)\in \mathrm{C}([0,1], B)\oplus A: f(0)=a\ \mathrm{and}\ f(1)=h_i(a)\},\,\,\, i=1,2.
\eneq
Let $H_i: A\to \mathrm{M}(\mathrm{S}B)$ be a completely positive contractive lifting of $\tau_i,$ $i=1,2.$ 
There are a monomorphism  $\phi_0: A\to \mathrm{M}(\mathrm{S}B\otimes {\cal K})$
and a unitary $w\in \mathrm{M}(\mathrm{S}B\otimes {\cal K})$ such that
$$
w^*(H_1(a)\oplus\phi_0(a))w-(H_2(a)\oplus\phi_0(a))\in \mathrm{S}B\otimes {\cal K}\rforal a\in A.
$$
Let $C_{000}$ denote the (separable) \SCA\, of $B$ generated by $A$, $h_1(A)$, and $h_2(A).$

Choose a system of matrix units $(e_{i,j})$ for ${\cal K}$, and 
choose
a dense sequence $(t_n)$ in $(0,1)$.
Choose an 
increasing
approximate unit $(E_n)$ for $\mathrm{S}B\otimes {\cal K}$
such that $E_n\in \mathrm{M}_{k(n)}(\mathrm{S}B),$ $n=1,2,....$

Denote by $D_0$ the (separable) \SCA\, of $\mathrm{S}B\otimes {\cal K}$
generated by
\beq\label{Lsep-1}
w^*(\diag(H_1(a), \phi_0(a))w-\diag(H_2(a), \phi_0(a)),\quad a\in A.
\eneq

Denote by $D_{00}$ the \SCA\, of $\mathrm{S}B\otimes {\cal K}$ generated by
$$
\{E_n, wE_n, E_nw, {E_n\phi_0(a), \phi_0(a)E_n}: a\in A, n\in \N\}.
$$
Let $D_{000}$ denote the (separable) \SCA\, of $\mathrm{S}B\otimes {\cal K}$ generated by $D_{00}$ and $D_0.$
Denote by $\pi_{t}: \mathrm{S}B\otimes {\cal K}\to B\otimes {\cal K}$
the point evaluation at $t\in(0, 1)$, and by $C_{00}$ the \SCA\, of $B\otimes {\cal K}$ generated by
$$
\{\pi_{t_n}(D_{000})+C_{000}\otimes {e_{1,1}}: n=1,2,...\}.
$$
Denote by $C_{0,n}\subseteq B\otimes \mathcal K$ the \SCA\, generated by
$
{\{(1\otimes e_{1,i})C_{00}(1\otimes e_{j,1}):1\le i,\, j\le n\}}, 
$
$n=1,2,....$  
{{Let $C'$ denote the (separable) sub-C*-algebra of $B$ generated by $\bigcup_{n=1}^{\infty} C_{0,n}.$}} 
Choose
a separable \SCA\, $C$ of $B$ containing $A$ and $C'$. By Lemma \ref{sep-cond}, if the inclusion $A\to B$ is full, then we 
may
choose $C$ such that the inclusion $A\to C$ is full. 
Note that (as $C_{000}\subseteq {C_{0,1}}$), $h_1(A), h_2(A)\subseteq C'\subseteq C$.
Consider the \SCA\, $C_{1}=C\otimes {\cal K}$ of $B\otimes\mathcal K$.
Fix
\beq\label{Lsep-n2}
b\in \{E_n, wE_n, E_nw, E_n\phi_0(a), \phi_0(a)E_n: a\in A, n\in \N\}\subseteq \mathrm{S}B\otimes\mathcal K.
\eneq
Keep in mind that $E_n\in \mathrm{S}B\otimes {\cal K}=\mathrm{C}_0((0, 1), B\otimes\mathcal K)$, in particular, $E_n(0)=E_n(1)=0,$ $n\in \N.$
Then, for each $t_n,$ $n=1, 2, ...$,
$$
\pi_{t_n}(b)\in \pi_{t_n}(\mathrm{S}C_{1}\otimes {\cal K}).
$$
It follows that
$b\in \mathrm{S}C_{1}\otimes\mathcal K.$  To see this, 
fix $\ep>0,$ 
and
choose a finite sequence $t_{n_i}\in (t_n)$, $i=1,2,...,k,$ such that
$$
0=t_{n_0}<t_{n_1}<t_{n_2}<\cdots {  <} t_{n_k}<t_{n_{k+1}}=1
$$
and 
$$
\|b(t)-b(t_{n_i})\|<\ep/4\rforal t\in(t_{n_i}, t_{n_{i+1}}),\,\,i=0,1,...,k.
$$
Set
$$c(t) = \left\{
\begin{array}{ll}
\displaystyle{(\frac{t}{t_{n_1}}) b(t_{n_1})}, & t\in (0,t_{n_1}) \\
\displaystyle{(\frac{t_{n_{i+1}}-t}{t_{n_{i+1}}-t_{n_i}})b(t_{n_i})+(\frac{t-t_{n_i}}{t_{n_{i+1}}-t_{n_i}})b(t_{n_{i+1}})}, & t\in [t_{n_i}, t_{n_{i+1}}), i=1,2,...,k, \\
\displaystyle{(\frac{1-t}{1-t_{n_k}})b(t_{n_k})}, & t\in (t_{n_k},1).
\end{array}
\right.$$
Then $c\in \mathrm{S}C_{1}\otimes {\cal K}.$
On the other hand,
$$
\|b(t)-c(t)\|<\ep\rforal t\in (0,1).
$$
Since $\eps>0$ is arbitrary, this implies that $b\in \mathrm{S}C_{1}\otimes\mathcal K.$

In particular, $E_n\in \mathrm{S}C_{1}\otimes\mathcal K$ and
$(E_n)$ is an approximate unit for
$\mathrm{S}C_{1}\otimes\mathcal K$, and so $\mathrm{M}(\mathrm{S}C_1\otimes\mathcal K) \subseteq \mathrm{M}(\mathrm{S}B\otimes\mathcal K)$.
Since also $wE_n, E_nw \in \mathrm{S}C_{1}\otimes {\cal K}$, and $w\in\mathrm{M}(\mathrm{S}B\otimes\mathcal K)$, 
it follows that $w\in \mathrm{M}(\mathrm{S}C_{1}\otimes {\cal K}).$  Similarly, since
$\phi_0(a)E_n, E_n\phi_0(a) \in \mathrm{S}C_{1}\otimes {\cal K}$ for all $a\in A$ and $\phi_0(w)\in\mathrm{M}(\mathrm{S}B\otimes\mathcal K)$, we may view $\phi_0$ as a monomorphism from
$A$ to $\mathrm{M}(\mathrm{S}C_{1}\otimes\mathcal K)\subseteq \mathrm{M}(\mathrm{S}B\otimes\mathcal K).$

A similar argument shows that $D_0\subseteq \mathrm{S}C_{1}\otimes\mathcal K.$

We now have
$$w, H_1(a)\oplus\phi_0(a), H_2(a)\oplus \phi_0(a) \in\mathrm{M}(\mathrm{S}C_1\otimes\mathcal K),$$
$$
w^*(H_1(a)\oplus\phi_0(a))w- (H_2(a)\oplus\phi_0(a))\in \mathrm{S}C_{1}\otimes {\cal K}
$$
for all $a\in A.$
This implies that $[h_1]=[h_2]$ in $\mathrm{KK}(A, C).$
\end{proof}

In a similar way (using Lemma \ref{sep-cond}), one also has the 
following result:
\begin{lem}\label{Lsep2}
Let $A, D$ be C*-algebras, with $A$ separable. Let $\phi, \psi, \sigma: A \to D$ be homomorphisms such that 
\begin{enumerate}
\item[] $[\phi] = [\psi]$ in $\mathrm{Hom}_{\Lambda}({\underline{\mathrm{K}}(A), \underline{\mathrm{K}}(D)})$, and
\item[] $\overline{D\sigma(a)D} = D$,  $0\neq a\in A$. 
\end{enumerate}
Then there is a separable sub-\CA\, $C\subseteq D$ such that
\begin{enumerate}
\item[] $\phi(A), \psi(A), \sigma(A)\subseteq C$,
\item[] $[\phi] = [\psi]$ in {$\mathrm{Hom}_{{\Lambda}}(\underline{\mathrm{K}}(A), \underline{\mathrm{K}}(C))$}, and
\item[] $\overline{C\sigma(a)C} = C$, $0\neq a\in A$.
\end{enumerate}
\end{lem}
\begin{proof}
The proof is in the same spirit 
as
that of \ref{Lsep1}. We sketch 
it
below.
Since $A$ is separable, it is easy to find a separable \CA\, $B_1\subseteq D$ such that
$\phi(A), \psi(A)\subseteq B_1$ and  $\phi_{*i}=\psi_{*i}$ ($i=0,1$) viewing $\phi$ and $\psi$  as maps from 
$A$ to $B_1.$ For each $m\ge 2,$ let $C_m\cong \mathrm{C}_0(X_m)$ for some locally compact and $\sigma$-compact metric space
$X_m$ such that $\mathrm{K}_0(C_m)=\Z/m\Z$ and $\mathrm{K}_1(C_m)=\{0\}.$ Denote by $Y_m$ the one-point compactification 
of $X_m$ with the point $\xi_0$ as the additional point.   Note $Y_m$ is separable. 

Let $\phi^{(m)}, \psi^{(m)}: C_m\otimes A\to C_m\otimes D$ be the 
natural
extensions of $\phi$ and $\psi.$
Suppose that $p$ and $q$ are two projections in $\mathrm{M}_l((C_m\otimes D)^\sim)$ for some $l\ge 1$
such that there exists $v\in \mathrm{M}_{l+k}((C_m\otimes D)^\sim)$ with $v^*v=p\oplus 1_k$ and 
$vv^*=q\oplus 1_k.$  We now view $p, q, v$ as functions in $C(Y_m, \mathrm{M}_{l+k}({\widetilde D})).$
Let $(y_n)$ be a dense 
sequence of $Y_m$ such that $y_1=\xi_0.$ Consider
the
 \SCA\, $B_{m,0}''$ of  $\mathrm{M}_{l+k}({\widetilde D})$ which is generated by $p(y_n)$, $q(y_n)$, and $v(y_n)$ for all $n\ge 1.$ Then $B_{m,0}''$ is separable. One then easily constructs a separable \SCA\, $B_{m,0}'$ of $D$ 
such that $p, q, v$ are in $\mathrm{M}_{l+k}((C_m\otimes B_{m,0}')^\sim).$ 
Similarly, if $u, w$ are unitaries in $\mathrm{M}_l((C_m\otimes D)^\sim)$ which are connected by a continuous path of unitaries,
then one may also 
construct
a separable \SCA\, $B_{m,1}'$ of $D$ such that
$u, w$ are in $\mathrm{M}_l((C_m\otimes D)^\sim)$ and are connected by a continuous path of unitaries in 
$\mathrm{M}_l((C_m\otimes D)^\sim).$

From this, one concludes that  there is  a separable \SCA\, $B_m\subseteq D$ such that 
$\phi^{(m)}(C_m\otimes A), \psi^{(m)}(C_m\otimes A)\subseteq C_m\otimes B_m$ and 
$\phi^{(m)}_{*i}=\psi^{(m)}_{*i}$ ($i=0,1$) viewing $\phi^{(m)}$ and $\psi^{(m)}$ as maps from $C_m\otimes A$ to 
$C_m\otimes B_m,$ $m=2,3,...$ 
Let $D_1$ be the \SCA
generated by $B_m,$ $m=1,2,....$ Then $D_1$ is separable.
By \ref{Lsep1}, there is a separable \SCA\, 
$C\supseteq D_1, \sigma(A)$ such that 
${\overline{C\sigma(a)C}}=C$ for all $a\in A\setminus \{0\}.$
Note now that $[\phi]=[\psi]$ in ${\mathrm{Hom}}_{\Lambda}(\underline{\mathrm{K}}(A), \underline{\mathrm{K}}(C))$
as $\phi^{(m)}_{*i}=\psi^{(m)}_{*i}$ 
with 
$\phi^{(m)}$ and $\psi^{(m)}$ 
viewed
as maps from 
$A$ into $C.$

\end{proof}

\begin{defn}\label{fullunif}
Let $M: {{(A_+\setminus \{0\})}} \times (0, 1) \to (0, +\infty)$ and 
$N: {{(A_+\setminus \{0\})}} \times (0, 1) \to \mathbb N$  be maps.
A positive map $\phi: A\to B$ will be said to be $(N,M)$-full if for any $1>\eps>0$, any $a\in A_+\setminus\{0\}$, and any $b\in B^+$ with $\|b\| \leq 1$, there are $b_1, b_2, ..., b_{N(a, \eps)}\in B$ with $\|b_i\| \leq M(a, \eps)$, $i=1, 2, ..., N(a, \eps)$, such that
$$\|b- (b_1^*\phi(a)b_1 + b_2^*\phi(a)b_2 + \cdots + b_{N(a, \eps)}^*\phi(a)b_{N(a, \eps)})\|\le \eps.$$

Write $F:=(N,M): (A_+\setminus \{0\})\times (0,1)\to \N\times \R_+$, and let ${\cal H}\subseteq A_+\setminus \{0\}.$
A positive map $L: A\to B$ will be said to be $F$-${\cal H}$-full if, 
for any $a\in {\cal H}$, any $b\in B_+$ with $\|b\|\le 1,$ and any $\ep>0,$
there are $x_1, x_2,...,x_m\in B$ with
$m\le N(a, \ep)$ and $\|x_i\|\le M(a, \ep)$ 
such that
\beq\label{localfull-1}
\|\sum_{i=1}^mx_i^*L(a)x_i-b\|\le \ep.
\eneq

The map $L$ will be said to be {\it uniformly} $(N,M)$-full if $N$ and $M$ are independent of $\ep$, 
(i.e., $N: A_+\setminus \{0\}\to \N$ and $M: A_+\setminus \{0\}\to \R_+\setminus \{0\}$) and to be {\it strongly uniformly} $(N,M$)-full, if, in addition, $\ep$
can be replaced by zero. 
The map $L$ will be said to be uniformly $F$-${\cal H}$-full, if $F$ is independent of $\ep.$ 

{{Let $B$ be any \CA\, and $D\subset B$ be a $\sigma$-unital \SCA.
Let $F=(N,M): (A_+\setminus \{0\})\times (0,1)\to \N\times \R_+$ be a map described above.
We would like to make the following remark:  If $L:  A\to D$ is $(F, {\cal H})$-full, 
then $j\circ L: A\to \overline{DBD}$ is $(F^{(1/2)}, {\cal H})$-full,
where $F^{(1/2)}((a, \ep))=F((a, \ep/2))$ and $j: D\to \overline{DBD}$ is the embedding.
In fact, for any $\ep>0,$ given any $b\in \overline{DBD}_+$ with $\|b\|\le 1,$ there is $d\in D_+$ with $\|d\|=1$ such 
that $\|b^{1/2} d b^{1/2}-b\|<\ep/2.$   Fix $a\in {\mathcal H \subseteq} A_+\setminus \{0\}.$ 
There  are $x_1, x_2,...,x_m$ with $m\le N(a, \ep/2)$ and $\|x_i\|\le M(a, \ep/2)$ such that
$$
\|\sum_{i=1}^m x_i^* L(a)x_i-d\|\le \ep/2.
$$
It follows that, for $a\in {\cal H},$ 
$$
\|\sum_{i=1}^m b^{1/2}x_i^*L(a) x_ib^{1/2}-b\|<\ep/2+\ep/2=\ep.
$$
Note that $\|x_ib^{1/2}\|\le \|x_i\|.$  So $j\circ L$ is ${{(F^{(1/2)},{\cal H})}}$-full. Note also, if $F(a, t)=F(a, t')$
for all $t, t'\in (0,1),$ (uniformly full), then $F^{(1/2)}=F,$ whence $j\circ L$ is still $(F, {\cal H})$-full.}}

\end{defn}

Let $A$ and $B$ be \CA s and $d: A\to B$ a map.
For each integer $n\ge 1,$ denote by $d_n: A\to {\mathrm{M}}_n(B)$ the map 
$d_n: a\mapsto \underbrace{d(a)\oplus d(a)\oplus\cdots\oplus d(a)}_n$ (for $a\in A$).

\begin{thm}{\rm{(cf. Theorem 3.9 of \cite{Lnauct})}} \label{T39}
Let $A$ be a separable amenable \CA\, and 
 let $B$ be a $\sigma$-unital \CA.
Let $h_1, h_2: A\to B$ be \hm s such that
$$
[h_1]=[h_2]\,\,\, {\text{in}}\,\,\, \mathrm{KL}(A, B).
$$
Suppose that there is an embedding $d: A\to B$ which is $(N, M)$-full for some $N: A_+\setminus \{0\} \times (0, 1) \to \mathbb N$ and $M: A_+\setminus \{0\} \times (0, 1) \to \R_+\setminus \{0\}.$

Then, for any $\ep>0$ and finite subset ${\cal F}\subseteq A,$  there is an integer $n\ge 1$ and
a unitary $u\in {\widetilde{\mathrm{M}_{n+1}(B)}}$  such that
\beq
\|u^*{\mathrm{diag}}(h_1(a), d_n(a))u-{\mathrm{diag}}(h_2(a), d_n(a))\|<\ep\rforal a\in {\cal F}.
\eneq
\end{thm}

\begin{proof}
Write $C=\prod_{k=1}^{\infty} B,$  $C_0=\bigoplus_{k=1}^{\infty} B$, and let
$\pi: C\to C/C_0$ denote the quotient map.
Let $H_i=(h_i): A\to C$ be defined by $H_i(a)=(h_i(a))$ for all $a\in A,$ $i=1,2.$
Define $H_0: A\to C$ by $H_0(a)=(d(a))$ for all $a\in A.$
It follows from 3.5 of \cite{Lnauct} that
\beq\label{T39-2}
[\pi\circ H_1]=[\pi\circ H_2]\,\,\,{\text{in}}\,\,\, \mathrm{KK}(A, C/C_0).
\eneq
Since $d: A\to B$ is 
$(N, M)$-full, for any $a\in A_+\setminus\{0\},$ let
$M(a,\ep)$ and $N(a,\ep)$ be as in Definition \ref{fullunif}.
Let $(b_n)\in (\prod_{n=1}^{\infty} B)_+$ with
$\|(b_n)\|\le 1.$
Then, for any $\ep>0,$ 
there are $b_{1,n}, b_{2,n},..., b_{N(a, \ep),n}\in B$ with
$\|b_{i,n}\|\le M(a,\ep)$ such that
$$
\|\sum_{i=1}^{N(a,\ep)}{ b_{i, n}^*}{d(a)}{ b_{i, n}}-b_n\|<\ep.
$$
Set $(b_{i,n})=z_i,$ $i=1,2,...,N(a,\ep).$ Then
$\|z_i\|=\sup\{\|b_{i,n}\|: n\in \N\}\le M(a,\ep),$ $i=1,2,...,N(a, {\eps}).$
Therefore, $z_i\in \prod_{n=1}^{\infty}B.$
We have
$$
\|\sum_{i=1}^{N(a, {\eps})} z_i^*H_0(a) z_i-(b_n)\|<\ep.
$$
This implies that the map $H_0: A\to \prod_{n=1}^{\infty}B$ is full.

It follows that the embedding  $\pi\circ H_0: A\to C/C_0$ is  full.
Combining this with \eqref{T39-2}, and applying Lemma \ref{Lsep1}, we obtain
a separable \SCA\, $D\subseteq C/C_0$ such that
$\pi\circ H_0(A), \pi\circ H_1(A), \pi\circ H_2(A)\subseteq D,$ the map
$\pi\circ H_0: A\to D$ is  full, and
$$
[\pi\circ H_1]=[\pi\circ H_2]\,\,\, {\text{in}}\,\,\, \mathrm{KK}(A, D).
$$
By Theorem \ref{no-1-uniq-hom-0}
there exist an integer $n\ge 1$ and a unitary $U\in {{\mathrm{M}_{n+1}(D)^\sim}}$ such that
$$
\|U^*\diag(\pi\circ H_1(a), d_n(\pi\circ H_0(a)))U-\diag(\pi\circ H_2(a), d_n(\pi\circ H_0(a)))\|<\ep,\quad a\in {\cal F}.
$$
Note that $U\in {{\mathrm{M}_{n+1}(C/C_0)^\sim}}.$
Therefore (by stable relations) there is a unitary $V=(v_k)\in {{M_{n+1}(C)^\sim}}$ such that
$\pi(V)=U.$
Then, for all sufficiently large $k,$
$$
\|v_k^*\diag(h_1(a), d_n(a))v_k-\diag(h_2(a), d_n(a))\|<\ep
\rforal a\in {\cal F}.
$$
Thus, the unitary $u=v_k$ with $k$ sufficiently large satisfies the conclusion of the theorem.
\end{proof}

\begin{defn}[Definition 2.1 of \cite{GL1}]\label{Blbm}
Fix a map $r_0: \N\to \Z_+,$  a map $r_1: \N\to \Z_+,$ a map $T: \N\times \N\to \N,$ and integers
$s\ge 1$ and $R\ge 1.$
  We shall say a   \CA\, $A$ belongs to the class
${\boldsymbol{C}}_{(r_0, r_1,T, s, R)}$ if

(a) for any integer $n\ge 1$ and any {pair} of projections $p,\, q\in {\mathrm{M}}_n({\widetilde A})$ with
$[p]=[q]$ in $\mathrm{K}_0(A),\,\,\,$  ${{p\oplus 1_{M_{r_0(n)}(\widetilde A)}}}$ and
$q\oplus 1_{\mathrm{M}_{r_0(n)}({\widetilde A})}$ are Murr{a}y-von Neumann equivalent, and
moreover, if $p\in \mathrm{M}_n({\widetilde A})$ and $q\in \mathrm{M}_m({\widetilde A})$ and
$[p]-[q]\ge 0,$ then there exists $p'\in \mathrm{M}_{n+r_0(n)}({\widetilde A})$
such that $p'\le p\oplus 1_{M_{r_0(n)}}$ and $p'$ is equivalent to $q\oplus 1_{M_{r_0(n)}};$

(b) if $k\ge 1,$ and $x\in K_0(A)$ such that  $-n[1_{\widetilde A}]\le kx \le n[1_{\widetilde A}]$ for some integer $n\ge 1,$ then
$${-}T(n,k)[1_{\widetilde A}]\le x \le T(n,k)[1_{\widetilde A}];$$

(c) the canonical map $\mathrm{U}(\mathrm{M}_s({\widetilde A}))/\mathrm{U}_0(\mathrm{M}_s({\widetilde A}))\to \mathrm{K}_1(A)$ is surjective;

(d) {i}f $u\in {\mathrm{U}(\mathrm{M}_n({\widetilde A}))}$ and $[u]=0$ in $\mathrm{K}_1({\widetilde A}),$ then $u\oplus 1_{\mathrm{M}_{r_1(n)}}\in \mathrm{U}_0(\mathrm{M}_{n+r_1(n)}({\widetilde A}));$

(f) {${\text{cer}}(\mathrm{M}_m({\widetilde A}))\le R$} 
for all $m\ge 1$ \,\,\,{{(\text{see 2.15 of \cite{GLN-TAS}, for example}).}}

If $A$ has stable rank one, and (a) to (f) hold, then they hold with $r_0=r_1=0$. \end{defn}

Let $A$ be a unital \CA\, and let $x\in A.$ Suppose 
that $\|x^*x-1\|<1/2$ and $\|xx^*-1\|<1/2.$ Then $x$ is invertible and $x|x|^{-1}$ is a unitary.
Let us use $\lceil x \rceil $ to denote $x|x|^{-1}.$ We will use this notation in the next statement (see \ref{Lauct-1}).

\begin{thm}[{cf.}~5.3 of \cite{Lnsuniq}, Theorem 3.1 of \cite{GL1},
Theorem 4.15 of \cite{DE-classification}, 5.9 of \cite{Lnauct}, and
Theorem 7.1 of \cite{Ln-hmtp}]\label{Lauct2}
Let $A$ be a 
non-unital 
separable amenable  \CA\, which satisfies the UCT,  let
$r_0, r_1: \N\to \Z_+,$ $T: \N\times \N\to \N$ be three maps, let $s, R\ge 1$ be integers,
and let $F: A_+\setminus \{0\}\to \N\times \R_+\setminus \{0\}$  and ${\boldsymbol{L}}: \mathrm{U}(\mathrm{M}_{\infty}({\widetilde A}))\to \R_+$ be  two additional  maps.
  For any $\ep>0$ and any finite subset ${\cal F}\subseteq A,$ there exist
$\dt>0,$ a finite subset ${\cal G}\subseteq A,$
a finite subset ${\cal P}\subseteq \underline{\mathrm{K}}(A),$ a finite subset
${\cal U}\subseteq \mathrm{U}(\mathrm{M}_{\infty}({\widetilde A})),$ a finite subset ${\cal H}\subseteq A_+\setminus \{0\}$, and an integer $K\ge 1$ satisfying the following condition:
For any two ${\cal G}$-$\dt$-multiplicative \morp s $\phi, \psi: A\to B,$  where
$B\in {\boldsymbol{C}}_{r_0, r_1, T, s, R},$ and any
${\cal G}$-$\dt$-multiplicative \morp\, $\sigma: A\to \mathrm{M}_l(B)$ (for any integer $l\ge 1$) which is (uniformly)
{{$T$-${\cal H}$-full}} and  such that
\beq\label{Lauct-1}
\hspace{0.3in} &&{\mathrm{cel}}(\lceil \phi(u)\rceil  \lceil \psi(u^*)\rceil)\le {\boldsymbol{L}}(u)\tforal u\in {\cal U},\ \textrm{and}\\\label{Lauct-1n}
&&
[\phi]|_{\cal P}=[\psi]|_{\cal P},
\eneq
{\rm (}see 1.1 of \cite{GL1} and \cite{Rn}  for the  definition  of $cel${\rm)} there exists a unitary $U\in {\widetilde{\mathrm{M}_{1+Kl}(B)}}$ such that
\beq\label{Lauct-2}
\|{\mathrm{Ad}}\, U\circ (\phi\oplus \sigma_K)(a)-(\psi\oplus \sigma_K)(a)\|<\ep\tforal a\in {\cal F},
\eneq
where
\vspace{-0.14in} $$
\sigma_K:=\overbrace{\sigma\oplus\sigma\oplus\cdots\oplus \sigma}^K : A \to \mathrm{M}_{Kl}(B).
$$
\end{thm}

\begin{proof}
Let us also use $\phi$ and $\psi$ for $\phi\otimes {\text{id}}_{\mathrm{M}_m}$ and
$\psi\otimes {\text{id}}_{\mathrm{M}_m},$ respectively.
Fix $A,$ $r_0, r_1, T, s, R,$ $F$, and ${\boldsymbol{ L}}$ as described above.
Suppose that the conclusion of the theorem is false for 
these data. Then there exist  $\ep_0>0$ and a finite subset ${\cal F}\subseteq A$
such that there are a sequence of positive numbers $(\dt_n)$ with $\dt_n\searrow 0,$ an increasing sequence
$({\cal G}_n)$ of finite subsets of $A$ such that $\bigcup_n {\cal G}_n$ is dense in $A,$
an increasing sequence $({\cal P}_n)$ of finite subsets of $\underline{\mathrm{K}}(A)$ such that
$\bigcup_{{n}}{\cal P}_n=\underline{\mathrm{K}}(A),$  an increasing sequence $({\cal U}_n)$ of finite subsets of
$\mathrm{U}(\mathrm{M}_{\infty}({\widetilde A}))$ such that $\bigcup_{{n}}{\cal U}_n\cap \mathrm{U}(\mathrm{M}_m({\widetilde A}))$ is dense
in $\mathrm{U}(\mathrm{M}_m({\widetilde A}))$ for each integer $m\ge1,$  an increasing sequence $({\cal H}_n)$ of finite subsets of 
$A_+^{\boldsymbol{1}}\setminus \{0\}$ such that, if $a\in {\cal H}_n$
and $f_{1/2}(a)\not=0,$ then 
$f_{1/2}(a)\in {\cal H}_{n+1}$, and $\bigcup_{{n}}{\cal H}_n$ is dense in
$A^{\boldsymbol  1}$, and 
(use {\ref{lem-of-bla}})
has dense intersection with the {{unit}} ball of each closed two-sided ideal of $A$, a sequence of integers $(k(n))$ with
$\lim_{n\to\infty} k(n)=+\infty$,  a sequence of unital \CA s $B_n\in {\boldsymbol C}_{r_0, r_1, T, s, R},$
two sequences of ${\cal G}_n$-$\dt_n$-multiplicative completely positive contractive maps $\phi_{n}, \psi_{n}: A\to B_n$
such that
\beq\label{stableun2-1}
[\phi_{n}]|_{{\cal P}_n}=[\psi_{n}]|_{{\cal P}_n}\quad\mathrm{and}\quad 
{\text{cel}}(\lceil \phi_{n}(u)\rceil \lceil \psi_{n}(u^*)\rceil)\le {\boldsymbol L}(u),\quad \textrm{for all $u\in {\cal U}_n$},
\eneq
a sequence of 
 ${\cal G}_n$-$\dt_n$-multiplicative completely positive contractive linear maps $\sigma_n: A\to
\mathrm{M}_{l(n)}(B_n)$ which are $F$-${\cal H}_n$-full and satisfy, for each $n=1, 2,...$,
\beq\label{stableun2-2}
&&\hspace{0.3in} \inf\{\sup\|v_n^*( \phi_{n}(a)\oplus {(\sigma_n)_{k(n)}(a)})v_n-(\psi_n(a)\oplus {(\sigma_n)_{k(n)}(a)})\|: a\in {\cal F}\}\ge \ep_0,
\eneq
where the infimum is taken among all unitaries $v_n\in {\widetilde{\mathrm{M}_{k(n)l(n)+1}(B_n)}}$ and ${(\sigma_n)_{k(n)}}: A \to \mathrm{M}_{k(n)l(n)}(B_n)$ is as above.

Set $\mathrm{M}_{l(n)}(B_n)=B_n',$ $\bigoplus_{n=1}^{\infty}B_n'=C_0,$ $\prod_{n=1}^{\infty}B_n'=C,$ 
and $C/C_0=Q(C),$ and denote by 
$\pi: C\to Q(C)$ the quotient map. Consider the maps $\Phi, \Psi, S: A\to C$ defined by
$\Phi(a)=(\phi_n(a))_{n\ge 1},$ $\Psi(a)=(\psi_n(a))_{n\ge 1}$, and $S(a)=(\sigma_n(a))_{n\ge 1}$, $a\in A.$  
Note that $\pi\circ \Phi,$
$\pi\circ \Psi$ and $\pi\circ S$ are \hm s.  Consider also the truncations
$\Phi^{(m)}, \Psi^{(m)}, S^{(m)}: A\to \prod_{{n\ge m}} B_n'$ defined by
$\Phi^{(m)}(a)=(\phi_n(a))_{n\ge m},$ $\Psi^{(m)}(a)=(\psi_n(a))_{n\ge m},$
and $S^{(m)}(a)=(\sigma_n(a))_{n\ge m}$, $a\in A$. 

For each $u\in {\cal U}_m,$ we have 
$u\in \mathrm{M}_{L(m)}({\widetilde A})$ {for some integer $L(m)\ge 1$.}  When $n\ge m,$  by 
hypothesis,
there exists a continuous path of unitaries  $\{u_n(t): t\in [0,1]\}\subseteq M_{L(m)}({\widetilde B_n'})$ such that
\vspace{-0.1in} \beq \nonumber
u_n(0)=\lceil \phi_n(u)\rceil, u_n(1)=\lceil \psi_n(u)\rceil\, \mathrm{and}\ 
{\text{cel}}(\{u_n(t)\})\le {\boldsymbol L}(u).\nonumber
\eneq
It follows from Lemma 1.1 of \cite{GL1} that, for all $n\ge m,$ there exists a continuous path
$\{U(t): t\in [0,1]\}\subseteq \mathrm{U}_0(\prod_{n{\ge}m} {\widetilde B_n'})$ such that
$U(0)=({\lceil \phi_n(u)\rceil})_{n\ge m}$ and $U(1)=({\lceil \psi_n(u)\rceil})_{n\ge m}.$
This in particular implies that
\beq\label{stableun2-4}
&&
{\lceil\Phi^{(m)}(u)\rceil \lceil \Psi^{(m)}(u^*) \rceil} \in \mathrm{U}_0(\mathrm{M}_{L(m)}(\prod_{n{\ge}m} {\widetilde B_n'}))\quad\mathrm{and}\quad [\pi\circ \Phi]_{*1}=[\pi\circ \Psi]_{*1}.
\eneq
By ({\ref{stableun2-1}}),  for all $n\ge m,$
\beq\label{stableun2-5}
[\phi_n]|_{{\cal P}_m}=[\psi_n]|_{{\cal P}_m}.
\eneq
By hypothesis and by \cite{GL1},  $\mathrm{K}_0(C)=\prod_b\mathrm{K}_0(B_n')$, 
it follows that
\beq\label{stableun2-6}
[\Phi^{(m)}]|_{\mathrm{K}_0(A)\cap {\cal P}_m}=[\Psi^{(m)}]|_{\mathrm{K}_0(A)\cap {\cal P}_m},\,\,m=1,2,....
\eneq
In particular,
\beq\label{stableun2-7}
[\pi\circ \Phi]_{*0}=[\pi\circ \Psi]_{*0}.
\eneq

Now let $x_0\in {\cal P}_m\cap \mathrm{K}_0(A,\Z/k\Z)$ for some $k\ge 2.$ Denote by ${\tilde x_0}\in \mathrm{K}_1(A)$ the image
of $x_0$ under the map $\mathrm{K}_0(A, \Z/k\Z)\to \mathrm{K}_1(A).$
We may assume that ${\tilde x_0}\in {\cal P}_{m_0}$ for some
$m_0\ge m.$  By  \eqref{stableun2-4},
$[\Phi^{(m_0)}]({\tilde x_0})=[\Psi^{(m_0)}]({\tilde x_0}).$
Set $y_0=[\Phi^{(m_0)}](x_0)-[\Psi^{(m_0)}](x_0).$ Then
$y_0\in \mathrm{K}_0((\prod_{n{\ge}m_0}B_n'), \Z/k\Z)$ must be in the image of $\mathrm{K}_0(\prod_{n{\ge}m_0}B_n')$,
which may be identified with
$\mathrm{K}_0(\prod_{n{\ge}m_0}B_n')/k\mathrm{K}_0(\prod_{n{\ge}m_0}B_n')$
(see \cite{GL1}).
However, by  (\ref{stableun2-5}),
$$y_0\in {\text{ ker}}\,\psi_0^{(k)},$$
where $\psi_0^{(k)}: \mathrm{K}_0(\prod_{n{\ge}m_0}B_n', \Z/k\Z)\to \prod_{n{\ge}m_0}\mathrm{K}_0(B_n', \Z/k\Z)$
is  as in 4.1.4 of \cite{Lncbms}.
By \cite{GL1},
$y_0=0.$ In other words,
$${[\Phi^{(m_0)}](x_0)=[\Psi^{(m_0)}](x_0)},$$
which implies that
\vspace{-0.1in} \beq\label{stableun2-10}
[\pi\circ \Phi]|_{\mathrm{K}_0(A,\Z/k\Z)}=[\pi\circ \Psi]|_{\mathrm{K}_0(A, \Z/k\Z)},\,\,k=2,3,....
\eneq

Now let $x_1\in \mathrm{K}_1(A,\Z/k\Z).$ Then $x_1\in {\cal P}_m$ for some $m\ge 1.$
Denote by ${\tilde x_1}\in \mathrm{K}_0(A)$ the image of $x_1$ under the map $\mathrm{K}_1(A,\Z/k\Z)\to \mathrm{K}_0(A).$
There is $m_1\ge m$ such that ${\tilde x}_1\in {\cal P}_{m_1}.$
By (\ref{stableun2-6}), $[\Phi^{(m_1)}]({\tilde x}_1)=[\Psi^{(m_1)}]({\tilde x}_1).$
Put $y_1=[\Phi^{(m_1)}](x_1)-[\Psi^{(m_1)}](x_1).$ Then
$y_1\in \mathrm{K}_1(\prod_{n=m_1}B_n')/k\mathrm{K}_1(\prod_{n=m_1}B_n')$ (see (\cite{GL1})).
However, by (\ref{stableun2-4}), $y_1\in {\text{ ker}}\psi_1^{(k)}$ (see 4.1.4 of \cite{Lncbms}).
It follows from
\cite{GL1} that $y_1=0.$ In other words,
$${[\Phi^{(m_1)}](x_1)=[\Psi^{(m_1)}](x_1)}.$$
Thus,
\vspace{-0.14in} \beq\label{stableun2-12}
[\pi\circ \Phi]|_{\mathrm{K}_1(A,\Z/k\Z)}=[\pi\circ \Psi]|_{\mathrm{K}_1(A, \Z/k\Z)}.
\eneq
Combining (\ref{stableun2-4}), (\ref{stableun2-7}), (\ref{stableun2-10}), and (\ref{stableun2-12}), we have
$$[\pi\circ \Phi]=[\pi\circ \Psi]\,\,\,{\text{in}} \,\,\,  \mathrm{Hom}_{\Lambda}(\underline{\mathrm{K}}(A), \underline{\mathrm{K}}(Q(C)),
$$
{where $C=\prod_{n\geq m} B_n'$.}
For each $a\in {\cal H}_m\subseteq A_+^{\boldsymbol 1}\setminus \{0\},$
any $(b_n)\in C_+^{\boldsymbol 1},$ and any $\eta>0,$  since $\sigma_n$ is $F$-${\cal H}_n$-full,
for all $n\ge m,$
there are $x_{i,n}(a)\in B_n'$ with $\|x_{i,n}\|\le M(a),$ $i=1,2,...,N(a),$
where $F(a)=M(a)\times N(a),$ such that
$$\|\sum_{i=1}^{N(a)} x_{i,n}(a)^*\sigma_n(a) x_{i,n}(a)- b_n\|<\eta.$$
Define $x(i,a)=(x_{i,n}(a)).$ Then $x(i,a)\in C.$
It follows
that
$$\|\sum_{i=1}^{N(a)}x(i,a)^*S^{(m)}(a)x(i,a)-(b_n)_{{n\ge m}}\|<\eta.
$$
This shows that $\pi\circ S(a)$ is a full element of $Q(C)$ for any $0\neq a \in \bigcup_{n=1}^{\infty} {\cal H}_n.$
Let $I$ be an ideal of $Q(C)$ and
consider the pre-image
$$
J=\{a\in A: \pi\circ S(a)\in I\}.
$$
By the choice of $({\cal H}_n),$ 
$J=\{0\}.$ It follows
that the map $\pi\circ S: A\to Q(C)$ is full.

By Lemma \ref{Lsep2}, there exists a separable \SCA\,
$D\subseteq Q(C)$ such that
$\pi\circ S(A), \pi\circ \Phi(A), \pi\circ \Psi(A)\subseteq D$, the map $\pi\circ S:A\to D$ is full, and
$$
[\pi\circ \Phi]=[\pi\circ \Psi]\,\,\, {\text{in}}\,\,\, {\mathrm{Hom}_{{\Lambda}}(\underline{\mathrm{K}}(A), \underline{\mathrm{K}}(D))}.
$$

Since $A$ satisfies the UCT, by \cite{DL},  {{$[\pi\circ \Phi]=[\pi\circ \Psi]$ in $KL(A,D).$}}
{{Then,  by Theorem \ref{T39}}}
(as at the end of the proof  of Theorem \ref{T39}), there exists an integer $K\ge 1$ and a unitary $V\in {\widetilde{\mathrm{M}_{K+1}(Q(C))}}$  such that
$$\|V^*(\pi\circ \Phi(a)\oplus\Sigma(a))V-(\pi\circ \Psi(a)\oplus\Sigma(a))\|<\ep_0/4,\quad a\in\mathcal F$$
where,
as above,
\vspace{-0.12in} $$\Sigma(a)=\overbrace{\pi\circ S(a)\oplus \pi\circ S(a)\oplus\cdots\oplus \pi\circ S(a)}^K,\quad a\in A.$$
Therefore, there exists a sequence of  unitaries $(v_n)\subseteq {\widetilde{\mathrm{M}_{K+1}(C)}}$ and an integer
$N_1$ such that $k(n)\ge K$ for all $n\ge N_1$ and
$$\|v_n^*(\phi_n(a)\oplus(\sigma_n)_K(a))v_n-(\psi_n(a)\oplus(\sigma_n)_K(a))\|<\ep_0/2,\quad a\in\mathcal F,$$
where
\vspace{-0.13in} $$
(\sigma_n)_K(a)=\overbrace{\sigma_n(a)\oplus\sigma_n(a)\oplus\cdots\oplus\sigma_n(a)}^K,\quad a\in A.
$$
This contradicts  
(\ref{stableun2-2}). 
\end{proof}

\begin{rem}\label{RRLuniq}
Suppose that $\mathrm{K}_1(A)\cap {\cal P}
=\{z_1,z_2,...,z_m\}.$ Then, by choosing sufficiently large ${\cal P},$
we can always choose ${\cal U}=\{w_1, w_2,...,w_m\}$ so that
$[w_i]=z_i,$ $i=1,2,...,m.$  In other words, we do not need to consider
unitaries in $\mathrm{U}_0(\mathrm{M}_{\infty}({\widetilde A})).$ In particular, if $\mathrm{K}_1(A)=\{0\},$
then we can omit the condition \eqref{Lauct-1}.
Moreover, if $B$ is restricted in the class of \CA s of real rank zero, then
one can choose ${\boldsymbol L} \equiv 2\pi +1$ and \eqref{Lauct-1} always holds if
${\cal P}$ is sufficiently large. In other words, in this case, condition \eqref{Lauct-1} can
also be dropped.

{{Let $B_0$ be a \CA\, with a strictly positive element $e_0$ and $B=\overline{e_b(M_{K+1}(B_0))e_b},$ where 
{{$e_b\in M_{K+1}(B_0)_+$ and}} $e_b\ge (\sum_{i=1}^K(e_0\otimes e_{ii})$
and $\{e_{i,j}: 0\le i,j\le K\}$ is a matrix unit for $M_{K+1}.$ 
Let $B_1=\overline{e_1Be_1},$ where 
$e_1\in \overline{(e_0\otimes e_{0,0})M_{K+1}(B_0)(e_0\otimes e_{0,0})}_+.$ 
We may  view $B_1\subset B_0.$
Suppose 
$B_0\in {\boldsymbol{C}}_{r_0, r_1, T, s, R}.$ 
Suppose that $\phi, \psi: A\to B_1\subset B$ and $\sigma: A\to B_0\subset B$ are as in Theorem \ref{Lauct2} 
($\phi$ and $\psi$ are ${\cal G}$-$\dt$-multiplicative, and $\sigma$ is $T$-${\cal H}$-full in $B_0$), and  
that ${\mathrm{cel}}(\lceil \phi(u)\rceil  \lceil \psi(u^*)\rceil)\le {\boldsymbol{L}}(u)\tforal u\in {\cal U}$
(viewing $\phi$ and $\psi$ as maps to  $B_0$ instead of $B_1$) and \eqref{Lauct-1n} holds. Then there exists $u\in {{{\widetilde{B}}}}$ such that
\beq
\|u^*\diag(\phi(a), \sigma_K(a))u-\diag(\psi(a), \sigma_K(a))\|<\ep\rforal a\in {\cal F},
\eneq
where $\sigma_K(a)=\diag(\sigma(a),..., \sigma(a)),$ where $\sigma(a)$ repeats $K$ times (see also below).}}

\end{rem}

\begin{cor}\label{CLuniq}
Let $A$ be a non-unital separable amenable  \CA\, which is $\mathrm{KK}$-contractible
and let $T: A_+\setminus \{0\}\to \N\times \R_+\setminus \{0\}$
be a  map.
  For any $\ep>0$ and any finite subset ${\cal F}\subseteq A,$ there exists
$\dt>0,$ a finite subset ${\cal G}\subseteq A,$
a finite subset ${\cal H}\subset A_+\setminus \{0\}$ and an integer $K\ge 1$ satisfying the following:

{{Let  $B_0$ be any \CA\, with a strictly positive element $e_0$ and $B=\overline{e_bM_{K+1}(B_0)e_b},$ where $
e_b\in M_{K+1}(B_0),$ $e_b\ge \sum_{i=1}^K(e_0\otimes e_{ii})$
and $\{e_{i,j}: 0\le i,j\le K\}$ is a matrix unit for $M_{K+1}.$ 
Let $B_1=\overline{e_1Be_1},$ where 
$e_1\in \overline{(e_0\otimes e_{0,0})M_{K+1}(B_0)(e_0\otimes e_{0,0})}_+.$}} 
For any two ${\cal G}$-$\dt$-multiplicative \morp s $\phi, \psi: A\to B_1,$  
and any
${\cal G}$-$\dt$-multiplicative \morp\, $\sigma: A\to  B_0$
 which is also
$T$-${\cal H}$-full in $B_0,$ 
there exists a unitary $U\in {{{\widetilde B}}}$
such that
\beq\label{CLauct-2}
\|{\mathrm{Ad}}\, U\circ (\phi\oplus \sigma_K)(a)-(\psi\oplus \sigma_K)(a)\|<\ep\tforal a\in {\cal F},
\eneq
where, as earlier,
\vspace{-0.1in} $$
\sigma_K=\overbrace{\sigma \oplus\sigma\oplus\cdots\oplus \sigma}^K: A\to\mathrm{M}_{K}(B_0)\subset B.
$$
\end{cor}

\begin{proof}
{ In Theorem \ref{Lauct2},} the only reason that the restriction has to be placed on $B$ is for the computation
of the K-theory of the maps $\phi$ and $\psi.$  {{More precisely, 
the restriction  is used to obtain
$$[\pi\circ \Phi]=[\pi\circ \Psi]\,\,\,{\text{in}} \,\,\,  \mathrm{Hom}_{\Lambda}(\underline{\mathrm{K}}(A), \underline{\mathrm{K}}(Q(C))$$
in the proof of \ref{Lauct2}.}} {{Since $A$ is $KK$-contractible, $\underline{K}(A)=\{0\}.$
Hence $[\pi\circ \Phi]=[\pi\circ\Psi]=0.$}}
Note that, since $\mathrm{KK}(A, A)=0$, $A$ satisfies the UCT.
\end{proof}

\begin{lem}\label{KK-Kun}
If a separable 
\CA\,
$B$ is KK-contractible, then $A\otimes B$ is KK-contractible for any  separable 
amenable
\CA\, $A$.
\end{lem}
\begin{proof}
Since $B$ is KK-contractible,
i.e.,
$\mathrm{id}_{B} \sim_{\mathrm{KK}} 0_{B}$,
there is a continuous path (in the strict topology) of pairs $(\phi^+_t, \phi^-_t)$, $t\in[0, 1]$, where $$\phi^{\pm}_t: B \to \mathrm M(B\otimes\mathcal K),\quad t\in[0, 1],$$
are homomorphisms such that 
$$\phi_t^+(a)-\phi_t^-(a) \in B\otimes\mathcal K,\quad t\in[0, 1],\ a\in B,$$
$$(\phi_0^+, \phi_0^-) = (\mathrm{id}_{B}, 0)\quad\mathrm{and}\quad (\phi_1^+, \phi_1^-) = (0, 0).$$

Let $A$ be a separable 
amenable
C*-algebra. Consider the two families of elements
$$\Phi_t^{\pm}(a\otimes b)=a\otimes\phi_t^{\pm}(b)\in A\otimes \mathrm M( B\otimes\mathcal K)\subseteq \mathrm M(A\otimes B\otimes\mathcal K),\quad a\in A,\ b\in B,\ t\in[0, 1].$$ (Nuclearity of $A$ implies that the two tensor products are unambiguous.)
Then $\Phi_t^{\pm}(a\otimes b)$, $t\in[0, 1]$, are continuous paths (in the strict topology) in $\mathrm M(A\otimes B \otimes\mathcal K)$, and
$$\Phi_t^+(a\otimes b) - \Phi_t^-(a\otimes b) = a\otimes (\phi_t^+(b) - \phi_t^-(b)) \in A\otimes B \otimes\mathcal K.$$
Moreover, 
$(\Phi^+_0, \Phi^-_0) = (\mathrm{id}_{A \otimes B}, 0)$ and $(\Phi^+_1, \Phi^-_1) = (0, 0)$. Therefore, $\mathrm{id}_{A\otimes B} \sim_{\mathrm{KK}} 0$, 
i.e.,
$A\otimes B$ is KK-contractible, as asserted.
\end{proof}

\section{ An
isomorphism theorem}

Recall that a non-unital \CA\, $A$ is said to have almost stable rank one
if the closure of  
the set of
invertible elements 
in 
${\widetilde A}$ contains  $A$, and if this holds also for each hereditary sub-C*-algebra of $A$ in place of $A$ 
(see \cite{Rob-0}).

Recall also that if $A\in {\cal D}$ {{is a separable simple \CA,}}  then $A$ has (Blackadar) strict comparison for positive elements, 
$A$
has stable rank one,
and the map from $\mathrm{Cu}(A)$ to  $\text{LAff}_{0+}(\overline{\mathrm{T}(\overline{aAa})}^{\mathrm{w}})$ is an isomorphism of ordered semigroups (for any non-zero element $a\in {\mathrm{Ped}}(A)$)
 (see 11.8 and 11.3  of \cite{eglnp})

In what follows, if $A$ is a \CA, we use $A^{\bbf 1}$ for the unit ball  of $A.$ 
We will use the following reformulation of Definition \ref{DDD} given by 11.10 of \cite{eglnp} when $\mathrm{K}_0(A)=\{0\}.$
\begin{prop}[11.10 and {10.8}  of \cite{eglnp}]\label{Cuniformful}
Let $A$ be a {{separable}} \CA\, in ${\cal D}$  with $\mathrm{K}_0(A)=\{0\}$. Let the strictly positive element $e\in A$
with $\|e\|\leq 1$ and the number $1>\mathfrak{f}_e>0$ be as in \ref{DDD}.
There is a map $T: A_+\setminus \{0\}\to \N\times \R_+\setminus \{0\}$
with the following property:
For any  finite subset ${\cal F}_0\subseteq A_+\setminus \{0\}$,
any $\ep>0,$  any
finite subset ${\cal F}\subseteq A$, any $b\in A_+\setminus \{0\}$, and any integer
$n\ge 1,$  there are ${\cal F}$-$\ep$-multiplicative \cpc s $\phi: A\to A$ and  $\psi: A\to D$  for some
\SCA\, $D = D\otimes e_{11}\subseteq \mathrm{M}_n(D)\subseteq A$ such that 
{$\psi(e)$} is strictly positive in $D$
and $T$-$\mathcal F_0\cup\{f_{1/4}(e)\}$-full as a map $A \to D$,
\beq\label{CDtad-1}
&&\|x-(\phi(x)\oplus \overbrace{\psi(x)\oplus \psi(x)\oplus\cdots\oplus \psi(x)}^n)\|<\ep\rforal x\in {\cal F}\cup \{e\},\\\label{CDtrdiv-2}
&& D\in {\cal C}_0,\  
\phi(e)\lesssim b,\,\,\, \phi(A)\perp \mathrm{M}_n(D),
\eneq
\beq\label{106+}
\phi(e)\lesssim \psi(e)\tand t\circ f_{1/4}(\psi(e))>\mathfrak{f}_e\tforal t\in T(D).
\eneq

\end{prop}

\begin{defn}\label{DAq}
Let $A$ be a \CA\, with $\mathrm{T}(A)\not={\O}$ such that $0\not\in \overline{\mathrm{T}(A)}^{\mathrm{w}}.$
There is an affine  map
$r_{\aff}: A_{\mathrm{s.a.}}\to {\mathrm{Aff}}(\overline{\mathrm{T}(A)}^{\mathrm{w}})$ defined by
$$
r_{\aff}(a)(\tau)=\hat{a}(\tau)=\tau(a),\quad \tau\in \overline{\mathrm{T}(A)}^{\mathrm{w}},\ a\in A_{\mathrm{s.a.}}.
$$
Denote by 
$A^{\mathrm{q}}$ 
the space  $r_{\aff}(A_{\mathrm{s.a.}}),$ 
$A_+^{\mathrm{q}}=r_{\aff}(A_+)$ and $A^{\bbf{1},\mathrm{q}}_+=r_{\aff}(A_+^{\bbf 1})$.

\end{defn}

\begin{thm}\label{TTMW}
Let $A$ and $B$ be two separable simple amenable \CA s in the class $\mathcal D$
with continuous scale. 
Suppose that both $A$ and $B$ are $\mathrm{KK}$-contractible. 
Then  $A\cong B$ if and only if
there is an affine homeomorphism $\gamma: \mathrm{T}(B)\to \mathrm{T}(A).$
Moreover,  the isomorphism $\phi: A\to B$ can be chosen such that
$\phi_\mathrm{T}=\gamma,$ where $\phi_\mathrm{T}$ is the map from $\mathrm{T}(B)$ to $\mathrm{T}(A)$ induced by $\phi.$
\end{thm}

\begin{proof}
By Theorem \ref{R2},
there exists a simple \CA\, $C=\lim_{n\to\infty} (C_n, \imath_n)$, where each
$C_n$ is a finite direct sum of copies of ${\cal W}$ and $\imath_n$ 
maps strictly positive elements to
strictly positive 
elements,
which has continuous scale, and is such that
$$
\mathrm{T}(A)\cong\mathrm{T}(C).
$$
It suffices to show that $A\cong C.$ (By symmetry, then also $B\cong C$.)
We will use $\Gamma: \mathrm{T}(C)\to \mathrm{T}(A)$ for  the  affine homeomorphism given above.
We will use the approximate intertwining argument of Elliott (\cite{Ell-AT-RR0}).
We would like recall that ${\cal W}$ is an inductive limit of Razak algebras with injective connecting maps
and the fact that $A$ has stable rank one (see 11.5 of \cite{eglnp}). 
{{Fix two sequences, $\{x_1,x_2,...,x_n,...\}$ of  $A$  and $\{y_1,y_2,...,y_n,...\}$ of $C$, which are dense in 
the unit ball of $A$ and $B,$ respectively.}}

{\bf Step 1}: Construction of $L_1.$

\noindent
Fix a finite subset ${\cal F}_1\subseteq A$ and $\ep>0.$
\Wlog, we may assume that ${{x_1}}\in {\cal F}_1\subseteq A^{\boldsymbol{1}}.$ 

 Since $A$ has continuous scale,  $A=\mathrm{Ped}(A)$ (3.3 of \cite{Lncs1}).
Choose  a strictly positive element $a_0\in A_+$ with $\|a_0\|=1$ and ${{\mathfrak{f}_{a_0}}}>0$ as in Definition \ref{DDD}.
We may assume, \wilog,
that
\beq\label{TC0k-nn1}
a_0y=ya_0=y,\,\, a_0\ge y^*y\andeqn  a_0\ge yy^* \rforal y\in {\cal F}_1.
\eneq
Let $T: A_+\setminus \{0\}\to \N\times \R_+\setminus \{0\}$
 with $T(a)=(N(a), M(a))$ ($a\in A_+\setminus \{0\}$) be as given
 by Proposition \ref{Cuniformful} (11.10 and {10.8} of \cite{eglnp}).

 Let $\dt_1>0$ (in place of  $\dt$), let
 ${\cal G}_1\subseteq A$ (in place of ${\cal G}$) be a finite subset,
 let ${\cal H}_{1,0}\subseteq A_+\setminus \{0\}$ (in place of ${\cal H}$)
 be a finite subset, and let $K_1\ge 1$ (in place of $K$) be an integer
as 
given by \ref{CLuniq} 
  for the above $T$, 
  $\ep/16$ (in place of $\ep$), and ${\cal F}_1.$  {{We may assume that $\dt_1<\ep.$}}

 \Wlog, we may assume that ${\cal F}_1\cup {\cal H}_{1,0}\subseteq {\cal G}_1\subseteq A^{\boldsymbol{ 1}}.$

 Choose $b_0\in A_+\setminus \{0\}$ with
 $d_\tau(b_0)<1/8(K_1+1).$

 It follows from Proposition \ref{Cuniformful} 
 that there are ${\cal G}_1$-$\dt_1/64$-multiplicative \cpc s
 $\phi_0: A\to A$ and $\psi_0: A\to D$ for some $D = D\otimes e_{1,1}\subseteq D\otimes {\mathrm {M}_{2K_1+1}}\subseteq A$ with
 $D\in {\cal C}_0'$  such that
$(D\otimes {\mathrm {M}_{2K_1+1}})\phi_0(A) = 0$ and 
 \beq\label{TCzeroK1-1}
&&\hspace{-0.4in} \|x-(\phi_0\oplus\overbrace{\psi_0\oplus\psi_0\oplus\cdots\oplus\psi_0}^{2K_1+1})(x)\|<\min\{\ep/128, \dt_1/128\}\rforal x\in {\cal G}_1,\\\label{TCzeroK1-1+20015}
&&  \phi_0(a_0)\lesssim b_0,  \,\,\, \phi_0(a_0)\lesssim \psi_0(a_0),
\eneq
$\psi_0(a_0)$ is strictly positive
in $D$, and,
moreover,
$\psi_0$ is $T$-${\cal H}_{1,0}\cup \{f_{1/4}(a_0)\}$-full
as a map from $A$ to $D$.
%

{{By \eqref{TCzeroK1-1+20015}, replacing $\phi_0$ by $f_\eta(\phi_0(a_0))\phi_0 f_\eta(\phi_0(a_0))$ for some sufficiently small 
$\eta,$ applying a result of R\o rdam (see also Lemma 3.2 of \cite{eglnp}), as $A$ has stable rank one (see 11.5 of \cite{eglnp}), one may assume 
that there is a unitary $w_0\in {\widetilde{A}}$ 
such that 
\beq\label{43-200106-1}
w_0^*\phi_0(a)w_0\in \overline{DAD}.
\eneq}}
Define $\phi_0': A\to A$ by $\phi_0'(a)=\diag(\phi_0(a), \psi_0(a))$ for all $a\in A.$
\noindent
 {{Let}} $D_{1,1}=\mathrm{M}_{2K_1}(D)$ and $D_{1,1}'=\mathrm{M}_{2K_1+1}(D).$
 Let $j_1: D\to \mathrm{M}_{2K_1}(D)$ be defined by
 $$
 j_1(d)={{\overbrace{d\oplus d\oplus\cdots\oplus d}^{2K_1}}} \rforal d\in D.
 $$
  Let
\vspace{-0.12in} \beq\nonumber
 &&{{d_{00}'=\overbrace{\psi_0(a_0)\oplus\psi_0(a_0)\oplus\cdots\oplus\psi_0(a_0)}^{1+2K_1}\in D_{1,1}'.}}
 \eneq	
 Let $\imath_1: D_{1,1}'\to A$ denote the embedding {{map}}, and  {{let}}
 $\mathrm{Cu}^{\sim}(\imath_1): \mathrm{Cu}^{\sim}(D_{1,1}')\to \mathrm{Cu}^{\sim}(A)$ {{denote}} the induced map. 
By 6.2.3 of \cite{Robert-Cu}, $\mathrm{Cu}^\sim(A)
=\mathrm{L}\aff_+^\sim(\mathrm{T}(A))$
(see also 7.3 
 and 11.8  
 of \cite{eglnp}).
This also holds {{with $C$}} in place of A.
  Let $\Gamma^{\sim}: \mathrm{Cu}^{\sim}(A)\to \mathrm{Cu}^{\sim}(C)$ be
 the isomorphism given by
 $\Gamma^{\sim}(f)(\tau)=f(\Gamma(\tau))$ for all $f\in \mathrm{L}\aff_+^\sim(\mathrm{T}(A))$ and $\tau\in \mathrm{T}(A)$ (see 7.3  of \cite{eglnp}).
 By Theorem 1.0.1 of \cite{Robert-Cu}, 
there  is a \hm\, $h_1':
 D_{1,1}'\to C$ such
that
 \beq\label{TC0k1-5}
 \mathrm{Cu}^{\sim}(h_1')=\Gamma^{\sim}\circ \mathrm{Cu}^{\sim}(\imath_1),\,\,\, {\mathrm{in\,\,\, particular}},
\la  h_1'(d_{00}')\ra =\Gamma^{\sim}\circ \mathrm{Cu}^{\sim}(\imath_1)(\la d_{00}'\ra).
 \eneq
Write $h_1=(h_1')|_{D_{1,1}},$ and $C'=\{c\in C: ch_1(d)=h_1(d)c=0\rforal d\in D_{1,1}\}.$
Note that
 $$
h_1'(\psi_0(a)\oplus \overbrace{0\oplus 0\oplus\cdots\oplus 0}^{2K_1}){\in} C'\rforal a\in A.
$$
Define  $h_0': A\to C'$ by
\vspace{-0.12in}$$
h_0'(a) =h_1'(\psi_0(a)\oplus \overbrace{0\oplus 0\oplus\cdots\oplus 0}^{2K_1})\rforal a\in A.
$$
Define $L_1: A\to C$ by
\vspace{-0.16in}\beq\label{43-20107-2}
L_1(a)=h_0'(a)\oplus h_1(\overbrace{\psi_0(a)\oplus \psi_0(a)\oplus\cdots\oplus\psi_0(a)}^{2K_1})\rforal a\in A.
\eneq
{{Note that $L_1$ is ${\cal G}_1$-$\dt_1/64$-multiplicative (see \eqref{TCzeroK1-1}).}}

{\bf Step 2}: Construct $H_1$ and the first approximate commutative diagram.

It follows from Theorem 1.0.1 of \cite{Robert-Cu}, as $A$ has stable rank one (by 11.5 of \cite{eglnp}),  
that there is a
\hm\, $H: C\to  A$ such that 
\vspace{-0.1in}\beq\label{TC0k1-6}
\mathrm{Cu}^{\sim}(H)=(\Gamma^{\sim})^{-1}.
\eneq
Note that (by \eqref{TC0k1-6}, {{\eqref{TCzeroK1-1+20015}}} and the definition of $h_0'$) 
\beq\label{TC0k1-9}
\la  H\circ h_0'(a_0)\ra \le \la \psi_0(a_0)\ra
\eneq
in the C*-algebra $A$.
Choose $\dt_1/4>\eta_0>0$ such that
\beq\label{TC0k1-10}
\hspace{-0.2in}\|f_{\eta_0}(H\circ h_0'(a_0)) x-x\|,\,\,\|x-xf_{\eta_0}( H\circ h_0'(a_0))\|<
\min\{\ep/128, \dt_1/128\}
\eneq
for all $x\in  H\circ h_0'({\cal G}_1).$
Again, since $A$  has stable rank one (11.5 of \cite{eglnp}), by a result of R\o rdam
(see also 3.2 of \cite{eglnp}), there is a unitary $u_0\in {\widetilde A}$ such that
\beq\label{TC0k1-11}
u_0^*f_{\eta_0}( H\circ h_0'(a_0))u_0\in \overline{\psi_{00}(a_0)A\psi_{00}(a_0)}{{=\overline{DAD}}},
\eneq
where
\vspace{-0.13in}$$
\psi_{00}(a)=\psi_0(a)\oplus \overbrace{0\oplus 0\oplus\cdots\oplus 0}^{2K_1} \in \mathrm{M}_{1+2K_1}(D)\subseteq A\rforal a\in A.
$$
{{Set $A_{0,1}'=\overline{u_0^*f_{\eta_0}( H\circ h_0'(a_0))u_0Au_0^*f_{\eta_0}( H\circ h_0'(a_0))u_0}.$}}
Define $H': A\to  {{A_{0,1}'\subset \overline{DAD}}}$
by
$$
H'(a)=u_0^*(f_{\eta_0}(H\circ h_0'(a_0)))H\circ h_0'(a)(f_{\eta_0}(H\circ h_0'(a_0)))u_0\rforal a\in A.
$$
Note that $H'$ is a ${\cal G}_1$-$\dt_1/32$-multiplicative \cpc.
Moreover, by \eqref{TC0k1-10},
\vspace{-0.02in}\beq\label{TC0k1-11+}
\|{\text{ Ad}}\, u_0\circ H\circ h_0'(a)-{H'(a)}\|<\min\{\ep/128, \dt_1/128\}\rforal a\in {\cal G}_1.
\eneq

Consider the \hm s
${\text{Ad}}\, u_0\circ H\circ h_1\circ j_1$ and $\imath_1\circ j_1$ (or rather $\imath_1|_{D_{1,1}}\circ j_1$).
Then, by \eqref{TC0k1-5} and \eqref{TC0k1-6},
\beq\label{TC0k1-7}
\mathrm{Cu}^{\sim}({\text{Ad}}\, u_0\circ H\circ h_1\circ j_1)={\mathrm{Cu}}^{\sim}(\imath_1\circ j_1).
\eneq
{{Put $A'=\{a\in A: a\perp   A_{0,1}'\}.$}}
{{Then $A'$ is a hereditary \SCA\, of $A.$}} 
{{Thus}} $A'\in {\cal D}$ and $K_0(A')=0.$  
Note that 
we may view 
both $\imath_1\circ j_1$ and ${\text{Ad}}\, u_0\circ H\circ h_1\circ j_1$
as
maps into $A'$ {{(recall $h_0'(A)\perp h_1(D_{1,1})$).}} 
{{By}} Theorem 3.3.1 of \cite{Robert-Cu} (as any hereditary sub-C*-algebra of $A$ has stable rank one)
{{and by \eqref{TC0k1-7},}} 
there exists a unitary $u_1\in {\widetilde A'}$ such that
\beq\label{TC0k1-8}
\|u_1^* ({\text{Ad}}\, u_0\circ H\circ h_1\circ j_1(x))u_1-\imath_1\circ j_1(x)\|<\min\{\ep/16,\dt_1/16\}\rforal x\in {{\psi_0({\cal G}_1)}}.
\eneq
Writing $u_1=\lambda+z$ with $z\in A'$, we may view $u_1$ is a unitary in ${\widetilde A}.$ Note that, for any
$b\in {{A_{0,1}',}}$
$u_1^*bu_1=b.$
In particular, for any $a\in A,$
\beq\label{TC0k1-15}
{\text{Ad}}\, u_1\circ H'(a)=H'(a)\rforal a\in A.
\eneq
{{Note that the map $\imath'\circ \psi_0: A\to \overline{DAD}$ is $T$-${\cal H}_{1,0}\cup \{f_{1/4}(a_0)\}$-full (see the last remark 
of \ref{fullunif}), where $\imath': D\to \overline{DAD}$ is the embedding.}}
By
Corollary 
\ref{CLuniq}, 
there is $u_2\in {\widetilde A}$  {{(see \eqref{43-200106-1})}}
such that
\beq\label{TC0k1-16}
\|{\text{Ad}}\, u_2\circ (H'(a)\oplus \imath_1\circ j_1\circ \psi_0(a))-({\phi_0'(a)}\oplus \imath_1\circ j_1\circ \psi_0(a))\|<\ep/16
\eneq
 for all $a\in {\cal F}_1.$
 {{Recall that $H\circ L_1(a)=H\circ h_0'(a) \oplus H\circ h_1\circ j_1\circ \psi_0(a)$
 for $a\in A$ (see \eqref{43-20107-2}).}}
 Combining  {{with \eqref{TC0k1-11+}, 
 \eqref{TC0k1-8}, 
  and \eqref{TC0k1-15},}}
 we have
 \beq\label{TC0k1-17}
 \|{\text{Ad}}\, (u_0 u_1u_2)\circ H\circ L_1(a)-{\text{Ad}}\, u_2\circ (H'(a)\oplus \imath_1\circ j_1\circ \psi_0(a))\|<{{\ep/128+\ep/16}}
 \eneq
 for all $ a\in {\cal F}_1.$
On the other hand,  by \eqref{TCzeroK1-1},
 \beq\label{TC0k1-18}
 \|{\text{id}}_A(a)-({\phi_0'(a)}\oplus {{\imath_1\circ}} j_1\circ \psi_0(a))\|<\ep/16\rforal a\in {\cal F}_1.
 \eneq
 Put $U_1=u_0u_1u_2.$  By {{\eqref{TC0k1-18},   \eqref{TC0k1-16},  and \eqref{TC0k1-17},}}
 we conclude that
 \beq\label{TC0k1-19}
 \|{\text{id}}_A(a)-{\text{Ad}}\, U_1\circ H\circ L_1(a)\|<\ep\rforal a\in {\cal F}_1.
 \eneq

 Put $H_1={\text{Ad}}\, U_1\circ H$ {{(note that $H_1$ is a \hm).}}
 Then we have the diagram
 \begin{displaymath}
\xymatrix{
A \ar[r]^{\id} \ar[d]_{L_1} & A\\
C \ar[ur]_{H_1}
}
\end{displaymath}
which is approximately commutative on the subset ${\cal F}_1$ to within $\ep.$

 {\bf Step 3}: Construct $L_2$ and the second approximately commutative diagram.

 We first return to $C.$
 Define $\Delta: {{C^{{\boldsymbol{1}}, q}_+}}\setminus \{0\}\to (0,1)$ by
 \beq\label{TC0k1-20}
 \Delta(\hat{a})=(1/2) \inf\{\tau(a): \tau\in T(C)\}
 \eneq
 (Recall that  $\mathrm{T}(C)$ is compact, by 5.3 of \cite{eglnp}  since $C$ has continuous scale.)

 Fix any $\eta_1>0$ and a finite subset ${\cal S}_1\subseteq C.$
 We may assume that $y_1\in {\cal S}_1\subseteq C^{\boldsymbol1}$ and $L_1({\cal F}_1)\subseteq {\cal S}_1.$

 Let ${{\cal G}_{2,C}}\subseteq C$ (in place of ${\cal G}$), ${\cal H}_{1,1}\subseteq C_+^{\boldsymbol 1}\setminus \{0\}$
 (in place of ${\cal H}_1$), and ${\cal H}_{1,2}\subseteq C_{s.a.}$ (in place
 of ${\cal H}_2$) be finite subsets,
and
 $\dt_2>0$ (in place of $\dt$) and  $\gamma_1>0$ (in place of $\gamma$)
 be real numbers as 
provided
by 7.8 of \cite{eglnp}
 for $C,$ $\eta_1/16$ (in place of $\ep$), and ${\cal S}_1$ (in place of ${\cal F}$),
 as well as $\Delta$ above.

 \Wlog, we may assume that ${\cal S}_1\cup {\cal H}_{1,2}\subseteq {{\cal G}_{2,C}}\subseteq C^{\boldsymbol 1}.$ 

 Fix $\ep_2>0$ {{(with $\ep_2<\ep/2$)}} and a finite subset ${\cal F}_2$ such that
 $\{x_1, x_2\}\cup H_1({\cal S}_1)\cup {\cal F}_1\subseteq {\cal F}_2.$
We may assume that ${\cal F}_2\subseteq A^{\boldsymbol 1}.$
 Let
 $$
 \gamma_0=\min\{\gamma_1, \inf\{\Delta(\hat{a}): a\in {\cal H}_{1,1}\cup {\cal H}_{1,2}\}\}.
 $$

\noindent
Fix a strictly positive element $a_1$ of $A$ with $\|a_1\|=1.$
We may assume, \wilog,
that
\beq\label{TC0k-n1}
a_1y=ya_1=y,\,\, a_1\ge y^*y\andeqn  a_1\ge yy^* \rforal y\in {\cal F}_2.
\eneq
Let the map $T: A_+\setminus \{0\}\to \N\times \R_+\setminus\{0\}$
 with $T(a)=(N(a), M(a))$ ($a\in A_+\setminus \{0\}$), be as in  \ref{Cuniformful} 
 (see 11.10 and {10.8} of \cite{eglnp})
 as mentioned in  {\bf Step 1}.

 Let $\dt_2'>0$ (in place of  $\dt$), let
 ${\cal G}_2\subseteq A$ (in place of ${\cal G}$) be a finite subset,
 let ${\cal H}_{2,0}\subseteq A_+\setminus \{0\}$ (in place of ${\cal H}$)
 be a finite subset, and let $K_2'\ge 1$ (in place of $K$) be an integer
as
 given by \ref{CLuniq} 
 for the above $T,$ $\ep_1/16$ (in place of $\ep$), and ${\cal F}_2.$

 \Wlog, we may assume that ${{H_1(\mathcal G_{2, C}), H_1({\cal H}_{1,1}\cup {\cal H}_{1,2})}}, {\cal H}_{2,0}\subseteq {\cal G}_2\subseteq A^{\boldsymbol 1}$  {{and 
 $\dt_2'<\min\{\dt_2, \gamma_0, \dt_1/2\}.$}}
 Choose $K_2\ge K_2'$ such that
 $1/K_2<\gamma_0/8.$
 Choose $b_{2,0}\in A_+\setminus \{0\}$ with
 \beq\label{TC0k-n2}
 \mathrm d_\tau(b_{2,0})<1/8(K_2+1).
 \eneq

 It follows from  Proposition \ref{Cuniformful} (11.10  and 10.7 of \cite{eglnp})
 that there are ${\cal G}_2$-$\dt_2'/64$-multiplicative \cpc s
 $\phi_{2,0}: A\to A$ and $\psi_{2,0}: A\to D_2$ for some ${{D_2=}}
 D_2\otimes e_{11}\subseteq D_2\otimes \mathrm {M_{2K_{2}+1}}\subseteq A$ with
 $D_2\in {\cal C}_0$  such that
$(D_2\otimes \mathrm {M_{2K_{2}+1}})\phi_{2,0}(A) = 0,$
 \beq\label{TCzeroK1-31}
&&\hspace{0.2in}\hspace{-0.4in} \|x-(\phi_{2,0}(x)\oplus\overbrace{\psi_{2,0}(x)\oplus \psi_{2,0}(x)\oplus\cdots\oplus \psi_{2,0}(x)}^{2K_2+1})\|<\min\{\ep_2/128, \dt_2'/128\},\quad x\in {\cal G}_2,	\\\label{TCzeroK1-31+}
&& \phi_{2,0}(a_1)\lesssim b_{2,0},\,\,\,\phi_{2,0}(a_1)\lesssim \psi_{2, 0}(a_1),
\eneq
and  $\psi_{2,0}(a_1)$ is strictly positive in $D_2,$  {{and, moreover
$\psi_{2,0}$ is $T$-${\cal H}_{2,0}\cup \{f_{1/4}(a_1)\}$-full in $D_2.$}}
{{As in {\bf Step 1}, we may assume that there is a unitary $w_1\in {\widetilde{A}}$ such
that 
\beq\label{43-20108-1}
w_1^*\phi_{2,0}(a_0)w_1\in \overline{D_2AD_2}
\,\,\,\,\,\,\text{(see  \eqref{43-200106-1}).}
\eneq
}}
 Define $\phi_{2,0}': A\to A$ by $\phi_{2,0}'(a)=\phi_{2,0}(a)\oplus  \psi_{2,0}(a)$ for all $a\in A.$
 {{Let}} $D_{2,1}=\mathrm{M}_{2K_2}(D_2)$ and $D_{2,1}'=\mathrm{M}_{2K_2+1}(D_2).$
 Let $j_2: D_2\to \mathrm{M}_{2K_2}(D_2)$ be defined by
 $$
 j_2(d)=\diag(\overbrace{d,d,...,d}^{2K_2})\rforal d\in D_2.
 $$


 Set
\beq\nonumber
 &&{{d_{2,00}'=\overbrace{\psi_{2,0}(a_1)\oplus \psi_{2,0}(a_1)\oplus\cdots\oplus \psi_{2,0}(a_1)}^{2K_2+1}\in D_{2,1}'.}}
 \eneq	
With $\imath_2: D_{2,1}'\to A$ the inclusion map, consider the induced map $\mathrm{Cu}^{\sim}(\imath_2): \mathrm{Cu}^{\sim}(D_{2,1}')\to \mathrm{Cu}^{\sim}(A)$.
 It follows {{from}} Theorem 1.0.1 of \cite{Robert-Cu} (as $C$ has stable rank one)
 that there  is a \hm\, $h_2':
 D_{2,1}'\to C$ such
 that
 \beq\label{TC0k1-35}
 \mathrm{Cu}^{\sim}(h_2')=\Gamma^{\sim}\circ \mathrm{Cu}^{\sim}(\imath_2),\,\,\,{\mathrm{in\,\,\,particular,}}\,\,\,
\la  h_2'(d_{2, 00}')\ra =\Gamma^{\sim}\circ \mathrm{Cu}^{\sim}(\imath_2)(\la d_{2, 00}'\ra).
 \eneq

 Let $h_2=(h_2')|_{D_{2,1}}.$  Denote by $C''=\{c\in C: ch_2(d)=h_2(d)c=0, \rforal d\in D_{2,1}\}.$
 Note that
 $$
h_2'(\psi_{2,0}(a)\oplus \overbrace{0\oplus 0\oplus\cdots\oplus 0}^{2K_2}){\in} {{C''}},
{{\rforal}}
a\in A.
$$
Define  $h_{2,0}': A\to C''$ by
\vspace{-0.12in}$$
h_{2,0}'(a) =h_2'(\psi_{2,0}(a)\oplus \overbrace{0\oplus 0\oplus \cdots \oplus 0}^{2K_2}), {{\rforal}}
a\in A.
$$
Define $L'_2: A\to C$ by, {{for all $a\in A,$}}
\vspace{-0.12in}\beq\label{43-20108-n2}
L'_2(a)=h_{2,0}'(a)\oplus h_2((\overbrace{\psi_{2,0}(a)\oplus \psi_{2,0}(a)\oplus \cdots \oplus \psi_{2,0}(a)}^{2K_2})
=h_{2,0}'(a)\oplus h_2\circ j_2(\psi_{2,0}).
\eneq
{{By \eqref{TC0k1-35},
\eqref{TCzeroK1-31}, \eqref{TCzeroK1-31+},  
and $1/K_2<\gamma_0/8,$ we 
have, for all $a\in {\cal G}_2,$
 \beq
 |\tau(h_2\circ j_2(\psi_{2,0}(a)))-\Gamma(\tau)(a)|
 =|\Gamma(\tau)(j_2(\psi_{2,0}(a))-\Gamma(\tau)(a)|<\gamma_0/128+\gamma_0/8.
 \eneq
It follows (see \eqref{43-20108-n2}) that
\beq\label{43-20107n10}
\hspace{0.2in}|\tau(L_2'(a))-\Gamma(\tau)(a)|<\gamma_0/128+\gamma_0/8+\gamma_0/8, \rforal a\in {\cal G}_2\andeqn \rforal \tau\in T(C).
\eneq}}
 {{Since ${\mathrm{Cu}^\sim(H_1)=\mathrm{Cu}}^\sim (H)=(\Gamma^\sim)^{-1},$ 
 $\Gamma(\tau)(H_1(x))=\tau(x)$ for all $x\in C$ and $\tau\in T(C).$
 Thus}}
\beq\label{TC0k1-36}
\sup\{ |\tau\circ L'_2\circ 
{{H_1}}(x)-\tau(x)|: \tau\in T(C)\}<\gamma_0 {{\le \gamma_1,\rforal  x\in {\cal H}_{1,1}\cup {\cal H}_{1,2}.}}
\eneq
This implies that, in particular, 
\beq\label{TCk1-37}
\tau(L'_2\circ
{{H_1}}(b))\ge \Delta(\hat{b}),\quad b\in {\cal H}_{1,1}.
\eneq
Note also that, by construction of $C$, $\mathrm{K}_0(C)=\mathrm{K}_1(C)=\{0\},$
and so we may apply 
7.8 of \cite{eglnp}.
In this way, {{by \eqref{TC0k1-36} and 
\eqref{TCk1-37},}} we obtain a unitary $V_1\in {\widetilde C}$ such that
\beq\label{TCk1-38}
\|{\text{Ad}}\, V_1\circ L'_2\circ H_1(a)-{\text{id}}_{C}(a)\|<\eta_1/2, {{\rforal}}
a\in {\cal S}_1.
\eneq
Set $L_2={\text{Ad}}\, V_1\circ L'_2.$
We have the diagram
 \begin{displaymath}
\xymatrix{
A \ar[r]^{\id} \ar[d]_{L_1} & A \ar[d]^{L_2}\\
C \ar[ur]_{H_1}\ar[r]_\id & C,
}
\end{displaymath}
with the upper triangle approximately commuting on $\mathcal F_1$ to within $\ep$ and the lower triangle approximately commuting on $\mathcal S_1$ to within $\eta_1$.  {Also note that $L_2$ is ${\cal G}_2$-$\dt_2'/64$-multiplicative.}

{\bf Step 4}: Show that the process continues.

We will repeat the argument of {\bf Step 2}.

{{Recall}}
\vspace{-0.1in}\beq\label{TC0k1-6n}
\mathrm{Cu}^{\sim}(H)=(\Gamma^{\sim})^{-1}.
\eneq
{{Thus}}
\vspace{-0.1in}\beq\label{TC0k1-49}
\la  H\circ {{h_{2,0}''(a_1)}}\ra \le \la \psi_{2,0}(a_1)\ra
\eneq
in the C*-algebra $A,$
{{where $h_{2,0}''={\mathrm{Ad}}\, V_1\circ h_{2,0}'.$
Put $h_2^\sim={\mathrm{Ad}}\, V_1\circ h_2.$}}

Choose $\dt_2/4>\eta_1>0$ such that
\beq\label{TC0k1-50}
\hspace{-0.2in}\|f_{\eta_1}(H\circ {{h_{2,0}''(a_1)}}) x-x\|,\,\,\|x-xf_{\eta_1}( H\circ {{h_{2,0}''(a_1)}})\|<
\min\{\ep_2/128, \dt_2'/128\}
\eneq
for all $x\in  H\circ h_{2,0}'({\cal G}_2).$
Since $A$ has stable rank one, by a result of R\o rdam
(see also 3.2 of \cite{eglnp}), there is a unitary $u_{2,0}\in {\widetilde A}$ such that
\beq\label{TC0k1-51}
u_{2,0}^*f_{\eta_1}( H\circ {{h_{2,0}''(a_1)}})u_{2,0}\in
\overline{\psi_{2,00}(a_1)A\psi_{2,00}(a_1)}= \overline{D_2AD_2},
\eneq
where
\vspace{-0.1in}$$
\psi_{2,00}(a_1)=(\psi_{2,0}(a_1) \oplus \overbrace{0\oplus 0\oplus\cdots\oplus 0}^{2K_2})).
$$
{{Set $A_{2,0}'=\overline{u_{2,0}^*f_{\eta_1}( H\circ h_{2,0}''(a_1))u_{2,0}Au_{2,0}^*f_{\eta_1}( H\circ h_{2,0}''(a_1))u_{2,0}}.$
Note that $A_{2,0}'$ is a hereditary \SCA\, of $A.$}} 
Define $H'': A\to   {{A_{2,0}'\subset \overline{D_2AD_2}}}$
by
$$
H''(a)=u_{2,0}^*(f_{\eta_1}(H\circ h_{2,0}''(a_1)))H\circ h_{2,0}''(a)(f_{\eta_1}(H\circ h_{2,0}''(a_1)))u_{2,0} {{\rforal}} a\in A.
$$
Note that $H''$ is a ${\cal G}_2$-$\dt_2'/32$-multiplicative \cpc.
Moreover, by \eqref{TC0k1-50},
\beq\label{TC0k1-52+}
\|{\text{Ad}}\, u_{2,0}\circ H\circ h_{2,0}''(a)-{H''(a)}\|<\min\{\ep_2/128, \dt_2'/128\} {{\rforal}} a\in {\cal G}_2.
\eneq

Consider the two \hm s
${\text{Ad}}\, u_{2,0}\circ H\circ {{h_2^\sim}}\circ j_2$ and $\imath_2\circ j_2.$
Then, by \eqref{TC0k1-35} and \eqref{TC0k1-36},
\beq\label{TC0k1-57}
\mathrm{Cu}^{\sim}({\text{Ad}}\, u_{2,0}\circ H\circ {{h_2^\sim}}\circ j_2)=\mathrm{Cu}^{\sim}(\imath_2\circ j_2).
\eneq
Put $A''=\{a\in A: a\perp   {{A_{2,0}'}}\}.$
Note that we may view 
both $\imath_2\circ j_2$ and ${\text{Ad}}\, u_{2,0}\circ H\circ h_2^\sim\circ j_2$
as maps into $A''.$
It follows from Theorem 3.3.1 of \cite{Robert-Cu}, as $A''$,
a hereditary subalgebra,
 has stable rank one, that
there exists a unitary $u_{2,1}\in {\widetilde A''}$ such that
\beq\label{TC0k1-8n}
\hspace{0.3in}\|u_{2,1}^* ({\text{Ad}}\, u_{2,0}\circ {{H\circ h_2^\sim}}\circ j_2(x))u_{2,1}-\imath_2\circ j_2(x)\|<\min\{\ep_2/16,\dt_2'/16\}
{{\rforal x\in \psi_{2,0}({\cal G}_2).}}
\eneq
\noindent
Writing $u_{2,1}=\lambda+z'$ for some $z'\in A''.$ Therefore
we may view $u_{2,1}$ as a unitary in ${\widetilde A}.$ Note that, for any
$b\in {{\overline{D_2AD_2},}}$
 $u_{2,1}^*bu_{2,1}=b.$
In particular, for any $a\in A,$
\beq\label{TC0k1-15n}
{\text{Ad}}\, u_{2,1}\circ H''(a)=H''(a)\rforal a\in A.
\eneq
{{Note that the map $\imath''\circ \psi_{2,0}: A\to \overline{D_2AD_2}$ is $T$-${\cal H}_{1,0}\cup \{f_{1/4}(a_1)\}$-full (see the last remark 
of \ref{fullunif}), where $\imath'': D_2\to \overline{D_2AD_2}$ is the embedding.}}
By Corollary \ref{CLuniq},  there is a unitary $u_{2,2}\in {\widetilde A}$ {{(see \eqref{43-20108-1})}}
such that
\beq\label{TC0k1-16n}
\|{\text{Ad}}\, u_{2,2}\circ (H''(a)\oplus \imath_2\circ j_2\circ \psi_{2,0}(a))-({\phi_{2,0}'}(a)\oplus  j_2\circ \psi_{2,0}(a))\|<\ep_2/16
\eneq
 for all $a\in {\cal F}_2.$  {{Recall that $H\circ L_2(a)=H\circ h_{2,0}''(a) \oplus H\circ h_2^\sim\circ j_1\circ \psi_0(a)$
 for $a\in A$ (see \eqref{43-20108-n2}) and the line after \eqref{TC0k1-49}.}}
 Combining {{with  \eqref{TC0k1-52+},}} \eqref{TC0k1-8n},  and \eqref{TC0k1-15n},
 we have
 \beq\label{TC0k1-17n}\hspace{0.5in} \|{\text{Ad}}\, (u_{2,0}u_{2,1} u_{2,2})\circ H\circ L_2(a)-{\text{Ad}}\, u_{2,2}\circ (H''(a)\oplus \imath_2\circ j_2\circ \psi_{2,0}(
a))\|<{{\ep_2/128+\ep_2/16}},
\eneq
for all $a\in {\cal F}_2.$
On the other hand,  by \eqref{TCzeroK1-31},
 \beq\label{TC0k1-18n}
 \|{\text{id}}_A(a)-({\phi_{2,0}'}(a)\oplus j_2\circ \psi_{2,0}(a))\|<\ep_2/16\rforal a\in {\cal F}_2.
 \eneq
 Set $U_2=u_{2,0}u_{2,1}u_{2,2}.$  By \eqref{TC0k1-18n},
\eqref{TC0k1-16n}, 
 and \eqref{TC0k1-17n},
 we conclude that
 \beq\label{TC0k1-19n}
 \|{\text{id}}_A(a)-{\text{Ad}}\, U_2\circ H\circ L_2(a)\|<\ep_2\rforal a\in {\cal F}_2.
 \eneq
 Thus, we have expanded the diagram above to the diagram
  \begin{displaymath}
\xymatrix{
A \ar[r]^{\id} \ar[d]_{L_1} & A \ar[d]^{L_2} \ar[r]^{\id}   & A\\
C \ar[ur]_{H_1}\ar[r]_\id & C
\ar[ur]_{H_2}
}\,\,,
\end{displaymath}
where $H_2:=\mathrm{Ad}{{U_2}}\circ H$ {{(which is a \hm),}} with the last triangle approximately commuting on ${\cal  F}_2$ to within $\ep_2{{(<\ep/2)}}.$

After continuing in this way (to construct $L_3$ and so on), the Elliott approximate intertwining argument (see \cite{Ell-AT-RR0}, Theorem 2.1)
shows that $A$ and $C$ are isomorphic.
\end{proof}

\begin{cor}\label{CZtW}
Let $A$ be a non-unital simple separable amenable \CA\, with continuous scale  and
satisfying the UCT. Suppose that $A\in {\cal D}$ and $K_0(A)=\ker\rho_A$, {where $\rho_A$ is the canonical map $K_0(A) \to \mathrm{Aff}(\mathrm{T}(A))$.}
Suppose that $B\in {\cal D}_{0}$ satisfies the UCT, has continuous scale  and
satisfies $\mathrm{K}_0(B)=\mathrm{K}_1(B)=\{0\},$ and suppose
that there is an affine homeomorphism $\gamma: \mathrm{T}(B)\to \mathrm{T}(A).$
Then there is an embedding $\phi: A\to B$ such that $\phi_T=\gamma.$
\end{cor}

\begin{proof}
Since $\ker\rho_A=\mathrm{K}_0(A),$
then, in the previous proof,  $\Gamma$ (extended 
to be zero on $\mathrm{K}_0(A)$) 
now
gives a \hm\, from
$\mathrm{Cu}^{\sim}(A)$, which is equal to $\mathrm{K}_0(A)\sqcup {\mathrm{LAff}}_+^{\sim}({\mathrm{T}}(A)),$
by 
6.2.3 
of \cite{Robert-Cu} and 7.3 of \cite{eglnp},  to $\mathrm{Cu}^{\sim}(C),$ 
where $C$ is a simple inductive limit of Razak algebras with continuous scale such that $ \mathrm{T}(A)\cong \mathrm{T}(C)$. Note that it follows from Theorem \ref{TTMW} that $C\cong B$.
We simply omit the construction of $H_1$ and keep Step 1  and Step 3
(in the 
(new) 
first step now we ignore anything related to Step 2).  A one-sided Elliott intertwining
yields a \hm\, from $A$ to $C.$ 
\end{proof}

\section{Tracial approximation and  non-unital versions of some results of Winter}

\begin{lem}{{\rm{(Prop. 2.1 of \cite{Winter-TA})}}}\label{comp}
Let $A$ be a simple \CA\, (with or without unit) belonging to the reduction class $\mathcal R$, and assume that $A$ has strict comparison. 

Let $F$ be a finite dimensional \CA\,, and let
\beq
\phi: F\to A\tand
\phi_i: F\to A\tforal  i\in \mathbb N
\eneq
be {{c.p.c.}} order-zero maps such that for each $c\in F_+$ and $f\in\mathrm{C}_{0}^{+}((0, 1])$,
\beq
\lim_{i\to\infty}\sup_{\tau\in{\mathrm{T(A)}}} |\tau({{f(\phi)}}(c)-{f(\phi_i)(c)}| = 0\tand\\
\limsup_{i\to\infty} \|{f(\phi_i)(c)}\| \leq \|{f(\phi)(c)}\|. 
\eneq
It follows that there are contractions
$$s_i\in \mathrm{M}_4 \otimes  A\tforal  i\in\mathbb N$$
such that 
\beq
&&\lim_{i\to\infty}\|s_i(1_4\otimes \phi(c)) - (e_{1, 1} \otimes \phi_i(c))s_i\| = 0\tforal  c\in {{F_+}}\tand\\
&&\lim_{i\to\infty}\|(e_{1, 1} \otimes \phi_i(c))s_is_i^*- e_{1, 1} \otimes \phi_i(c)\| = 0. 
\eneq
\end{lem}
{{(See 4.2 of \cite{WZ-OR0} for the definition of $f(\psi)$ where $\psi$ is an order zero map.)}}
\begin{proof}
The proof is the same as for Proposition 2.1 of \cite{Winter-TA} (the argument does not require the \CA\, to be unital; the hypothesis of strict comparison is sufficient for the argument to proceed).
\end{proof}

The following 
lemma
is a slight modification of 
4.2 of \cite{Winter-Z-stable-02}.

\begin{lem}\label{app-unit}
Let $A$ be a separable \CA\, with nuclear dimension at most $m$. Let $(e_n)$ be an increasing approximate unit for $A$.
Then there is a sequence of $(m+1)$-decomposable completely positive approximations
$$
\xymatrix{
\widetilde{A} \ar[r]^-{\tilde{\psi}_j} & F_j^{(0)}\oplus F_j^{(1)}\oplus\cdots\oplus F_j^{(m)} \oplus \Comp \ar[r]^-{\tilde{\phi}_j} & \widetilde{A},\quad j=1, 2, ...
}
$$ 
(i.e., each ${\tilde{\phi}_j}|_{F_j^{(l)}}$ is of order zero) such that, for each $j=1, 2, ...$,
\begin{equation}\label{m-decom-cond-1}
\tilde{\phi}_j(F_j^{(l)})\subseteq A,\quad l=0, 1, ..., m,
\end{equation}
\begin{equation}\label{m-decom-cond-2}
\tilde{\phi}_j|_\Comp(1_{\Comp}) = 1_{\widetilde{A}}-e_{n_j},\quad\textrm{for some $e_{n_j}$ in the approximate unit $(e_n)$,
and}
\end{equation} 
\begin{equation}\label{refine-0} 
 \lim_{j\to\infty}\|\tilde{\phi}_j\tilde{\psi}_j(a)-a\|=0, \,\,\,
\lim_{j\to\infty}\|\tilde{\phi}^{(l)}_j\tilde{\psi}^{(l)}_j(1_{\widetilde{A}}) a - \tilde{\phi}_j^{(l)}\tilde{\psi}^{(l)}_j(a)\|=0,\quad l=0,
 1, ..., m, \ a\in \widetilde{A},
\end{equation} 
where $\tilde{\phi}^{(l)}_j$ and $\tilde{\psi}^{(l)}_j$ are the restriction of $\tilde{\phi}_j$ to $F_j^{(l)}$ and the projection of $\tilde{\psi}_j$ to $F_j^{(l)}$, respectively.
\end{lem}

\begin{proof}
Let $\mathcal F\subseteq \widetilde{A}$ be a finite set of positive elements with norm one,  and let $\eps>0$ be arbitrary. 
{{Each element $a\in {\cal F}$ may be written as 
$\pi(a)\cdot 1_{\widetilde{A}}+x(a),$ where  $\pi: \widetilde{A} \to \Comp$ is the canonical quotient map and 
and $x(a)\in A.$}}
{{Let $\{e_n\}$ be an approximate identity of $A$ with $e_{n+1}e_n=e_ne_{n+1},$ $n=1,2,....$
Choose $N$ such that
\beq
\|e_Nx(a)e_N-x(a)\|<\eps/4\andeqn \|x(a)\|\le 2\rforal a\in {\cal F}.
\eneq
Set $e=e_{N+1}$ and, for $a\in {\cal F},$ $a'=\pi(a)\cdot 1_{\tilde A}+e_Nx(a)e_N.$}}
{{It follows that $a-a'\in A.$ Moreover,}}
\beq\label{52-200110-1}
\|a-a'\|<\eps/4,\quad a'e=ea',\quad\mathrm{and}\quad 
(a'-\pi(a')\cdot 1_{\widetilde{A}})(1-e)=0
\eneq
where $\pi: \widetilde{A} \to \Comp$ is the canonical quotient map.
Denote by $\mathcal F'$ the set of such $a'$.
{{Let ${\cal F}_1=\{e^{\frac{1}{2}} a' e^{\frac{1}{2}}, e^{\frac{1}{2}} (a-a') e^{\frac{1}{2}}: a\in {\cal F}, a'\in {\cal F}'\}.$}}
Then choose a factorization
$$
\xymatrix{
A \ar[r]^-{\psi} & F^{(0)}\oplus F^{(1)}\oplus\cdots\oplus F^{(m)} \ar[r]^-{\phi} & A
}
$$ 
such that 
\beq\label{2020-728-1}
{{\|\phi(\psi(x)-x\|<\ep/4\rforal x\in {\cal F}_1,}}
\eneq
%
and
the restriction of $\phi$ to each direct summand $F^{(l)}$, $l=0, 1, ..., m$, is of order zero.

Then, define maps
\beq
\tilde{\psi}: \widetilde{A}\ni a \mapsto \psi(e^{\frac{1}{2}}ae^{\frac{1}{2}})\oplus \pi(a)  \in (F^{(0)}\oplus F^{(1)}\oplus\cdots\oplus F^{(m)})\oplus\Comp, \ \mathrm{and}\\
\tilde{\phi}: (F^{(0)}\oplus F^{(1)}\oplus\cdots\oplus F^{(m)})\oplus\Comp\ni (a, \lambda) \mapsto \phi(a) + \lambda(1-e). 
\eneq
For any $a\in\mathcal F$, one has,
\begin{eqnarray*}
\|\tilde{\phi}(\tilde{\psi}(a)) - a\|  &=& {{\|\tilde{\phi}(\tilde{\psi}(a')+\tilde{\phi}(\tilde{\psi}(a-a'))-a'-(a-a')\|}}\\
&<& {{\|\tilde{\phi}(\tilde{\psi}(a')) - a'\|+\|\tilde{\phi}(\tilde{\psi}(a-a'))-(a-a')\|}}\\
&=&{{\|\tilde{\phi}(\tilde{\psi}(a')) - a'\|+\|{\phi}({\psi}(a-a'))-(a-a')\|}}\hspace{0.3in} {{(\text{recall}\,\,a-a'\in A)}}\\
& < & \|\tilde{\phi}(\tilde{\psi}(a')) - a'\|+{{\eps/4}}\hspace{1.7in} {{(\text{see}\,\, \eqref{2020-728-1})}} \\
& = &\|\phi(\psi(e^{\frac{1}{2}}a'e^{\frac{1}{2}})) +\pi(a')(1-e) -a' \| + \eps/4 \\
& < & \| e^{\frac{1}{2}}a'e^{\frac{1}{2}} + \pi(a')(1-e) - a'\| + \eps/2<\eps\hspace{0.5in} {{(\text{see}\,\,\eqref{52-200110-1}).}} 
\end{eqnarray*}
It is clear that the restriction of $\tilde{\phi}$ to each direct summand $F^{(l)}$ of $F^{(0)}\oplus F^{(1)}\oplus\cdots\oplus F^{(m)}\oplus\Comp$, $l=0, 1, ..., m$, has order zero. 

Since $\mathcal F$ and $\eps$ are arbitrary, one obtains the $(m+1)$-decomposable  completely positive approximations $(\tilde{\psi}_j, \tilde{\phi}_j)$, $j=1, 2, ...$, which satisfy \eqref{m-decom-cond-1} and \eqref{m-decom-cond-2} of the lemma.

In the same way as in the proof of Proposition 4.2 of \cite{Winter-Z-stable-02}, $\tilde{\psi}_j$ and $\tilde{\phi}_j$ can be modified to satisfy \eqref{refine-0}. Indeed, consider
the maps
$$\hat{\psi}_j: \widetilde{A} \ni a \mapsto \tilde{\psi}_j(1_{\widetilde{A}})^{-\frac{1}{2}}\tilde{\psi}_j(a) \tilde{\psi}_j(1_{\widetilde{A}})^{-\frac{1}{2}} \in (F_j^{(0)}\oplus F_j^{(1)}\oplus\cdots\oplus F_j^{(m)})\oplus\Comp,$$
where the inverse is taken in the hereditary sub-\CA\, generated by $\tilde{\psi}(1_{\widetilde{A}})$,
and
$$\hat{\phi}_j: (F_j^{(0)}\oplus F_j^{(1)}\oplus\cdots\oplus F_j^{(m)})\oplus\Comp \ni a\mapsto {\tilde \phi}_j(\tilde{\psi}_j(1_{\widetilde{A}})^{\frac{1}{2}} a \tilde{\psi}_j(1_{\widetilde{A}})^{\frac{1}{2}}) \in \widetilde{A}.$$
Then the proof of Proposition 4.2 of \cite{Winter-Z-stable-02} shows that
$$\lim_{j\to\infty}\|\hat{\phi}_j^{(l)}\hat{\psi}^{(l)}_j(a) - \hat{\phi}^{(l)}_j\hat{\psi}^{(l)}_j(1_{\widetilde{A}}){{ \td{\phi}_j\td{\psi}_j(a)}}\|
=0,\quad a\in A, \quad l = 0, 1, ... ,m.$$

Note that $\pi(1_{\widetilde{A}}) = 1_{\Comp}$. One has that $\pi(1_{\widetilde{A}})\Comp\pi(1_{\widetilde{A}})=\Comp$, and the restriction of $\hat{\phi}$ to $\Comp$ is the map $\lambda \mapsto \lambda(1-e)$. It follows that the decompositions $(\hat{\psi}_j, \hat{\phi}_j)$ satisfy the requirements of the lemma.
\end{proof}
 
 \begin{defn}\label{DfS}
 In the next statement, 
 denote by ${\cal S}$ a  fixed  
class
of  
non-unital separable amenable \CA s $C$ 
 such that $\mathrm{T}(C)\not={\O}$ and $0\not\in \overline{\mathrm{T}(C)}^\mathrm{w}.$ 
 If $C\in \mathcal S$ and $e_C\in C$ is a strictly positive element, 
 define $\lambda_s(C)=\inf\{\mathrm{d}_\tau(e_C): \tau\in \overline{\mathrm{T}(C)}^\mathrm{w}\},$ {{where $\mathrm{d}_\tau(e_C) := \lim_{\ep\to 0} \tau(f_\ep(e_C)).$}}
 
 {{Suppose that $C=\overline{\cup_{n=1}^\infty C_n}$ is a simple \CA\, such that 
 $C_n\subset C_{n+1}$ and $C_n\in {\cal S},$ $n\in \N.$  Suppose that $C$ has continuous scale.
 In the following statement, we assume that there are $e_n\in {C_n}_+$ with $\|e_n\|=1$ 
 satisfies}}
 
  {{(1) $\{e_n\}$ forms an approximate identity for $C$ and 
  $d_t(e_n)>1-1/n$ for all $t\in T(C_m)$ for all $m\ge n.$}}
 
{{ This, in fact, is always the case when $C_n\in {\cal S}$ and $C$ has continuous scale. 
 Let $c_n\in C_n$ be a strictly positive element with $\|c_n\|=1.$ 
 Then $c=\sum_{n=1}^\infty c_n/2^{n+1}$ is a strictly positive element of $C.$ 
 Thus $\{c^{1/k}\}$ forms an approximate identity for $C.$  Since $C$ has continuous scale,
 $\tau(c^{1/k})\nearrow 1$ uniformly on $T(C).$   Put $d_n=\sum_{j=1}^n c_j/2^{j+1}.$
 Then $d_n^{1/k}\le d_m^{1/k_1}$ if $n\le m$ and $k<k_1.$ Note that $d_n^{1/k}\in C_n.$
 It follows that a  choice of subsequence of the form $\{d_n^{1/k}\}$ forms an approximate identity.
 So, passing to a subsequence, we relabel it as $c_n\in C_n.$ 
 Note that $\tau(c_n)\to 1$ uniformly to 1 on $T(C).$  We may assume that $\tau(c_n)>1-1/2n$ for all 
 $\tau\in T(C).$ One then shows that, for each fixed $n,$ there is $N(n)\ge n$ such
 that $\tau(c_n)>1-1/n$ for all $\tau\in T(C_m)$ for all $m\ge N(n),$ using a weak * compactness argument.
 This, by passing to another subsequence, implies (1) holds.}}
 
 {{Note also condition (1) implies that $\lambda_s(C_n)\ge 1-1/n.$}}

 \end{defn}
 
 The following is a non-unital version of 2.2 of \cite{Winter-TA}.

\begin{thm}\label{fdim}

Let $A$ be a stably  projectionless separable simple \CA\, in ${\cal R}$ 
with $\mathrm{dim}_{\mathrm{nuc}} A=m<\infty.$

Fix a 
 positive element $e\in A_+$ with $0\le e\le 1$
such that $\tau(e),\,\tau(f_{1/2}(e))\ge r_0>0$ for all $\tau\in \mathrm{T}(A).$
Let $C=\overline{\bigcup_{n=1}^{\infty} C_n}$ be a non-unital simple \CA\, with continuous scale,  
where $C_n\subseteq C_{n+1}$ and $C_n\in {\cal S}$ {{which also satisfies condition (1) in \ref{DfS}.}}
Suppose that there is an affine homeomorphism $\Gamma:  \mathrm{T}(C)\to \mathrm{T}(A)$ 
and suppose that there 
are sequences
of \cpc s $\sigma_n: A\to C$ and
\hm s $\rho_n: C\to A$ such that
\beq\label{TWv-1}
&&\lim_{n\to\infty}\|\sigma_n(ab)-\sigma_n(a)\sigma_n(b)\|=0,\quad a, b\in A,\\\label{TWv-2}
&&\lim_{n\to\infty}\sup\{|t\circ \sigma_n(a)-\Gamma(t)(a)|: t\in \mathrm{T}(C)\}=0,\quad a\in A,\\\label{TWv-2+}
&&\lim_{n\to\infty}\sup\{|\tau(\rho_n\circ \sigma_n(a))-\tau(a)|:\tau\in \mathrm{T}(A)\}=0,\quad a\in A, \tand
\eneq
$\sigma_n(e)$ is strictly positive in $C$ for all $n\in \N.$

Then $A$ has the following property: For any finite set $\mathcal F\subseteq A$ and any $\eps>0$, there are a projection $p\in \mathrm{M}_{4(m+2)}(\widetilde{A})$, a sub-\CA\, $S\subseteq p\mathrm{M}_{4(m+2)}(A)p$ with $S\in\mathcal {\cal S},$ and 
an  ${\cal F}$-$\ep$-multiplicative \cpc\, $L: A\to S$ such that
\begin{enumerate}
\item $\|[p, 1_{4(m+2)} \otimes a]\| < \eps$, $a\in\mathcal F$,
\item $p(1_{4(m+2)} \otimes a) p\in_\eps S$, $a\in\mathcal F$,
\item $\|L(a)-p(1_{4(m+2)} \otimes a) p\|<\ep,$ $a\in \mathcal F,$
\item $p\sim e_{11}$ in $\mathrm{M}_{4(m+2)}(\widetilde{A})$,
\item $\tau(L(e)),\, \tau(f_{1/2}(L(e)))>7r_0/32(m+2)$ for all $\tau\in \mathrm{T}(\mathrm{M}_{4(m+2)}(A)),$
\item ${{(1_{4(m+2)}-p)\mathrm{M}_{4(m+2)}(A)(1_{4(m+2)}-p)}} \in \mathcal R$, and
\item $t(f_{1/4}(L(e)))\ge (3r_0/8)\lambda_s(C_1)\tforal t\in \mathrm{T}(S).$
\end{enumerate}
\end{thm}

\begin{proof}
Since $A$ has finite nuclear dimension, one has that $A\cong A\otimes \mathcal Z$ (\cite{Winter-Z-stable-02} for the unital case and \cite{T-0-Z} for the non-unital case). Therefore, $A$ has strict comparison for positive elements (Corollary 4.7 of \cite{Ror-Z-stable}). 

The proof is essentially the same as that of Theorem 2.2 of \cite{Winter-TA}. We give the proof in the present very much analogous situation for the convenience of the reader. 
Let  $e\in A_+$ with 
$\|e\|=1$, {{$\tau(e) > r_0$}}, and $\tau(f_{1/2}(e))>r_0$ for all $\tau\in \mathrm{T}(A).$ 

 Let $(e_n)$ be an 
(increasing)
 approximate unit for  $A$. Since $A\in\mathcal R$, and
 since $A$ is also assumed to be projectionless, one may assume that $\mathrm{sp}(e_n)=[0, 1]$. Since $\mathrm{dim}_\mathrm{nuc}(A)\leq m$, by Lemma \ref{app-unit}, there is a system of $(m+1)$-decomposable completely positive approximations
$$
\xymatrix{
\widetilde{A} \ar[r]^-{\psi_j} & F_j^{(0)}\oplus F_j^{(1)}\oplus\cdots\oplus F_j^{(m)} \oplus \Comp \ar[r]^-{\phi_j} & \widetilde{A},\quad j=1, 2, ...
}
$$ 
such that 
\beq
&&\phi_j(F_j^{(l)})\subseteq A, \quad l=0, 1, ..., m,\andeqn\\
\label{defn-phi-j}
&&\phi_j|_\Comp(1_\Comp) = 1_{\widetilde{A}}-e_{j},
\eneq
where $e_j$ is an element of $(e_n)$.

Write
$$\phi_{j}^{(l)} = \phi_j|_{F_j^{(l)}}\quad\mathrm{and}\quad {\phi_{j}^{(m+1)}}= \phi_j|_\Comp,\quad l=0, 1, ..., m.$$
As in Lemma \ref{app-unit}, one may assume that
\begin{equation}\label{refine}
\lim_{j\to\infty}\|\phi^{(l)}_j\psi^{(l)}_j(1_{\widetilde{A}})a - \phi_j^{(l)}\psi^{(l)}_j(a)\|=0,\quad l=0 ,1 , ..., m, \ a\in A.
\end{equation}
Note that $\phi_{j}^{(l)}: F_j^{(l)} \to A$ is of order zero, 
and
the relation for 
an
order zero map is weakly stable 
(see  (${\mathcal P}$) and (${\mathcal P} 1$) of 2.5 of \cite{KW}. 
On the other hand, if $i$ 
is
large enough, then $ \sigma_i\circ\phi_{j}^{(l)}$ satisfies the relation for order zero 
to within an arbitrarily small 
tolerance, since $\sigma_i$ will be sufficiently multiplicative. It 
follows that there are order zero maps $$\tilde{\phi}_{j, i}^{(l)}: F_j^{(l)} \to C$$
such that
$$\lim_{i\to\infty}\| \tilde{\phi}_{j, i}^{(l)}(c) - \sigma_i(\phi_{j}^{(l)}(c))\|=0,\quad c\in F_{j}^{(l)}.$$
We will identify $C$ with  $S_i=\rho_i(C)\subseteq A$, $\sigma_i: A \to C$ with $\rho_i\circ \sigma_i: A \to S_i\subseteq A$, and $\tilde{\phi}_{j, i}^{(l)}$ with $\rho_i\circ \tilde{\phi}_{j, i}^{(l)}$. There is 
a positive linear map (automatically order zero)
$$\tilde{\phi}_{j, i}^{(m+1)}: \Comp\ni 1 \mapsto 1_{\widetilde{A}}-\sigma_i(e_j)\in \tilde{S_i}=\textrm{C*}(S_i, 1_{\widetilde{A}})\subseteq \widetilde{A},\quad i\in\mathbb N.$$
Note that
\begin{equation}\label{defn-phi-j-p}
 \tilde{\phi}_{j, i}^{(m+1)}(\lambda) = \sigma_i(\phi_{j}^{(m+1)}(\lambda)),\quad \lambda \in F_{j}^{(m+1)}=\Comp,
\end{equation} 
where one still uses $\sigma_i$ to denote the induced map $\widetilde{A} \to \widetilde{S_i}$.

Note that for each $l=0, 1, ..., m$,
$$\lim_{i\to\infty}\|{f(\tilde{\phi}_{j, i}^{(l)})(c)} - \sigma_i ({f(\phi_{j}^{(l)})(c)})\| = 0,\quad c\in (F_j^{(l)})_+, \  f\in\mathrm{C}_0((0, 1])_+,$$
{{(see the comment  before the proof  of \ref{comp}  for the notation $f(\tilde{\phi}_{j, i}^{(l)})$ and
$f(\phi_j^{(l)})$)}} and hence, from \ref{TWv-2+},
$$\lim_{i\to\infty}\sup_{\tau\in{\mathrm{T}(A)}}| \tau({f(\tilde{\phi}_{j, i}^{(l)})(c)} - {f(\phi_{j}^{(l)})(c)})| = 0,\quad c\in (F_j^{(l)})_+, \  f\in\mathrm{C}_0((0, 1])_+.$$
Also note that
$$\limsup_{i\to\infty} \|f(\tilde{\phi}_{j, i}^{(l)})(c)\| \leq \| f(\phi_j^{(l)})(c)\|, \quad c\in (F_j^{(l)})_+, \  f\in\mathrm{C}_0((0, 1])_+.$$

Applying Lemma \ref{comp} to $(\tilde{\phi}_{j, i}^{(l)})_{i\in\mathbb N}$ and $\phi_j^{(l)}$ for each $l=0, 1, ..., m$, we obtain contractions $$s^{(l)}_{j, i} \in \mathrm{M}_4(A) \subseteq \mathrm{M}_4(\widetilde{A}),\quad i\in\mathbb N,$$
such that
\beq
&&\lim_{i\to\infty}\|s_{j, i}^{(l)}(1_4\otimes \phi_j^{(l)}(c)) - (e_{1, 1}\otimes\tilde{\phi}_{j, i}^{(l)}(c)) s_{j, i}^{(l)}\| = 0,\quad c\in F_j^{(l)}, \andeqn\\
&&\lim_{i\to\infty}\|(e_{1, 1} \otimes {\tilde \phi}_{j,i}^{(l)}(c))s_{j,i}^{(l)}(s_{j, i}^{(l)})^*- e_{1, 1} \otimes {\tilde \phi}_{j,i}^{(l)}(c)\| = 0. 
\eneq

Note that $\mathrm{sp}(e_j)=[0,1].$  Put $C_0=C_0((0,1]).$ Define 
\beq
\Delta_j(\hat{f})=\inf\{\tau(f(e_j)): \tau\in \mathrm{T}(A)\}\rforal f\in (C_0)_+\setminus\{0\}.
\eneq
Since $A$ is assumed to have continuous scale, $\mathrm{T}(A)$ is compact and $\Delta_j(\hat{f})>0$
for all $f\in {{(}}C_0)_+\setminus\{0\}.$
For $l=m+1$, since $\mathrm{sp}(e_j) = [0, 1]$, by considering $\Delta_j$ 
for each $j,$ since $i$ is chosen after $j$ is fixed,  by applying \ref{Cuniqcone}  in the appendix,
one obtains unitaries $$s^{(m+1)}_{j, i} \in \widetilde{A},\quad i\in\mathbb N,$$ such that $$\lim_{i\to\infty}\|s_{j, i}^{(m+1)}e_j-\sigma_i(e_j)s_{j, i}^{(m+1)}\|=0,$$ and hence
$$\lim_{i\to\infty}\|s_{j, i}^{(m+1)}(1_{\widetilde{A}}-e_j)-(1_{\widetilde{A}}-\sigma_i(e_j))s_{j, i}^{(m+1)}\|=0.$$
By \eqref{defn-phi-j} and \eqref{defn-phi-j-p}, one has
\begin{equation*}
\lim_{i\to\infty}\|s_{j, i}^{(m+1)}\phi_j^{(m+1)}(c) - \tilde{\phi}_{j, i}^{(m+1)}(c) s_{j, i}^{(m+1)}\| = 0,\quad c\in F_j^{(m+1)}=\Comp.
\end{equation*}

Considering the 
element
$e_{1, 1}\otimes s_{j, i}^{(m+1)} \in \mathrm{M_4}\otimes \widetilde{A}$, and still denoting it by $s_{j, i}^{(m+1)}$, we have
\begin{equation*}
\lim_{i\to\infty}\|s_{j, i}^{(m+1)}(1_4\otimes \phi_j^{(m+1)}(c)) - (e_{1, 1}\otimes\tilde{\phi}_{j, i}^{(m+1)}(c)) s_{j, i}^{(m+1)}\| = 0,\quad c\in F_j^{(l)}
\end{equation*}
and $$(e_{1, 1} \otimes {\tilde \phi}_{j,i}^{(m+1)}(c))s_{j,i}^{(m+1)}(s_{j, i}^{(m+1)})^* = e_{1, 1} \otimes 
{\tilde \phi}^{(m+1)}_{j,i}(c). $$

Therefore,
\begin{equation}\label{pre-s-1}
\lim_{i\to\infty}\|s_{j, i}^{(l)}(1_4\otimes\phi_j^{(l)}(c)) - (e_{1, 1} \otimes \tilde{\phi}_{j, i}^{(l)}(c)) s_{j, i}^{(l)}\| = 0,\quad c\in F_j^{(l)},\ l=0, 1, ..., m+1.
\end{equation}

\begin{equation}\label{pre-s-1-1}
\lim_{i\to\infty}\|(e_{1, 1}\otimes {\tilde \phi}_{j,i}(c))s_{j,i}^{(l)}(s_{j, i}^{(l)})^*- {\tilde \phi}_{j,i}(c)\| = 0, \quad c\in F_j^{(l)},\ l=0, 1, ..., m+1.
\end{equation}

Let $\tilde{\sigma}_i: \widetilde{A} \to 
\widetilde C $ and $\tilde{\rho}_i: 
\widetilde C \to\widetilde{A}$ denote the unital maps induced by 
$\sigma_i: {A} \to C$ and $\rho_i:  C\to {A}$, respectively.

Consider the contractions
$$s_j^{(l)} := (s_{j, i}^{(l)})_{i\in\mathbb N}\in (\mathrm{M}_4\otimes \widetilde{A})_\infty,\quad l=0, 1, ..., m+1, \quad j=1, 2, ....$$ 
By \eqref{pre-s-1} and \eqref{pre-s-1-1}, these satisfy
$$s_j^{(l)}(1_4\otimes \bar\iota(\phi_j^{(l)}(c))) = (e_{1, 1}\otimes \bar\rho\bar\sigma({\phi}_j^{(l)}(c))) s_{j}^{(l)}\andeqn$$
$$(e_{1,1}\otimes \bar{\rho}\circ \bar{\sigma}(\phi_j^{(l)}(c)))s_j^{(l)}(s_j^{(l)})^* =
 (e_{1,1}\otimes \bar{\rho}\circ {\bar{\sigma}}(\phi_j^{(l)}(c))),$$
where 
$$\bar{\sigma}: \widetilde{A}_\infty \to  \prod \widetilde{C}/\bigoplus \widetilde{C}~~~~\mbox{and}~~~~ \bar\rho: \prod \widetilde{C}/\bigoplus \widetilde{C} \to \widetilde{A}_\infty$$
are the homomorphisms induced by $\tilde{\sigma}_i$ and $\tilde{\rho}_i$,  
and the map
$$\bar\iota: (\widetilde{A})_\infty \to ((\widetilde{A})_\infty)_\infty$$ is the embedding induced by the canonical embedding $\iota: \widetilde{A}\to (\widetilde{A})_\infty$.

Let $$\bar\gamma: \widetilde{A}_\infty \to ((\widetilde{A})_\infty)_\infty$$ denote the homomorphism induced by
the composed map
$$\bar{\rho}\bar{\sigma}: {{\widetilde{A}_\infty}} \to (\widetilde{A})_\infty,$$
For each $l=0, 1, ..., m+1$, let 
\beq
&&\bar{\phi}^{(l)}: \prod_j F_j^{(l)}/\bigoplus_j F_j^{(l)} \to A_\infty\quad\mathrm{and}\\
&&\bar{\psi}^{(l)}: A \to \prod_j F_j^{(l)}/\bigoplus_j F_j^{(l)}
\eneq	
denote the maps induced by $\phi_j^{(l)}$ and $\psi_j^{(l)}$. 

Consider the  contraction
$$\bar{s}^{(l)} = (s_j^{(l)}) \in (\mathrm{M}_4\otimes \widetilde{A}_\infty)_\infty.$$
Then
$$
{\bar{s}^{(l)}}(1_4\otimes \bar\iota\bar{\phi}^{(l)}\bar\psi^{(l)}(a)) = (e_{1, 1} \otimes \bar\gamma\bar\phi^{(l)}\bar\psi^{(l)}(a)) {\bar{s}^{(l)}},\quad a\in \widetilde{A},\andeqn$$
$$
\hspace{-0.7in}(e_{1, 1} \otimes \bar\gamma\bar\phi^{(l)}\bar\psi^{(l)}(a)) {\bar{s}^{(l)}}({\bar{s}^{(l)}})^* = (e_{1, 1} \otimes \bar\gamma\bar\phi^{(l)}\bar\psi^{(l)}(a)).
$$

By \eqref{refine}, one has
$$\bar\phi^{(l)}\bar\psi^{(l)}(1_{\widetilde A})\iota(a) = \bar\phi^{(l)}\bar\psi^{(l)}(a),\quad a\in A.$$
In particular, 
$$((\bar\phi^{(l)}\bar\psi^{(l)}(1_{\widetilde A}))^{\frac{1}{2}}\iota(a) \in \textrm{C*}(\bar\phi^{(l)}\bar\psi^{(l)}(A)),$$
and hence
\begin{eqnarray}\label{twrist}
\bar{s}^{(l)}(1_4 \otimes (\bar{\iota}\bar{\phi}^{(l)}\bar{\psi}^{(l)}(1_{\widetilde{A}}))^{\frac{1}{2}})(1_4\otimes \bar{\iota}\iota(a)) &= & \bar{s}^{(l)}(1_4\otimes \bar{\iota}\bar{\phi}^{(l)}\bar{\psi}^{(l)}(1_{\widetilde{A}})^{\frac{1}{2}}\iota(a)) \nonumber \\
& = & (e_{1, 1} \otimes \bar{\gamma}(\bar{\phi}^{(l)}\bar{\psi}^{(l)}(1_{\widetilde{A}}))^{\frac{1}{2}}\iota(a)) \bar{s}^{(l)} \nonumber \\
& = & (e_{1, 1} \otimes \bar{\gamma}\iota(a)(\bar{\phi}^{(l)}\bar{\psi}^{(l)}(1_{\widetilde{A}}))^{\frac{1}{2}}) \bar{s}^{(l)} \nonumber \\
& = & (e_{1, 1} \otimes \bar{\gamma}(\iota(a)))(e_{1, 1}\otimes \bar{\gamma}(\bar{\phi}^{(l)}\bar{\psi}^{(l)}(1_{\widetilde{A}}))^{\frac{1}{2}})\bar{s}^{(l)}.
\end{eqnarray}
Set
\vspace{-0.1in}\begin{eqnarray*}
\bar{v} & = & \sum_{l=0}^{m+1} e_{1, l} \otimes ((e_{1, 1} \otimes \bar{\gamma}\bar{\phi}^{l}\bar{\psi}^{(l)}(1_{\widetilde{A}}))^{\frac{1}{2}}\bar{s}^{(l)}) \\
& = & \sum_{l=0}^{m+1} e_{1, l} \otimes (\bar{s}^{(l)}(1_4\otimes \bar{\iota}\bar{\phi}^{l}\bar{\psi}^{(l)}(1_{\widetilde{A}}))^{\frac{1}{2}}) \in \mathrm{M}_{m+2}(\Comp) \otimes \mathrm{M}_4(\Comp)\otimes ({\widetilde{A}}_\infty)_\infty.
\end{eqnarray*}
Then
\begin{eqnarray*}
\bar{v}\bar{v}^* & = & \sum_{l=0}^{m+1} e_{1, 1} \otimes (e_{1, 1} \otimes \bar{\gamma}\bar{\phi}^{l}\bar{\psi}^{(l)}(1_{\widetilde{A}})) =  e_{1, 1} \otimes e_{1, 1} \otimes \bar{\gamma}(1_{\widetilde{A}}).
\end{eqnarray*}
Thus, $\bar{v}$ is an partial isometry. 
Moreover, for any $a\in \widetilde{A}$,
\begin{eqnarray*}
\bar{v}(1_{(m+2)}\otimes 1_4 \otimes\bar{\iota}\iota(a)) & = & \sum_{l=0}^{m+1} e_{1, l} \otimes (\bar{s}^{(l)}(1_4\otimes \bar{\iota}\bar{\phi}^{l}\bar{\psi}^{(l)}(1_{\widetilde{A}})^{\frac{1}{2}})(1_4\otimes \bar{\iota}\iota(a)))\\
& = & \sum_{l=0}^{m+1} e_{1, l} \otimes (e_{1, 1} \otimes \bar{\gamma}(\iota(a)))(e_{1, 1} \otimes \bar{\gamma}(\bar{\phi}^{(l)}\bar{\psi}^{(l)}(1_{\widetilde{A}})^{\frac{1}{2}}\bar{s}^{(l)}) \quad\quad\textrm{(by \eqref{twrist})}\\
& = & (e_{1, 1}\otimes e_{1, 1} \otimes\bar{\gamma}(\iota(a))) \sum_{l=0}^{m+1} e_{1, l} \otimes e_{1, 1} \otimes \bar{\gamma}(\bar{\phi}^{(l)}\bar{\psi}^{(l)}(1_{\widetilde{A}})^{\frac{1}{2}}\bar{s}^{(l)}) \\
& = & (e_{1, 1}\otimes e_{1, 1} \otimes \bar{\gamma}(\iota(a)))\bar{v}.
\end{eqnarray*}
Hence
\begin{equation*}
\bar{v}^*\bar{v} (1_{m+2}\otimes 1_4 \otimes \bar\iota\iota(a))  =  \bar{v}^*(e_{1, 1}\otimes e_{1, 1} \otimes \bar\gamma\iota(a)) \bar{v} =  (1_{m+2}\otimes 1_4\otimes\bar\iota\iota(a)) \bar{v}^*\bar{v},\quad a\in \widetilde{A}.
\end{equation*}

Then, for any finite set $\mathcal G \subseteq \widetilde{A}$ and any  $\dt>0,$ there are 
{{$i \in \mathbb N$}} and ${{v_i}}\in \mathrm{M}_{m+2}(\Comp)\otimes \mathrm{M}_{4}(\Comp) \otimes \widetilde{A}$ such that
\beq\label{53-200110-10}
&&{{v_iv_i^*}}= e_{1, 1}\otimes e_{1, 1}\otimes \tilde{\rho}_i(1_{\tilde{S_i}})= e_{1, 1}\otimes e_{1, 1}\otimes 1_{\widetilde{A}},\\
&&\|[{{v_i^*v_i}}, 1_{m+2} \otimes1_4\otimes  a]\| < \dt \rforal 
a\in \mathcal G,\\
&&\|{{v_i^*v_i}}(1_{m+2}\otimes 1_4\otimes  a) - {{v_i}}^*(e_{1, 1}\otimes e_{1, 1}\otimes\tilde{\rho}_i\tilde{\sigma}_i(a)){{v_i}}\|< \dt
\rforal a\in \mathcal G\andeqn\\\label{53-200110-11}
&& \tau(\rho_i\circ \sigma_i(e)),\, \tau(f_{1/2}(\rho_i\circ \sigma_i(e)))\ge 15r_0/16\rforal \tau\in \mathrm{T}(A).
\eneq

Define
${{\kappa_i}}: \widetilde{S}_i \to \mathrm{M}_{m+2}\otimes \mathrm{M}_{4} \otimes \widetilde{A}$
by
$${{\kappa_i}}(s)={v_i}^*(e_{1, 1}\otimes e_{1, 1} \otimes{{\rho_i(s)}}){ v_i}.$$
Note that
$${{\kappa_i}}(S_i) \subseteq \mathrm{M}_{m+2}\otimes \mathrm{M}_{4}\otimes A.$$

Then ${{\kappa_i}}$ is an embedding; and on setting $p_i=1_{{\kappa_i}(\tilde{S}_i)}=v_i^*v_i$, one has
\begin{enumerate}
\item[(i)] ${p_i} \sim e_{1, 1}\otimes e_{1, 1}\otimes 1_{\widetilde A}$,
\item[(ii)] $\|[{p_i}, 1_{m+2}\otimes 1_4 \otimes a]\| <\dt$, $a\in \mathcal G$,
\item[(iii)] ${p_i}(1_{m+2}\otimes 1_4 \otimes a){p_i} \in_{\dt} {{\kappa_i}}(\tilde{S}_i)$, $a\in \mathcal G$.
\end{enumerate}
%
%
Note that $A$ is ${\cal Z}$-stable (by \cite{Winter-Z-stable-02}) and hence has strict comparison
(by \cite{Ror-Z-stable}).  Let $e\in (1_{4(m+2)}-{p_i})\mathrm{M}_{4(m+2)}(A)(1_{4(m+2)}-{p_i})$ 
be a strictly positive element.
By (i), ${\mathrm d}_\tau(e)=\tau(1_{4(m+2)}-p_i)=\tau(1_{4(m+2)}-e_{1, 1}\otimes e_{1, 1}\otimes 1_{\widetilde A})$
for all $\tau\in T(A)$, where $\tau$ is naturally extended to $\tilde{A}.$
Since $A$ and $M_{4(m+2)}(A)$ have continuous scale, $\tau\mapsto {\text{d}}_\tau$
is continuous on $T(A).$ Hence 
$(1_{4(m+2)}-{p_i})\mathrm{M}_{4(m+2)}(A)(1_{4(m+2)}-{p_i})$ also
has continuous scale (see 5.4 of \cite{eglnp}) and 
is still in the reduction class $\mathcal R$ (so condition (6) holds).
%
%

Define $L_i: A\to {{\kappa_i(S_i)}}$ by  $L_i(a)={v_i}^*(e_{1, 1}\otimes e_{1, 1} \otimes{\rho}_i(\sigma_i(a))){v_i}$
for all $a\in A.$  Then 
\begin{enumerate}
\item[(iv)] $\|L_i(a)-{p_i}(1_{4(m+2)}\otimes a){p_i}\|<\dt\rforal a\in {\cal G}$ and 
\item[(v)] $\tau(L_i(e)),\, \tau(f_{1/2}(L_i(e)))\ge \displaystyle{15r_0\over{64(m+2)}}\rforal \tau\in \mathrm{T}(\mathrm{M}_{4(m+2)}(A)).$
\end{enumerate}
{{Let $\tau_i\in \mathrm{T}(\kappa_i(S_i)).$ 
Then  $\tau_i\circ L_i$ is a positive linear functional. Let ${{\bar t}}$ be a weak *-limit of $\{\tau_i\circ L_i\}.$ 
Note that, for any $1/2>\ep>0,$ since $A$ has continuous scale, there is $e_A\in A$  with $\|e_A\|=1$
such that $\tau(e_A)>1-\ep/2$ for all $\tau\in T(A).$
By \eqref{TWv-2+} (see also \eqref{53-200110-10}),  we may assume that
$\tau_i\circ L_i(e_A)>1-\ep$ for all large $i.$ It follows that ${{\bar t}}(e_A)\ge 1-\ep.$ Hence $\|{{\bar t}}\|\ge 1-\ep$
for any $1/2>\ep>0.$ It follows that ${{\bar t}}$ is a state of $A.$  Then, by \eqref{TWv-1} and \eqref{TWv-2+},
${{\bar t}}$ is a tracial state of $A.$ 
Therefore,}} with sufficiently small $\dt$ and large ${\cal G}$ (and sufficiently large $i$), by also \eqref{53-200110-11}, 
we may assume that
\beq\label{lambdas-2}
t(f_{1/4}(L_i(e)))\ge 7r_0/8\rforal t\in \mathrm{T}({{\kappa_i(S_i)}}).\quad 
\eneq

Since ${{\kappa_i(S_i)}}\cong C,$ we may write ${{\kappa_i(S_i)}}=\overline{\bigcup_{n=1}^{\infty} S_{i,n}},$
where each $S_{i,n}\cong C_n$ and, 
{{by  condition (1) of \ref{DfS}, 
there exists a positive element $e_C\in S_{i,1}\subset S_{i,n}$
with $\|e_C\|=1$   such that $t(e_C)>\lambda_s(C_1)/2$ for all $t\in T(S_{i,n})$
for all $n\ge 1.$}}
%
Since each $S_{i,n}$ is amenable, 
there exists \cpc\, $\Phi_n: {{\kappa_i(S_i)}}\to S_{i,n}$ such 
that 
\beq\label{lambdas-1}
\lim_{n\to\infty}\|\Phi_n(s)-s\|=0
\rforal s\in {{\kappa_i(S_i)}}\,\,\,\mathrm{and}\,\,\,
\|\Phi_n(e_C)-e_C\|<1/2^{n+1}.
\eneq
We assert that, for all sufficiently large $n$,
for fixed $i$,
\beq\label{Claim518}
t(f_{1/4}(\Phi_n\circ L_i(e)))> (3r_0/8)\lambda_s(C_1)\rforal t\in \mathrm{T}(S_{i,n}).
\eneq
Otherwise, there exists a sequence $(n(k))$ and $t_k\in \mathrm{T}(S_{i,n(k)})$
such that
\beq\label{claim518-2}
t_k(f_{1/4}(\Phi_{n(k)}\circ L_i(e)))<{{(3r_0/8)}}\lambda_s(C_1).
\eneq
Note that, since $e_C\in S_{i,n(k)},$ $t_{k+1}|_{S_{i,n(k)}}\in T(S_{i,n(k)}).$
Let $t_0$ be a weak* limit of $\{t_k\circ \Phi_{n(k)}\}.$
Then, by \eqref{claim518-2},
\beq\label{claim518-3}
t_0(f_{1/4}(L_i(e)))\le (3r_0/8)\lambda_s(C_1).
\eneq
Note that $t_k(e_C)\ge \lambda_s(C_1)/2$ for all $k.$ Thus, by \eqref{lambdas-1}, one computes that 
$t_0(e_C)\ge  \lambda_s(C_1)/2.$  It follows that $t_0$ is a trace  of $S_i$ with $\|t_0\|\ge \lambda_s(C_1)/2.$ 
Then, by \eqref{lambdas-2},
\beq
t_0(f_{1/4}(L_i(e)))\ge (7r_0/8)(\lambda(C_1)/2).
\eneq
This contradicts \eqref{claim518-3} and so the assertion \eqref{Claim518} holds.
We then define $L=\Phi_n\circ L_i$ for some sufficiently large $n$ (and $i$). 
The conclusion of the theorem follows from (i),(ii), (iii), (iv), (v),
and \eqref{Claim518}.
\end{proof}

\begin{lem}\label{LWL1}
Let $A$ be a stably projectionless simple separable \CA\ with almost stable rank one (recall 
that
by definition this includes hereditary 
\SCA s).
Suppose that $A$ has continuous scale and
has 
strict comparison for positive element. Suppose also 
that map $\imath: \mathrm{W}_+(A)\to {\mathrm{LAff}}_{b,+}(\mathrm{T}(A))$ is surjective. 
Suppose that  there {are} $1>\eta>0$  and $1>\lambda>0$ such that every   hereditary \SCA\, $B$  with continuous scale has the following property:

Let {$r_0>0$ and} let ${a_0}\in B_+$ be a 
positive element
with $\|{a_0}\|=1$  with $\tau({a_0})\ge r_0$ and $\tau(f_{1/2}({a_0}))\ge r_0>0$ for all $\tau\in \mathrm{T}(B).$ Suppose that,
for any $\ep>0,$  any
finite subset ${\cal F}\subseteq B,$ 
there are ${\cal F}$-$\ep$-multiplicative \cpc s $\phi: B\to B',$ where $B'$ is a hereditary \SCA\, of $B,$  and  $\psi: B\to D$  for some
\SCA\, $D\subseteq B$,    and $D\perp B',$ 
such that
\beq\label{LWLtr1div-1}
&&\|x- (\phi(x)+\psi(x))\|<\ep\rforal x\in {\cal F}\cup \{{a_0}\},\\\label{LWLn-2}
&&d_\tau(\phi({a_0}))<1-\eta\tforal \tau\in {\mathrm{T}}(A),\\\label{LWLn-2+1}
&& \tau'(\phi({a_0})), \tau'(f_{1/2}(\phi({a_0}))\ge r_0-\ep\rforal \tau'\in T(B')\\\label{LWLn-2+2}
&&D\in {\cal C}_0' (\in {{\cal C}_0^{0}}'),\\\label{LWLn-3}
&&\tau(\psi({a_0}))\ge r_0\eta\tforal \tau\in \mathrm{T}(B),\\\label{LWLn-4}
&&t(f_{1/4}(\psi({a_0})))\ge r_0\lambda\tforal t\in T(D).
\eneq

Then $A\in {\cal D}$ (or ${\cal D}_0$).
\end{lem}

\begin{proof}
Let $b_0\in A_+\setminus \{0\}$ with $\|b_0\|=1.$
Choose $k\ge 1$ such that
\beq\label{LWLn-180228}
(1-\eta/2)^k<\inf\{\mathrm{d}_\tau(b_0): \tau\in {\mathrm{T}}(A)\}.
\eneq

Choose a strictly positive element  $a_0\in A$ with $\|a_0\|=1$
such that
$\tau(a_0),\, \tau(f_{1/2}(a_0))\ge 1-1/64$ for all $\tau\in \mathrm{T}(A).$ 
Put $r_0=1-1/64$ and put ${\mathfrak{f}}_a=(r_0/2){\lambda}.$ 
 
Fix $1>\ep>0.$ Put
$\ep_1=\min\{r_0\ep/2(k+1), r_0\eta/4(k+1)\}.$ 
We choose $\dt_1>0$  small enough 
such that
\beq
\|f_{\sigma'}(a')-f_{\sigma'}(b')\|<\ep_1,
\eneq
whenever $\|a'-b'\|<\dt_1$
for any $0\le a', b'\le 1$ in any \CA, where $\sigma'\in \{1/2, 1/4\}.$

Fix a finite subset ${\cal F}\subseteq A^{\bbf 1}.$
Let $\dt_2=\min\{\dt_1/2(k+1), \ep_1/2(k+1)\}.$
Choose some $g\in \mathrm{C}_0((0,1])$ with $0\le g\le 1$ and let $a_1=g(a_0)$ such that
$a_1\ge a_0$ and
\vspace{-0.1in}\beq\label{LWL1-1-}
\|a_1xa_1-x\|<{{\dt_2/64}} \rforal x\in {\cal F}\cup \{a_0\}.
\eneq

Let ${\cal F}_1$ be a finite subset containing ${\cal F}\cup \{a_i,f_{1/4}(a_i),\, f_{1/2}(a_i): i=0,1\}.$

By hypothesis, there are ${\cal F}_1$-$\dt_2/{{64}}$-multiplicative \cpc s $\phi_1': A\to B',$ where $B'$ is a 
hereditary \SCA\, of $A,$ and  $\psi_1: A\to D_1$  for some
\SCA\, $D_1\subseteq A$ such that $D_1\in {\cal C}_0'$ 
(or $\in {\cal C}_0^{0'}),$ $ D_1\perp \phi_1'(A),$ and
\beq\label{LWL-1}
&&\|x-(\phi_1'(x)+\psi_1(x))\|<\dt_2/16\rforal x\in {\cal F}_1,\\\label{LWL-2}
&&\mathrm{d}_\tau({\phi_1'}(a_0))<1-\eta\rforal \tau\in \mathrm{T}(A),\\\label{LWL-2+1}
&& \tau'(\phi_1'(a_0)), \tau'(f_{1/2}(\phi_1'(a_0))\ge r_0-\dt_2/16\rforal \tau'\in {\mathrm{T}}(B')\\\label{LWL-2+2}
&&\tau(\psi_1(a_0)), \tau(f_{1/2}(\psi_1(a_0)))\ge r_0\eta\tforal \tau\in \mathrm{T}(A),\\\label{LWL-4}
&&t(f_{1/4}(\psi_1(a_0)))\ge r_0\lambda\rforal t\in {\mathrm{T}}(D_1).
\eneq
We have, {by \eqref{LWL1-1-},}
\beq\label{LWL-10}
\|\phi_1'(a_1)\phi_1'(x)\phi_1'(a_1)-\phi_1'(x)\|<\dt_2/8\rforal x\in {\cal F}_1.
\eneq
Therefore, for some $\sigma>0,$
\beq\label{LWL-11}
\|f_{\sigma}(\phi_1'(a_1))\phi_1'(x)f_{\sigma}(\phi_1'(a_1))-\phi_1'(x)\|<\dt_2/4\rforal x\in {\cal F}_1.
\eneq

By 7.2 of \cite{eglnp},
there exists $0\le e\le 1$ such that
\beq\label{LWL-12}
f_{\sigma}(\phi_1'(a_1))\le e\le f_{2\sigma'}(\phi_1'(a_1))
\eneq
 and $\mathrm{d}_\tau(e)$ is continuous on $\overline{\mathrm{T}(A)}^\textrm{w},$
 where $0<\sigma' <\sigma/4.$
Define
${\phi_1}: A\to A$ by
\beq\label{LWL-13}
\phi_1(a)=e^{1/2}\phi_1'(a)e^{1/2}\rforal a\in A.
\eneq
We also {{have}}
\beq\label{LWL-14}
e^{1/2}((\phi_1'(a_1)-\sigma'/2)_+)e^{1/2}\le e^{1/2}\phi_1'(a_1)e^{1/2}\le e.
\eneq
But
\beq\label{LWL-14+}
e&=&e^{1/2}f_{\sigma'}(\phi_1'(a_1))e^{1/2}\le   e^{1/2}( (2/\sigma')(\phi_1'(a_1)-\sigma'/2)_+)e^{1/2}\\
&=&(2/\sigma')(e^{1/2}((\phi_1'(a_1)-\sigma'/2)_+)e^{1/2}).
\eneq
Combining these two inequalities, we conclude that
$d_\tau(\phi_1(a_1))=d_\tau(e)$ for all $\tau\in \mathrm{T}(A).$
In particular,
$d_\tau(\phi_1(a_1))$  is continuous on $\mathrm{T}(A).$
By  \eqref{LWL-11}, we have 
\beq\label{LWL-15}
\|\phi_1'(a)-\phi_1(a)\|<\dt_2/4\rforal a\in {\cal F}_1.
\eneq
By the choice of $\dt_1,$ we have 
\beq
\|f_{1/2}(\phi_1(a_0))-f_{1/2}(\phi_1'(a_0))\|<\ep_1.
\eneq
It follows that
\beq\label{560-2001110-n2}
\tau'(f_{1/2}({\phi_1}(a_0)))\ge r_0-\dt_2/16-\ep_1\rforal \tau'\in {\mathrm{T}}(B').
\eneq
Put $B_1=\overline{\phi_1(a_1)A\phi_1(a_1)}.$ 
Then 
by 5.4 of \cite{eglnp}
$B_1$ has continuous scale.  
Note $\phi_1$ maps $A$ into $B_1.$ 
We also have
\beq\label{561-200110-n1}
\|x-(\phi_1(x)+\psi_1(x))\|<\dt_2/2\rforal x\in {\cal F}_1.
\eneq
Moreover, since  $B_1\subset B',$  we have  $B_1\perp D_1,$ and 
\beq
\tau'({\phi_1(a_0)}), \tau'(f_{1/2}({\phi_1(a_0)})\ge r_0-\dt_2/16-\ep_1\rforal \tau'\in {\mathrm{T}}(B_1).
\eneq

{{Since $B_1\in {\cal D}$ (or in ${\cal D}_0$),  we can repeat the process above for $B_1.$}}
{{ Therefore we may now apply }}
the
hypothesis to $B_1$ in place of $A,$  and 
continue the process and stop at stage $k.$

In this way, we obtain hereditary \SCA s $B_1,B_2,...,B_k,$ 
and \SCA s $D_1, D_2,...,D_k$ such that
$B_{i+1}\subseteq B_i,$ $B_i\perp D_i,$ $D_{i+1}\subseteq B_i,$
$D_i\in {\cal C}_0'$ (or ${\cal C}_0^{0'}$), ${\cal F}_i$-$\dt_2/16\cdot 2^{i+1}$-multiplicative
\cpc s $\phi_{i+1}: B_i\to B_{i+1}$ and $\psi_{i+1}: B_i\to D_{i+1}$
such that
$${\cal F}_{i+1}=\{\phi_i(x); x\in {\cal F}_i, c_j,\,f_{1/2}(c_j), f_{1/4}(c_j),\, j=0,1\},$$
where $c_j=\phi_i\circ \phi_{i-1}\circ \cdots \circ \phi_{{1}}(a_j),$ $j=0,1,$ $i=1,2,...,k-1,$ 
\beq\label{LWL-20}
&&\|x-(\phi_{i+1}(x) \oplus \psi_{i+1}(x))\|<\dt_2/2^{i+1}\rforal x\in  {\cal F}_{i+1}\hspace{0.3in} {{\text{(as in \eqref{561-200110-n1}),}}}\\
&&d_\tau({{\phi_{i+1}(c_{i,0})}})<(1-\eta)^{i+1}\rforal \tau\in \mathrm{T}(A), \hspace{0.9in}{{\text{(as in \eqref{LWL-2}, see also \eqref{LWL-13}),}}}\\
&&\tau'({{\phi_{i+1}(c_{i,0}}})),
\tau'(f_{1/2}({{\phi_{i+1}(c_{i,0})}})
\ge (r_0- (i+1)(\dt_2/16+\ep_1))
\\\nonumber
&&\rforal \tau'\in {\mathrm{T}}(B_{i+1}) \hspace{2.5in}{{\text{(as in \eqref{560-2001110-n2}),}}} \\\
&&\tau({{\psi_{i+1}(c_{i,0})}}) \ge  (r_0-(i+1)(\dt_2/16+\ep_1)) \eta\rforal \tau\in \mathrm{T}(B_i)\hspace{0.4in} {{\text{(as in \eqref{LWL-2+2})}}},\\
&&t{{(f_{1/4}(\psi_{i+1}(c_{i,0})))}}
\ge (r_0-(i+1)(\dt_2/16+\ep_1)) \lambda\rforal t\in T(D_{i+1})\hspace{0.4in} {{\text{(as in \eqref{LWL-4}),}}}
\eneq

and $B_{i+1}$ has continuous scale,
$i=1,2,...,k-1.$
Note that $(r_0-k(\dt_2/16+\ep_1))\ge r_0/2.$
Let $D=\bigoplus_{i=1}^kD_i$ and let
$\Psi: A\to D$ be defined by
$$
\Psi(a)=(\psi_1(a)\oplus \psi_2(\phi_1((a))\oplus\cdots\oplus\psi_k(\phi_{k-1}\circ \cdots \circ \phi_1(a)))\rforal a\in A.
$$
By \eqref{LWL-20}, with $\Phi=\phi_k\circ \phi_{k-1}\circ \cdots \phi_1: A\to B_k,$ 
\beq\label{LWLn-28+1}
\|x-(\Phi(x)\oplus  \Psi(x))\|<\ep\rforal x\in {\cal F},\\\label{LWLn-28+2}
t(f_{1/4}(\Psi(a_0)))\ge (r_0/2)\lambda={\mathfrak{f}}_a\tforal t\in {\mathrm{T}}(D),
\eneq
We also have $D\in {\cal C}_0',$ or $D\in {\cal C}_0^{0'}.$

Moreover, 
$$
d_\tau(\Phi(a_0))\le (1-\eta)^k\tforal \tau\in \mathrm{T}(A).
$$
This implies, by \eqref{LWLn-180228},  that
\beq\label{LWLn-28+3}
\Phi(a_0)\lesssim b_0,
\eneq
since $A$ is assumed to have  strict comparison for positive elements.
By  \eqref{LWLn-28+1}, \eqref{LWLn-28+2}, and \eqref{LWLn-28+3},
we conclude that $A$ is in ${\cal D}$ or in ${\cal D}_0.$

\end{proof}

\begin{defn}[10.1 of \cite{eglnp}]\label{DDiv}
Let $A$ be a non-unital and $\sigma$-unital simple \CA.
 $A$ is said to be tracially approximately divisible in the non-unital sense 
if the following property holds:

For any $\ep>0,$ any finite subset ${\cal F}\subseteq A,$ any
$b\in A_+\setminus \{0\},$ and any
integer $n\ge 1,$ there are $\sigma$-unital \SCA s
$A_0, A_1$ of $A$ such that
$$
{ \mathrm{dist}}(x, B_d)<\ep \rforal x\in {\cal F},
$$
where $B_d\subseteq B:=A_0\oplus  \mathrm{M}_n(A_1)\subseteq A,$
$A_0\perp {\mathrm{M}}_n(A_1),$ 
\vspace{-0.1in} \beq\label{Dappdiv-1}
B_d=\{(x_0, \overbrace{x_1,x_1,...,x_1}^n): x_0\in A_0, x_1\in A_1\}
\eneq
and $a_0\lesssim  b,$ where $a_0$ is a strictly positive element of $A_0.$

\end{defn}

\begin{thm}\label{TTTAD}
Let $A$ be a stably projectionless  separable simple \CA\, in the class $\mathcal R$
with $\mathrm{dim}_{\mathrm{nuc}} A=m<\infty$.

Suppose that every hereditary \SCA\, $B$ of $A$ with continuous scale   has the following properties:
Let $e_B\in B$ be a strictly positive element with $\|e_B\|=1$ and $\tau(e_B)>1-1/64$
for all $\tau\in \mathrm{T}(B).$ 
With
$C$
the unique non-unital simple \CA\, 
$C$ 
in ${\cal M}_0\cap\mathcal{R}$ such that $\mathrm{T}(C)\cong\mathrm{T}(B)$, for each affine homeomorphism $\gamma: {\mathrm{T}(B)\to \mathrm{T}(C)},$ 
there 
exist sequences
of \cpc s $\sigma_n: B\to C$ and 
\hm s $\rho_n: C\to B$ 
such that
\beq\label{TTWv-1}
&&\lim_{n\to\infty}\|\sigma_n(ab)-\sigma_n(a)\sigma_n(b)\|=0\tforal a, b\in B,\\\label{TTWv-2}
&&{{\lim_{n\to\infty}\sup\{|t\circ \sigma_n(a)-\gamma^{-1}(t)(a)|: t\in T(C)\}=0\tforal a\in A,}}\\\label{TTWv-3}
&&\lim_{n\to\infty}\sup\{|\tau (\rho_n\circ \sigma_n(b))-{\tau(b)}|:\tau\in \mathrm{T}(B)\}{{=0.}}
\eneq
%
Suppose also that every hereditary \SCA\, $A$ is tracially approximately divisible.
Then $A\in {\cal D}_{0}.$ 
\end{thm}

\begin{proof}
By \cite{T-0-Z} (see also \cite{Winter-Z-stable-02}),
$A'\otimes {\cal Z}\cong A'$ for every hereditary \SCA\, $A'$ of $A.$ It follows from \cite{Rob-0}
that $A$ has almost stable rank one.
Let $B$ be a hereditary  \SCA\, with continuous scale. Then $B$ has finite nuclear dimension (see \cite{WZ-ndim}).
By \cite{T-0-Z} again, $B$ is ${\cal Z}$-stable. 
It follows from 6.6 of \cite{ESR-Cuntz}
that the map from $\mathrm{Cu}(A)$ to ${\mathrm{LAff}}_+(\mathrm{{\tilde T}}(B))$ is surjective.
Note that
{{the map}} from $\mathrm{W}(A)_+$ to ${\mathrm{LAff}}_{b,+}(\mathrm{T}(A))$ is surjective.
We will apply Theorem \ref{fdim} and Lemma \ref{LWL1}.

Fix a strictly positive element $e\in B$ with $\|e\|=1$ and positive element $e_1\in B$ 
such that $f_{1/2}(e)e_1=e_1=f_{1/2}(e)e_1$ with $d_\tau(e_1)> 1-1/64(m+2)$ for all $\tau\in \mathrm{T}(B).$
Let $1>\ep>0,$ ${\cal F}\subseteq B$ be a finite subset.
and let $b\in B_+\setminus \{0\}.$
Choose $b_0\in B_+\setminus \{0\}$ and $64(m+2)\la b_0\ra \le {\la b \ra}$ in 
$\mathrm{Cu}(A).$ 
{{Since we assume that $A$ is tracially approximately divisible (see \eqref{DDiv}),}}
{{there}} are 
$e_0\in B_+$ and a   hereditary \SCA\,
$A_0$ of $A$ such that $e_0\perp M_{4(m+2)}(A_0),$ $e_0\lesssim b_0$ and 
$$
{\mathrm{dist}}(x, B_{1,d})<\ep/64(m+2)\rforal x\in {\cal F}\cup \{e\},
$$
where $B_{1,d}\subseteq  \overline{e_0Be_0}\oplus \mathrm{M}_{4(m+2)}(A_0)\subseteq  B$ and 
\beq\label{TTTAD-1}
B_{1,d}=\{ x_0\oplus (\overbrace{x_1\oplus x_1\oplus \cdots \oplus x_1}^{4(m+2)}):x_0\in \overline{e_0Be_0},\, x_1\in A_0\}.
\eneq
Therefore,  we  may assume, \wilog,  that 
$\|e_0x-xe_0\|<\ep/64(m+2),$ 
and there is $e_1\in M_{4(m+2)}(A_0)$ with $0\le e_1\le 1$ such that
$\|e_1x-xe_1\|<\ep/64(m+2)$ 
for all $x\in {\cal F}\cup \{e, f_{1/2}(e), f_{1/4}(e)\}.$
Moreover,  since  the map from $W(A)_+$ to ${\mathrm{LAff}}_{b,+}(\mathrm{T}(A))$
is surjective, as in the proof of \ref{LWL1} (when 7.2 of \cite{eglnp} is applied), \wilog, we may assume 
that $A_0$ has continuous scale. 

\Wlog, we may further assume that ${\cal F}\cup \{e\}\subseteq B_{1,d}.$
Write 
$$
x=x_0+\overbrace{x_1\oplus  x_1\oplus\cdots\oplus x_1}^{4(m+1)}.
$$
Let ${\cal F}_1=\{x_1: x\in \mathcal F\cup \{e\}\}.$ 
Note that we may write $\overbrace{x_1\oplus  x_1\oplus\cdots\oplus x_1}^{4(m+2)}=x_1\otimes 1_{4(m+2)}.$
Then ${ \mathrm{dim}}_{\mathrm{nuc}}A_0=m$ (see \cite{WZ-ndim}).
Also, $A_0$ is a non-unital separable simple \CA\, 
which has continuous scale.
We may then apply
\ref{fdim} to $A_0$ with ${\cal S}={\cal R}_{\mathrm{az}}.$ 
By \ref{R2},  in \ref{fdim}, we may choose  $C=\overline{\bigcup_{n=1}^{\infty} W_n},$ where each $W_n$ is a finite direct sum of 
${\cal W}$'s, $W_n\subset W_{n+1}$ and strictly positive elements of $W_n$ are strictly positive elements of $W_{n+1}$ for all $n.$
Since (see  9.6 of  \cite{eglnp}) ${\cal W}=\overline{\bigcup_{k=1}^{\infty} R_k},$ where $R_m\subset R_{m+1},$ 
strictly positive elements of $R_m$ are strictly positive elements of $R_{k+1},$ and 
each $R_k$ are Razak algebras (as in \ref{DRazak}), where $\lambda_s(R_k)\to 1,$ as ${k}\to\infty$ (see  for 
$\lambda_s$ in \ref{DfS}, and also \eqref{Razaklambda}),  we may 
write  $C=\overline{\bigcup_{n=1}^{\infty} C_n},$ where $C_n\subset  C_{n+1},$ strictly positive elements of $C_n$ are strictly 
positive elements of $C_{n+1}.$ Moreover, $\lambda_s(C_n)\ge 1/2$ for all $n.$ 
Put $r_0=(1-1/64(m+2)).$
Choose $\eta_0=7/32(m+2)$ and $\lambda=3/16.$
Thus, by applying \ref{fdim}, 
we have, with $\phi_1(b)=(E-p)b(E-p)$ for all $b\in M_{4(m+2)}(A_0),$
where $E=1_{M_{4(m+2)}({\widetilde{A_0}})},$ and $p\in M_{4(m+2)}({\widetilde{A_0}})$
 is a projection given by \ref{fdim},
and $L: A_0\to D_1$ is an ${\cal F}_1$-$\ep$-multiplicative \cpc\, 
\beq
&&\|x-(\phi_1(x)+L(x))\|<\ep/4\rforal x\in {\cal F}_1,\\
&&{d_\tau}(\phi_1(e))\le 1-1/4(m+2)\rforal \tau\in \mathrm{T}(A_0),\\
&&\tau'(\phi_1(e)), \tau'(f_{1/2}(\phi_1(e))\ge r_0-\ep/4\rforal \tau'\in {\mathrm{T}}((1-p)M_{4(m+1)}(A_0)(1-p)),\\\label{751831+}
&&D_1\in {\cal C}_0^0, D_1 {\subseteq} pM_{4(m+2)}(A_0)p,\\
&&\tau(L(e))\ge r_0\eta_0\rforal \tau\in  \mathrm{T}( M_{4(m+2)}(A_0))\andeqn\\\label{571831}
&&t(f_{1/4}(L(e))\ge r_0\lambda\rforal t\in {\mathrm{T}}(D_1).
\eneq
Let  $B_1=(1-p)M_{4(m+1)}(A_0))(1-p)\oplus \overline{e_0Be_0}$ and 
$\phi: {B} \to B_1$ be defined by $\phi(a)=\phi_1(e_1ae_1)+e_0ae_0$ for $a\in A.$  Then $\phi$ is ${\cal F}$-$\ep$-multiplicative.
Put $\eta=\eta_0/2<{\eta_0\over{1+\ep/64(m+2)}}.$ 
Then, in addition to  \eqref{571831} and \eqref{751831+},
\beq\nonumber
&&\|x-(\phi(x)+L(x))\|<\ep\rforal x\in {\cal F},\\\nonumber
&&{d_\tau}(\phi(e))\le 1-\eta\rforal \tau\in \mathrm{T}(B),\\\nonumber
&&\tau'(\phi_1(e)), \tau'(f_{1/2}(\phi_1(e))\ge r-\ep\rforal \tau'\in {\mathrm{T}}(B_1),\\\nonumber
&&\tau(L(e))\ge r_0\eta\rforal \tau\in  \mathrm{T}(B).
\eneq
 Note this holds 
for every such $B.$ 
Thus,  the 
hypotheses
of \ref{LWL1} are satisfied.  We then   apply \ref{LWL1}.
\end{proof}

\section{The \CA\  ${\cal W}$ and UHF-stability}

\begin{defn}[12.1 of \cite{eglnp}]\label{DWtrace}
Let $A$ be a non-unital separable \CA. 
Suppose that
$\tau\in \mathrm{T}(A).$
Recall that $\tau$ was said to be a ${\cal W}$-trace in \cite{eglnp}
if there exists a sequence of
\cpc s $(\phi_n)$ from $A$ into ${\cal W}$ such that
\beq\nonumber
&&\lim_{n\to\infty}\|\phi_n(ab)-\phi_n(a)\phi_n(b)\|=0\rforal a,\,b\in A,\andeqn\\
&&\tau(a)=\lim_{n\to\infty}{ \tau_{\mathcal W}}(\phi_n(a))\rforal a\in A,
\eneq
where ${ \tau_{\mathcal W}}$ is the unique tracial state on ${\cal W}.$

\end{defn}
The following two statements
(\ref{Pwtracest}, and \ref{Ttracekerrho})  are 
taken from \cite{eglnp}
(and the proofs are straightforward). 

\begin{prop}[{12.4} of \cite{eglnp}]\label{Pwtracest}
Let $A$ be a {{separable}} simple \CA\,  with a ${\cal W}$-tracial state $\tau\in \mathrm{T}(A).$
Let $0\le a_0\le 1$ be a strictly positive element of $A.$
Then there exists a sequence of \cpc s  $\phi_n: A\to {\cal W}$ such that
$\phi_n(a_0)$ is a strictly positive element,
\beq\nonumber
&&\lim_{n\to\infty}\|\phi_n(a)\phi_n(b)-\phi_n(ab)\|=0\rforal a,\, b\in A\andeqn\\
&&\tau(a)=\lim_{n\to\infty} { \tau_{\mathcal W}}\circ \phi_n(a)\rforal a\in A,
\eneq
where ${ \tau_{\mathcal W}}$ is the unique tracial state of ${\cal W}.$
\end{prop}


\begin{thm}[{{12.2}} of \cite{eglnp}]\label{Ttracekerrho}
Let $A$ be a separable simple \CA\, 
with
$A=\mathrm{Ped}(A).$
If every tracial state $\tau\in \mathrm{T}(A)$ is a ${\cal W}$-trace, then
$\mathrm{K}_0(A)={\mathrm{ker}}{\rho_A}.$
\end{thm}

\begin{prop}\label{Pwtrace}
Let $A$ be a separable \CA\,  with $A=\mathrm{Ped}(A)$ such that  every tracial state $\tau$ of $A$ is
quasidiagonal.
Let $Y\in {\cal D}_0$ be a  simple \CA\, which is an inductive limit of \CA s in ${\cal C}_0'$
such that  $\mathrm{K}_0(Y)={ \ker}\rho_Y$, and $Y$ has a unique
trace, which is bounded.
Then all  tracial states of $A\otimes Y$  are  ${\cal W}$-tracial states. In particular, all  tracial states of $A\otimes { \cal W}$  are  
${\cal W}$-tracial states.
\end{prop}

\begin{proof}
Let $\tau\in {\mathrm T}(A).$ Denote by $t$ the unique tracial state of $Y.$
We will show $\tau\otimes t$ is a ${\cal W}$-trace on $A\otimes Y.$

{{By  8.12 of \cite{eglnp}, $Y$ is an inductive limit of 1-dimensional non-commutative CW compleces 
(\CA s in ${\cal C}_0$) with $K_1(Y)=\{0\}.$}}
For each $n,$ there is a \hm\, $h_n: \mathrm{M}_n(Y)\to {\cal W}$ (by Theorem 1.0.1 of \cite{Robert-Cu}) such that
$h_n$ maps a strictly positive element of $\mathrm{M}_n(Y)$ to a strictly positive element of ${\cal W}.$
{ Consider $\tau_{\mathcal W}\in \mathrm{T}({\cal W})$.} Then ${\tau_{\mathcal W}}\circ h_n$ is a tracial state of $Y.$
Therefore ${t\otimes\mathrm{tr_n}}(a)={\tau_{\mathcal W}}\circ h_n(a)$ for all $a\in \mathrm{M}_n(Y).$
Moreover,  for any $a\in \mathrm{M}_n$ and $b\in Y,$
$$
{\mathrm{tr}}_n(a)t(b)=\tau_{\mathcal W}\circ h_n(a\otimes b),
$$
where ${\mathrm{tr}}_n$ is the normalized trace on $\mathrm{M}_n,$ $n=1,2,....$

Since $\tau$ is quasidiagonal, there is a sequence
$\psi_n: A\to \mathrm{M}_{k(n)}$ of \cpc s such that
\beq\label{Pwt-1}
&&\lim_{n\to\infty}\|\psi_n(ab)-\psi_n(a)\psi_n(b)\|=0\rforal a, \, b\in A\andeqn\\
&&\tau(a)=\lim_{n\to\infty}\mathrm{tr}_{k(n)}\circ\psi_n(a)\rforal a\in A.
\eneq

Define $\phi_n: A\otimes Y\to {\cal W}$
by $\phi_n(a\otimes b)=h_{k(n)}(\psi_n(a)\otimes b)$ for all $a\in A$ and $b\in Y.$
Then $\phi_n$ is \cpc\, and, for any $a\in A$ and $b\in Y,$
\beq\label{Pwt-3}
(\tau\otimes t)(a\otimes b)&=&\lim_{n\to\infty} \mathrm{tr}_{k(n)}(\psi_n(a))t(b)\\
&=& \lim_{n\to\infty} {\tau_{\mathcal W}}\circ h_{k(n)}(\psi_n(a)\otimes b)=
 \lim_{n\to\infty} {\tau_{\mathcal W}}(\phi_n(a\otimes b)).
\eneq
Therefore $\tau\otimes t$ is a ${\cal W}$-trace.
\end{proof}

\begin{thm}\label{TMW}
Let $A$ be a simple separable \CA\, with finite nuclear dimension  which has {bounded} scale 
and is such that
$\mathrm{K}_0(A)=\ker \rho_A$ and
every tracial state is a ${\cal W}$-trace.
Suppose that  every hereditary \SCA\, of $A$ with continuous scale is tracially approximately divisible. 
Then $A\in {\cal D}_0.$
(In particular, $A\otimes U\in {\cal D}_{0}$ for any UHF-algebra $U.$)
\end{thm}

\begin{proof}
By \cite{T-0-Z}, $A$ is ${\cal Z}$-stable.  By Remark 5.2 of \cite{eglnp}, 
$A$ has a non-zero hereditary \SCA\, $A_0$ with continuous scale. 
Then $\mathrm{M}_k(A_0)$  also has continuous scale for every integer $k\ge 1.$   Since $A$ has bounded scale, it 
is isomorphic to a hereditary \SCA\, of $\mathrm{M}_k(A_0)$ for some possibly large $k.$ 
Since $\mathrm{M}_k(A_0)$ has the same properties as assumed for $A$, it then follows from 8.6 of \cite{eglnp}  that, to prove that $A$ is in ${\cal D}_0,$  we may assume that $A$ has continuous scale.

It follows from Theorem \ref{R2} that there is a simple \CA\, $B=\lim_{n\to\infty} (B_n, \imath_n),$
where each $B_n$ is a finite direct sum of copies of ${\cal W}$  and
$\imath_{n, \infty}$ 
maps strictly positive elements to strictly positive elements,
each
$B_n$ has bounded scale,
 and
$\mathrm{T}(B)\cong \mathrm{T}(A).$
Denote by $\gamma:{{\mathrm{T}(A)\to \mathrm{T}(B)}}$ the affine homeomorphism.
By \cite{Tsang-W},  we may assume that $B=\lim_{n\to\infty}(R_n, \imath_n),$ where each $R_n$ 
is a Razak algebra and $\imath_n$ is injective.
It follows from Theorem \ref{Tappendix}
of
the  appendix that 
there exists a 
\hm\, $\rho: B\to A$  which induces 
$\gamma,$ i.e., 
\beq\label{200115-1}
\tau(\rho(b))=\gamma(\tau)(b)\rforal b\in B {{\andeqn \tau\in T(A).}}
\eneq

Let $(\imath_{n, \infty})_\mathrm{T}: \mathrm{T}(B)\to \mathrm{T}(B_n)$ be  the continuous affine map 
such that, for $t\in \mathrm{T}(B),$
$$
t\circ \imath_{n, \infty}(b)=(\imath_{n, \infty})_T(t)(b)
$$
for all $b\in B_n,$ $n=1,2,....$  
Recall that $\mathrm{T}(\mathcal W)=\{ \tau_{\mathcal W}\},$ where
$\tau_{\mathcal W}$ is the unique tracial state of ${\cal W}.$

Fix a strictly positive element $a_0\in A.$
Fix $\ep>0$ and a finite subset ${\cal F}\subseteq A,$
{{since}} $B=\lim_{n\to\infty} (B_n, \imath_n)$ and $\mathrm{T}(B_n)$ has finitely many extremal traces,  then, as is standard and easy to see,
there 
are
an integer $n_1\ge 1$ and a continuous affine map $\kappa: \mathrm{T}(B_{n_1})\to \mathrm{T}(A)$
such that, for $\tau\in \mathrm{T}(B),$ 
\beq\label{TMW-2}
\sup_{\tau\in  {\mathrm{T}}(B)}|\kappa\circ (\imath_{n_1, \infty})_{\mathrm T}(\tau)(f)-{\gamma^{-1}}(\tau)(f)|<{{\ep/3}}\rforal f\in {\cal F}.
\eneq
Write $B_{n_1}=W_1\oplus W_2\oplus \cdots \oplus W_{m},$
where each $W_i\cong {\cal W}.$
Denote by $\tau_{\mathrm{W}_1}, \tau_{\mathrm{W}_2},...,\tau_{\mathrm{W}_m}$ the unique tracial states on $W_i,$
and $\theta_i=\kappa(\tau_{\mathrm{W}_i}),$ $i=1,2,...,m.$
By the assumption, there exists, for each $i,$ a sequence  of \cpc\,
$\phi_{n,i}: A\to W_i$ such that
\beq\label{TMW-3}
&&\lim_{n\to\infty}\|\phi_{n,i}(a)\phi_{n,i}(b)-\phi_{n,i}(ab)\|=0\rforal a,b\in A, {\mathrm{and}}\\
&&\theta_i(a)=\lim_{n\to\infty}\tau_{\mathrm{W}_i}\circ \phi_{n,i}(a)\rforal a\in A.
\eneq
Moreover, by \ref{Pwtracest}, we may assume that $\phi_{n,i}(a_0)$ is strictly positive.
Define $\phi_n: A\to B_{n_1}$ by
\beq\label{TMW-4}
\phi_n(a)= \phi_{n,1}(a)\oplus \phi_{n,2}(a)\oplus\cdots { \oplus}\phi_{n,m}(a),\quad a\in A.
\eneq
Then
\vspace{-0.1in}\beq\label{TMW-4+}
\lim_{n\to\infty}\sup_{\tau\in {\mathrm{T}}(B_{n_1})}\{|\tau(\phi_n(a))-\kappa(\tau)(a)|\}=0
\rforal a\in A.
\eneq
Define ${{\sigma_n}}: A\to B$ by
\vspace{-0.12in}\beq\label{TMW-5}
{{\sigma}}_n(a)=\imath_{n_1,\infty}\circ \phi_n(a),\quad a\in A.
\eneq
Note that ${{\sigma_n}}(a_0)$ is a strictly positive element.
We also have  that
\beq\label{TMW-6}
\lim_{n\to\infty}\|{{\sigma_n}}(ab)-{{\sigma_n}}(a){{\sigma_n}}(b)\|=0,\quad  a,\, b\in A.
\eneq
Moreover,  for any $\tau\in \mathrm{T}(B)$ and any $f\in {\cal F},$
\beq\label{TMW-7}
|{{\gamma^{-1}}}(\tau)(f)-\tau\circ {{\sigma_n}}(f)| &\le & |{{\gamma^{-1}}}(\tau)(f)-\kappa\circ (\imath_{n_1, \infty})_{{T}}(\tau)(f)|\\
&&+|\kappa\circ (\imath_{n_1, \infty})_\mathrm{T}(\tau)(f)-\tau\circ {{\sigma_n}}(f)|\\
&<&\ep/3+|\kappa\circ (\imath_{n_1, \infty})_{\mathrm{T}}(\tau)(f)-\tau\circ \imath_{n_1, \infty}\circ \phi_n(f)|\\
&\le & \ep/3+\sup_{t\in \mathrm{T}(B_{n_1})}\{|\tau(\phi_n(f))-\kappa(\tau)(f)|\}.
\eneq
By \eqref{TMW-4+}, there exists $N\ge 1$ such that, for all $n\ge N,$
\beq\label{TMW-8}
\sup_{\tau\in \mathrm{T}(B)}\{|{{\gamma^{-1}}}(\tau)(f)-\tau\circ {{\sigma_n}}(f)|\}<{{2\ep/3 \rforal}} f\in {\cal F}.
\eneq
{{Thus the map
$\sigma_n$ satisfies  \eqref{TTWv-1}
and \eqref{TTWv-2}.}}
{{By \eqref{200115-1} and  \eqref{TMW-8},  for all $n\ge N,$}} 
$$
{{\sup_{\tau\in \mathrm{T}(A)}\{|\tau(f)-\tau(\rho\circ \sigma_n(f))|\}
=\sup_{\tau\in \mathrm{T}(A)}\{|\tau(f)-\gamma(\tau)(\sigma_n(f))|\}<\ep\rforal f\in {\cal F}.}}
$$
{{Thus \eqref{TTWv-3} also holds (with $\rho=\rho_n$).}}  Therefore,
by \ref{TTTAD}, $A\in\mathcal D_0$.
\end{proof}

\begin{thm}\label{TWtrace}
Let $A$ be a non-unital separable simple  \CA\, with finite nuclear dimension and with $A=\mathrm{Ped}(A).$
Suppose that $\mathrm{T}(A)\not={\O}$.
Then 
$A\otimes {\cal W}\in {\cal D}_{0}.$
\end{thm}

\begin{proof}
By Lemma \ref{KK-Kun}, $A\otimes {\cal W}$ is KK-contractible. Therefore $A\otimes {\cal W}$ satisfies the UCT. Since $\mathcal W$ has finite nuclear dimension, so also does $A\otimes\mathcal W$. Hence by \cite{TWW-QD}, every tracial state is quasi-diagonal.
It follows by \ref{Pwtrace} that every tracial state of $A\otimes {\cal W}$ is a ${\cal W}$-trace.  
We also have $\mathrm{K}_0(A\otimes {\cal W})=\{0\}.$ 
Let $b\in (A\otimes {\cal W})_+.$ Since ${\cal W}$ has a unique tracial state, by 11.8 of \cite{eglnp}, 
$W(A\otimes {\cal W})=\mathrm{LAff}_{b,0+}(\overline{\mathrm{T}(A)}^w).$ Therefore, 
there 
are
$a\in A$ and $b_1\in  M_2({\cal W})_+$ such that $b\sim a\otimes b_1.$
Put $B=\overline{b(A\otimes {\cal W})b}$ and $B_1=\overline{(a\otimes b_1)(A\otimes {\cal W})(a\otimes b_1)}.$
Then $B\cong B_1.$ Note that $\overline{b_1{\cal W}b_1}\cong {\cal W}.$ It follows that 
$B_1\cong \overline{aAa}\otimes {\cal W}.$ But ${\cal W}\otimes Q\cong {\cal W}.$ This implies that $B_1$ 
is tracially approximately divisible.  Therefore $B$ is tracially approximately divisible.
 Then \ref{TMW} applies.
\end{proof}

{{Added in proof:   The condition of finite nuclear dimension in \ref{TWtrace} can be much weakened 
to the condition that $A$ is amenable. Since $A\otimes {\cal W}$ is ${\cal Z}$-stable,
by a recent preprint of J. Castillejos and S. Evington, arXiv:1901.11441, it has finite nuclear dimension, 
as kindly pointed by the referee.}}

\begin{cor}\label{cor-WW}
Let $A$ be a simple separable finite \CA\, such that $A\otimes\mathcal Z$ has finite nuclear dimension. Then the \CA\, $A\otimes \mathcal W$ belongs to the class ${\cal M}_0,$ 
and so 
$A\otimes {\cal W}$ is isomorphic to an inductive limit of \CA s in ${\cal R}_{\mathrm{az}}.$
in particular, $\mathcal W\otimes \mathcal W \cong \mathcal W$.
\end{cor}

\begin{proof}
By \ref{TWtrace}, $A\otimes {\cal W}\in {\cal D}_0$ and, by \ref{KK-Kun}, $A\otimes {\cal W}$ is KK-contractible. 
Then \ref{TTMW} applies. 
\end{proof}

\begin{lem}\label{untwrist}
Let $A$ be a  separable simple \CA\, in ${\cal R}$ 
with finite nuclear dimension which is $\mathrm{KK}$-contractible
 and assume that all tracial states of $A$ 
are ${\cal W}$-traces.
Let $\mathrm{M}_\p$ and $\mathrm{M}_\q$ be two UHF algebras, where 
$\p$ and $\q$ are relatively prime supernatural numbers. Then, there exist an isomorphism $\phi: A\otimes \mathrm{M}_\p \to A \otimes \mathrm{M}_\p \otimes\mathrm{M}_\q$ and a continuous path of unitaries $u_t\in \mathrm{M}(A\otimes\mathrm{M}_\p\otimes \mathrm{M}_\q)$, $1\leq t < \infty$, such that $u_1=1$ and
$$ \lim_{t\to\infty} u^*_t(a \otimes r\otimes 1_\q)u_t = \phi(a\otimes r),\quad a\in A,\ r\in \mathrm{M}_\p.$$
\end{lem}

\begin{proof}
Note  that every hereditary \SCA\, $B$ of $A\otimes \mathrm{M}_\p$ and $A\otimes Q$ is tracial approximately divisible,
since $M_\p$ and $Q$ are strongly self-absorbing.
By the assumption and \ref{TMW}, $A\otimes \mathrm{M}_\p$ and $A\otimes Q$ are in ${\cal D}_0.$
It follows from \ref{TTMW} 
that $A\otimes \mathrm{M}_\p\cong A\otimes Q$.
Let $\xi: A\otimes \mathrm{M}_\p\to A\otimes Q$
be an isomorphism.
It is well known that any isomorphism $\psi: Q\to Q\otimes {{\mathrm{M}}}_\q$ is asymptotically 
unitarily equivalent to the embedding 
 $Q\to Q\otimes \mathrm{M}_\q$ given by  $r\to r\otimes 1_\q$  for all 
$r\in Q$ (see, for instance, \cite{Lncbms}). Therefore 
there exists a continuous path of unitaries $v_t\in Q\otimes \mathrm{M}_\q$ such that  $v_1=1$
and
$$ \lim_{t\to\infty} v^*_t(r\otimes {{1_\q}})v_t = \psi(r)\rforal  r\in Q.$$
{{Define $\phi_1: A\otimes Q\to A\otimes Q\otimes M_\q$ by  $\phi_1(a\otimes r)=a\otimes \psi(r)$
for all $a\in A$ and $r\in Q.$}}
Therefore
\beq\label{819f-1}
 \lim_{t\to\infty} (1_A\otimes v^*_t)(b\otimes 1_\q)(1_A\otimes v_t) = {{\phi_1(b) \rforal}} b\in A\otimes Q.
\eneq
{{Define}} ${{\phi}}: A\otimes \mathrm{M}_\p \to A\otimes \mathrm{M}_\p\otimes \mathrm{M}_\q$ by ${{\phi}}=(\xi^{-1}\otimes {\mathrm{id}}_{M_\q})\circ{{\phi_1}}\circ\xi$
and let ${{u_t={\tilde \xi}^{-1}(1_A\otimes v_t)}}$, where ${\tilde \xi}:\mathrm{M}(A\otimes\mathrm{M}_\p\otimes \mathrm{M}_\q)\to\mathrm{M}(A\otimes Q\otimes \mathrm{M}_\q)$ is the extension of $\xi\otimes {\mathrm{id}}_{M_\q}: A\otimes\mathrm{M}_\p\otimes \mathrm{M}_\q\to A\otimes Q\otimes \mathrm{M}_\q$. Note that $u_1=1$, since $v_1=1$ and 
{{$\{u_t\}$}} is a continuous path of unitaries in ${\mathrm M}(A\otimes {{\mathrm{M}_\p}}\otimes \mathrm{M}_\q).$
Suppose that $a\in A$ and {{$r\in
\mathrm{M}_\p.$}}  {{So}} $\xi(a\otimes r)\in A\otimes Q.$
Then we have 
\beq\nonumber
&&\lim_{t \to \infty}{{u}}^*_t(a\otimes r\otimes 1_\q) {{u_t}}=\lim_{t \to \infty} {\tilde \xi}^{-1}(1_A\otimes v^*_t) \left({{(\xi^{-1}\otimes{\mathrm{id}}_{{\mathrm M}_\q})}}(\xi(a\otimes r)\otimes 1_\q)\right){\tilde \xi}^{-1}(1_A\otimes v_t)\\\nonumber
&&=(\xi^{-1}\otimes {\mathrm{id}}_{\mathrm{M}_\q})(\lim_{t\to\infty}\left((1_A\otimes u^*_t)(\xi(a\otimes r)\otimes 1_\q)(1_A\otimes u_t)\right))
\\\nonumber
&&{{=^{\text{see}\, \eqref{819f-1}}}} \xi^{-1}\otimes {\mathrm{id}}_{\mathrm{M}_\q}({{\phi_1}}(\xi(a\otimes r)))
={{\phi}}(a \otimes r)
\eneq
as desired.
\end{proof}

\begin{thm}\label{Z-stable}
Let $A$ be a non-unital separable simple \CA\,  in ${\cal R}$ 
with finite nuclear dimension  which is $\mathrm{KK}$-contractible
and such that every trace is a ${\cal W}$-trace. Then $A \cong A\otimes Q$.
\end{thm}

\begin{proof}
It follows from \cite{T-0-Z} that $A\cong A\otimes {\cal Z}.$ 
Decompose $A\otimes \mathcal Z$ as an inductive limit of copies of $A\otimes\mathcal Z_{\p, \q}$, where $\p, \q$ are two relatively prime supernatural numbers such that $\mathrm{M}_\p\otimes\mathrm{M}_\q = Q$. By Corollary 3.4 of \cite{TW-D}, in order to show that $A$ is $Q$-stable, it is enough to show that $A\otimes\mathcal Z_{\p, \q}$ is $Q$-stable. Note that
$$A\otimes\mathcal Z_{\p, \q} =\{f\in\mathrm{C}([0, 1],  A\otimes\mathrm{M}_\p\otimes\mathrm{M}_\q): f(0)\in A\otimes \mathrm{M}_\p\otimes 1_\q,\ f(1)\in A\otimes 1_\p\otimes \mathrm{M}_\q \}.$$ Applying Lemma \ref{untwrist} to both endpoints, one obtains isomorphisms $$\phi_0: A\otimes \mathrm{M}_\p \to A \otimes \mathrm{M}_\p \otimes\mathrm{M}_\q,\quad \phi_1: A\otimes \mathrm{M}_\q \to A \otimes \mathrm{M}_\p \otimes\mathrm{M}_\q,$$
together with a continuous path of unitaries $u_t\in \mathcal M(A\otimes\mathrm{M}_\p\otimes \mathrm{M}_\q)$, $0 < t < 1$, such that $u_{\frac{1}{2}}=1$, 
$$ \lim_{t\to 0} u^*_t(a \otimes r\otimes 1_\q)u_t = \phi_0(a\otimes r),\quad a\in A,\ r\in \mathrm{M}_\p,$$
and
$$ \lim_{t\to 1} u^*_t(a \otimes 1_\p \otimes r)u_t = \phi_1(a\otimes r),\quad a\in A,\ r\in \mathrm{M}_\q.$$

Define the continuous field map
$\Phi: A \otimes \mathcal Z_{\p, \q} \to \mathrm{C}([0, 1], A\otimes \mathrm{M}_\p\otimes\mathrm{M}_\q)$ by
$$\Phi(f)(t) = u_t^*f(t)u_t,\quad t\in[0, 1],$$
where $\Phi(f)(0)$ and $\Phi(f)(1)$ are understood as $\phi_0(f(0))$ and $\phi_1(f(1))$, respectively. Then the map $\Phi$ is an isomorphism (the inverse  is $\Phi^{-1}(g)(t) = u_tg(t)u^*_t$, $t\in(0, 1)$, $\Phi^{-1}(g)(0)=\phi_0^{-1}(g(0))$, and $\Phi^{-1}(g)(1)=\phi_0^{-1}(g(1))$), and hence $A\otimes\mathcal Z_{\p, \q} \cong \mathrm{C}([0, 1], A\otimes \mathrm{M}_\p\otimes\mathrm{M}_\q).$ Since the trivial field $\mathrm{C}([0, 1], A\otimes \mathrm{M}_\p\otimes\mathrm{M}_\q)$ is $Q$-stable, one has that $A\otimes\mathcal Z_{\p, \q}$ is $Q$-stable, as desired.
\end{proof}

\section{The 
case of finite nuclear dimension}

Let $A$ be a non-unital separable \CA.
Since ${\widetilde A}\otimes Q$ is unital, we may view ${\widetilde{A\otimes Q}}$ as
a \SCA\, of ${\widetilde A}\otimes Q$ with the unit $1_{{\widetilde A}\otimes Q}.$
In the following corollary we use $\imath$ for the embedding
from $A\otimes Q$ to ${\widetilde A}\otimes Q$ as well as from
${\widetilde{A\otimes Q}}$ to ${\widetilde A}\otimes Q.$
Since $\mathrm{K}_1(Q)=\{0\},$  from the six-term exact sequence in K-theory, one concludes that
 the \hm\,
$\imath_{*0}: \mathrm{K}_0(A\otimes Q)\to \mathrm{K}_0({\widetilde A}\otimes Q)$  is injective.

We will use this fact and identify $x$ with $\imath_{*0}(x)$ for all $x\in \mathrm{K}_0(A\otimes Q)$ in the following corollary.

\begin{lem}\label{LfactorQ}
Let $A$ be a non-unital separable \CA\, and let $(\psi_n)$ be a sequence of approximately multiplicative  \cpc  s
from ${\widetilde A}\otimes Q$ to $Q.$ Then $\phi_n=\psi_n\circ \imath$ is a  sequence of approximately multiplicative  completely positive contractive maps from $A$ into $Q,$ where ${{\imath_0:}} A\to {\widetilde A}\otimes Q$
is the embedding defined by $a\mapsto a\otimes 1$ for all $a\in A.$

Conversely,  if $(\phi_n)$ is  a sequence of approximately multiplicative  completely positive contractive maps
from $A$ to $Q,$
then, there exists a sequence of approximately multiplicative completely positive contractive maps
$(\psi_n): {\widetilde A}\otimes Q\to Q$
such that
$$
\lim_{n\to\infty}\|\phi_n(a)-\psi_n\circ {{\imath_0}}(a)\|=0\rforal a\in A.
$$

Moreover, if $\lim\sup \|\phi_n(a)\|\not=0$ for some $a\in A$  and if $\{e_n\}$ is an approximate
unit, then, we can choose $\psi_n$ such that
$$
\mathrm{tr}(\psi_n(1))=\mathrm{d}_{\mathrm{tr}}(\phi_n(e_n)))\tforal n.
$$
\end{lem}

\begin{proof}
We prove only the second part.
Write $Q=\overline{\bigcup_{n=1}^{\infty} \mathrm{M}_{n!}}$ with the embedding
$j_n: B_n:=\mathrm{M}_{n!}\to \mathrm{M}_{n!}\otimes \mathrm{M}_{n+1}=\mathrm{M}_{(n+1)!},$ $n=1,2,....$
\Wlog, we may assume
that $\phi_n$ maps $A$ into $B_n,$  $n=1,2,....$
Consider $\phi_n'(a)=\phi_n(e_n^{1/2}ae_n^{1/2}),$ $n=1,2,....$
Choose $p_n$ to be  the range projection of $\phi_n(e_n)$ in $B_n.$
Define $\psi_n': {\widetilde A}\otimes Q\to Q\otimes Q$ by ${\psi'_n}(a\otimes 1_Q)=\phi_n'(a)\otimes 1_{Q}$
for all $a\in A,$ $\psi'_n(1\otimes r)={p_n}\otimes r$ for all $r\in Q.$
Then
$$
\lim_{n\to\infty}\|\psi_n'(a\otimes 1)-\phi_n(a)\otimes 1\|=0\rforal a\in A.
$$
Moreover, $\mathrm{tr}(\psi_n'(1))=\mathrm{d}_{\mathrm{tr}}(\phi_n(e_n))$ for all $n.$
There is an isomorphism $h: Q\otimes Q\to Q$ such that
$h\circ {{\imath_Q}}$  is approximately unitarily equivalent to ${\text{id}}_Q,$
{{where $\imath_Q: a\mapsto a\otimes 1_Q$ is the embedding.}}
By choosing  some unitaries $u_n\in Q,$ we can choose $\psi={\text{Ad}}\, u_n\circ h\circ {\psi_n},$
$n=1,2,....$

\end{proof}

\vspace{-0.1in}The following is a non-unital version of Lemma 4.2 of \cite{EGLN-DR}.

\begin{lem}\label{Lpartuniq}
Let $A$ be a non-unital simple separable amenable \CA\,  with
$\mathrm{T}(A)\not=\O$
which has bounded scale and
which satisfies the UCT.
Fix a strictly positive element $a\in A_+$ with $\|a\|=1$
such that
\beq\label{Lpartuniq-n1}
\tau(f_{1/2}(a))\ge d\rforal \tau\in \overline{\mathrm{T}(A)}^{\mathrm w}.
\eneq

For any $\ep>0$ and any finite subset ${\mathcal F}$ of $A$, there exist
$\dt>0,$   a finite subset ${\mathcal G}$ of $A$,
and a finite subset ${\mathcal P}$ of $K_0(A)$ with the following property.
\noindent
Let $\psi,\phi: A\to Q$ be two ${\mathcal G}$-$\dt$-multiplicative completely positive contractive maps such that
\beq\label{puniq-1}
&&[\psi]|_{\mathcal P}=[\phi]|_{\mathcal P}\tand\\\label{puniq-1n}
&&\mathrm{tr}(f_{1/2}({\psi}(a)))\ge d/2 \tand \mathrm{tr}(f_{1/2}({\phi}(a)))\ge d/2,
\eneq
where $\mathrm{tr}$ is the unique tracial state of $Q.$
Then there is a unitary $u\in Q$ and an  ${\mathcal F}$-$\ep$-multiplicative completely positive contractive {{map}} $L: A\to {\mathrm{C}}([0,1], Q)$
such that
\beq\label{puni-2}
&& \pi_0\circ L=\psi,\,\,\, \pi_1\circ L=\mathrm{Ad}u\circ \phi.
\eneq
Moreover, if
\begin{equation}\label{puniq-3}
 |\mathrm{tr}\circ \psi(h)-\mathrm{tr}\circ \phi(h)|<\ep'/2\tforal  h\in {\mathcal H},
\end{equation}
for a finite set ${\mathcal H}\subseteq A$ and $\ep'>0$,
then $L$ may be chosen such that
\begin{equation}\label{puniq-4}
 |\mathrm{tr}\circ \pi_t\circ L(h)-\mathrm{tr}\circ \pi_0\circ L(h)|<\ep'\rforal  h\in {\mathcal H}\andeqn t\in [0, 1].
\end{equation}
Here, $\pi_t: {\mathrm{C}}([0,1], Q)\to Q$ is the point evaluation at $t\in [0,1].$
\end{lem}

\begin{proof}
 Let $T: A_+\setminus \{0\}\to \N\times \R_+\setminus \{0\}$ be given by  5.7 of \cite{eglnp}
 (with above $d$ and $a$).
 In the notation in \ref{Blbm}, $Q\in {\boldsymbol{C}}_{0,0, t, 1, 2},$ where $t: \N\times \N\to \N$ is defined
 to be $t(n,k)=n/k$ for all $n, k\ge 1.$  Now ${\boldsymbol{ C}}_{0,0, t, 1,2}$ is fixed.
 We are going to apply Theorem \ref{Lauct2} together with the Remark \ref{RRLuniq} (note
 that $Q$ has real rank zero and $\mathrm{K}_1(Q)=\{0\}$).

Let ${\mathcal F}\subseteq A$ be a finite subset and let $\ep>0$ be given.   We may assume that $a\in {\mathcal F}$ and every element of ${\mathcal F}$ has norm at most one.
Write  ${\mathcal F}_1=\{ab: a, b\in {\mathcal F}\}\cup {\cal F}$.

Let  $\delta_1>0$ (in place $\dt$), ${\mathcal G}_1$ (in place
of ${\cal G}$) and  ${\cal H}_1$(in place of ${\cal H}$), ${\mathcal P},$ and $K$ be as assured by
Theorem \ref{Lauct2} for ${\mathcal F}_1$ and $\ep/4$ as well as $T$ (in place of $F$). (As stated earlier we will also use the Remark \ref{RRLuniq} so that we drop ${\boldsymbol{ L}}$  and
 condition \eqref{Lauct-1}.)
Since $\mathrm{K}_1(Q)=\{0\}$ and
$\mathrm{K}_0(Q)=\Q,$ we may choose ${\cal P}\subseteq \mathrm{K}_0(A).$

We may also assume that $\mathcal F_1\cup {\cal H}_1\subseteq \mathcal G_1$ and
$K\ge 2.$

Now, let ${\cal G}_2\subseteq A$ (in place of ${\cal G}$) be a finite subset and
let $\dt_2>0$ (in place of $\dt_1$) given by 5.7 of \cite{eglnp}
for the above ${\cal H}_1$ and
$T.$

Let $\dt=\min\{\ep/4, \dt_1/2, \dt_2/2\}$ and ${\cal G}={\cal G}_1\cup {\cal G}_2.$
\Wlog, we may assume that ${\cal G}\subseteq A^{\boldsymbol{ 1}}.$

Since $Q\cong Q\otimes Q,$ we may assume, \wilog,
that $\phi(a), \psi(a)\in Q\otimes 1$ for all $a\in A.$
Pick mutually equivalent projections
$e_0, e_1,e_2,...,e_{2K}\in Q$ satisfying $\sum_{i=0}^{2K} e_i=1_Q.$
Then, consider the maps $\phi_i, \psi_i: A\to Q\otimes e_iQe_i$, $i=0, 1, ..., 2K$, which are defined by
$$\phi_i(a)=\phi(a)\otimes e_i\quad\mathrm{and}\quad \psi_i(a)=\psi(a)\otimes e_i,\quad a\in A,$$ 
and consider the maps
$$\Phi_{K+1}:=\phi=\phi_0\oplus \phi_1\oplus\cdots \oplus \phi_{2K},\quad \Phi_0:=\psi=\psi_0\oplus \psi_1\oplus \cdots \oplus \psi_{2K}$$ and
\vspace{-0.12in}$$\Phi_i:=\phi_0\oplus\cdots  \oplus \phi_{i-1}\oplus \psi_{i}\oplus \cdots \oplus \psi_{2K},\quad i=1,2,...,2K.$$
Since $e_i$ is unitarily equivalent to $e_{{0}}$ for all $i$, one has
$$[\phi_i]|_{\mathcal P}=[\psi_j]|_{\mathcal P},\quad 0\leq i, j\leq 2K.$$
and in particular, 
\vspace{-0.1in}\beq\label{puniq-10+}
[\phi_{i}]|_{\mathcal P}=[\psi_{i}]|_{\mathcal P},\quad i=0, 1, .., 2K.
\eneq

Note that, for each $i=0, 1, ..., n$, $\Phi_i$ is unitarily equivalent to
$$
\psi_{i}\oplus (\phi_0\oplus \phi_1\oplus \cdots \oplus \phi_{i-1}\oplus  \psi_{i+1}\oplus  \psi_{i+2}\oplus \cdots \oplus \psi_{2K}),
$$
and $\Phi_{i+1}$ is unitarily equivalent to
$$
 \phi_{i}\oplus (\phi_0\oplus \phi_1\oplus \cdots  \oplus \phi_{i-1}\oplus \psi_{i+1}\oplus  \psi_{i+2}\oplus \cdots \oplus \psi_{2K}).
$$

Using \eqref{puniq-1n},
on applying  5.7 of \cite{eglnp},
we obtain
that 
{{maps
$\phi_i$ and $\psi_i$ are}}
$T$-${\cal H}_1$-full in $e_iQe_i,$ $i=0,1,2,...,2K.$

In view of this, and \eqref{puniq-10+}, applying Theorem \ref{Lauct2} (and its remarks),
 we obtain unitaries $u_i\in Q$, $i=0, 1, ..., 2K$,  such that
\begin{equation}\label{puniq-11}
\|{\tilde \Phi}_{i+1}(a)-{\tilde \Phi}_{i}(a)\|<\ep/4,\quad a\in {\mathcal F}_1,\,\,\,{{\mathrm{where}}}
\end{equation}
\vspace{-0.12in}$${\tilde \Phi}_0:=\Phi_0=\psi \quad\textrm{and}\quad {\tilde \Phi}_{i+1}:={\text{{Ad}}}\,  u_{i}\circ \cdots \circ {\text{{Ad}}}\ u_1 \circ   {\text{{Ad}}}\ u_0\circ \Phi_{i+1},\quad i=0, 1,...,2K.$$
Put $t_i=i/(2K+1)$, $i=0, 1, ..., 2K+1$,
and define $L: A\to {\text{ C}}([0,1], Q)$ by
$$
\pi_t\circ L=(2K+1)(t_{i+1}-t){\tilde \Phi_i}+(2K+1)(t-t_i){\tilde \Phi_{i+1}},\quad t\in [t_i, t_{i+1}],\ i=0,1,...,2K.
$$
By construction,
\begin{equation}\label{eq-verification}
\pi_0\circ L={\tilde \Phi}_0=\psi\quad\mathrm{and}\quad \pi_1\circ L={\tilde \Phi}_{n+1}={\text{{Ad}}}\, u_n\circ\cdots \circ{ \text{Ad}}\ u_1 \circ{\text{Ad}}\ u_0 \circ \phi.
\end{equation}
Since $\tilde{\Phi}_i$, $i=0, 1, ..., 2K$, are $\mathcal G$-$\delta$-multiplicative (in particular $\mathcal F$-$\ep/4$-multiplicative), it follows from \eqref{puniq-11} that $L$ is ${\mathcal F}$-$\ep$-multiplicative. By \eqref{eq-verification}, $L$ satisfies \eqref{puni-2} with $u=u_{2K}\cdots u_1 u_0$.

Moreover, if there is a finite set $\mathcal H$ such that \eqref{puniq-3} holds, it is then also straightforward to verify that $L$ satisfies \eqref{puniq-4}, as desired.
\end{proof}

\begin{rem}\label{Ruct3}
If $A$ is 
 KK-contractible,
 then the assumption
that $A$ satisfies the UCT can of course be dropped.
\end{rem}

\begin{thm}\label{TalWtrace}
Let $A$ be a non-unital simple separable amenable \CA\, with $\mathrm{K}_0(A)={\mathrm{Tor}}(\mathrm{K}_0(A))$ which satisfies the UCT.  Suppose that $A=\mathrm{Ped}(A).$
Then every trace in $\overline{\mathrm{T}(A)}^{\mathrm{w}}$  is  a ${\cal W}$-trace.
\end{thm}

\begin{proof}
It suffices to show that every tracial state of $A$ is a ${\cal W}$-trace.
It follows from \cite{TWW-QD} that every trace is quasidiagonal.
For a fixed  $\tau\in \mathrm{T}(A),$  there exists a sequence of approximately multiplicative \cpc s
$(\phi_n)$ from $A$ into $Q$ such that
$$
\lim_{n\to\infty} \mathrm{tr}\circ \phi_n(a)=\tau(a)\rforal a\in A.
$$
By Lemma \ref{LfactorQ}, we may assume that $\phi_n=\psi_n\circ \imath,$
where $\imath: A\to A\otimes Q$  is the embedding defined by
$\imath(a)=a\otimes 1_Q$ for all $a\in A$ and $\psi_n: A\otimes Q\to Q$ is a sequence
of approximate multiplicative \cpc s.

Therefore it suffices to show that every tracial state of $A\otimes Q$ is a ${\cal W}$-trace.
Set $A_1=A\otimes Q.$  Then  $\mathrm{K}_0(A_1)=\{0\}.$

Fix $1>\ep>0,$ $1>\ep'>0,$  a finite subset ${\cal F}\subseteq A_1$ and a finite ${\cal H}\subseteq A_1.$
Put ${\cal F}_1={\cal F}\cup {\cal H}.$
\Wlog, we may assume that ${\cal F}_1\subseteq A_1^{\boldsymbol{ 1}}.$
Note that $A$ is non-unital. Choose  a strictly positive element $a\in A_+$ with $\|a\|=1.$
We may also assume  that 
$$
\tau(f_{1/2}(a))\ge d>0\rforal \tau\in \overline{\mathrm{T}(A)}^{\mathrm w}
\,\,\,{{{\mathrm{(for\,\,\, some}}\,\,\, d>0).}}
$$

Let $1>\dt>0,$ ${\cal G}\subseteq A_1$ be a finite subset as 
provided
by
\ref{Lpartuniq} for $A_1$ (in place of $A$), $d/2$ (in place of $d$), $\ep/16$ (in place of $\ep$),
and ${\cal F}_1.$ (Note since $\mathrm{K}_0(A_1)=\{0\},$ the required set ${\cal P}$  in \ref{Lpartuniq}  does not  appear here.)

Let  ${\cal G}_1={\cal G}\cup {\cal F}_1$ and let
$
\ep_1={\ep\cdot \ep'\cdot \dt/{2}}.
$
Let $\tau\in {\mathrm{T}}(A_1).$  Since $\tau$ is quasidiagonal,
 there exists a ${\cal G}_1$-$\ep_1$-multiplicative \cpc\,
${\psi}: A_1\to Q$
such that
\beq\label{Talwtr-11}
&&|\tau(b)-\mathrm{tr}\circ \psi(b)|<\ep'/16 {{\rforal}} b\in {\cal G}\cup {\cal F},\ \mathrm{and}\\
&&\mathrm{tr}(f_{1/2}(\psi(a)))>2d/3.
\eneq
Choose an integer $m\ge 3$ such that
$$
1/m<\min\{\ep_1/64, d/8\}.
$$

Let $e_1, e_2,...,e_{m+1}\in Q$ be a set of mutually orthogonal and mutually equivalent  projections
such that
$$
\sum_{i=1}^{m+1}e_i=1_Q\quad\mathrm{and}\quad \mathrm{tr}(e_i)={1\over{m+1}},\,\,\,i=1,2,...,m+1.
$$
Let $\psi_i: A_1\to (1\otimes e_i)(Q\otimes Q)(1\otimes e_i)$ be defined by
$\psi_i(b)=\psi(b)\otimes e_i,$ $i=1,2,...,m+1.$
Set
$$\sum_{i=1}^m\psi_i = \Psi_0 \quad\mathrm{and}\quad\sum_{i=1}^{m+1}\psi_i=\Psi_1.$$
Identify $Q\otimes Q$ with $Q.$
Note that
\beq\label{Talwtr-12}
\mathrm{tr}(f_{1/2}(\Psi_i(a)))\ge d/2,\,\,\,i=0,1.
\eneq
Moreover,
$$
|\tau\circ {\Psi}_0(b)-\tau\circ {\Psi}_1(b)|<{1\over{m+1}}<\min\{\ep_1/64, d/8\}\rforal b\in A_1.
$$

Again, keep in mind that $K_0(A_1)=\{0\}.$ Applying \ref{Lpartuniq}, we obtain
 a unitary $u\in Q$ and a  ${\mathcal F}_1$-$\ep/16$-multiplicative \cpc\,  $L: A\to {\text{C}}([0,3/4], Q)$
such that
\hspace{-0.1in}\beq\label{Talwtr-13}
 \pi_0\circ L=\Psi_0,\,\,\, \pi_{3/4}\circ L=\text{Ad}u\circ \Psi_1.
\eneq
Moreover,
\begin{equation}\label{Talwtr-14}
 |\mathrm{tr}\circ \pi_t\circ L(h)-\mathrm{tr}\circ \pi_0\circ L(h)|<1/m<\ep'/64, \quad h\in {\mathcal F}_1,\ t\in [0, 3/4].
\end{equation}
Here, $\pi_t: {\text{C}}([0,3/4], Q)\to Q$ is the point evaluation at $t\in [0,3/4].$
There is a continuous path of unitaries $\{u(t): t\in [3/4,1]\}$ such that
$u(3/4)=u$ and $u(1)=1_Q.$
Define $L_1: A_1\to \mathrm{C}([0,1], Q)$
by
 $\pi_t\circ L_1=\pi_t\circ L$ for $t\in [0, 3/4]$ and
$\pi_t\circ L_1={\text{{Ad}} }u_t \circ \Psi_1$ for $t\in (3/4, 1].$
$L_1$ is a ${\cal F}_1$-$\ep/16$-multiplicative \cpc\, from $A_1$ into $\mathrm{C}([0,1], Q).$
Note now
\beq\label{Talwtr-15}
&&\pi_0\circ L_1=\Psi_0\quad\mathrm{and}\quad\pi_1\circ L_1=\Psi_1\andeqn\\
&&|tr\circ \pi_t\circ L(h)-\mathrm{tr}\circ \Psi_1(h)|<\ep'/64\rforal h\in {\cal H}.
\eneq

Fix an integer $k\ge 2.$ Let $\kappa_i: \mathrm{M}_k\to \mathrm{M}_{k(m+1)}$  ($i=0,1$) be defined by
\beq\label{Talwtr-16}
\kappa_0(c)=(\overbrace{c\oplus c\oplus\cdots\oplus c}^m\oplus 0), \quad\mathrm{and}\,\,\,
\kappa_1(c)=(\overbrace{c\oplus c\oplus\cdots\oplus c}^{m+1})
\eneq
\hspace{-0.1in}for all $c\in \mathrm{M}_k.$ Define
$$
C_0=\{(f,c): \mathrm{C}([0,1], \mathrm{M}_{k(m+1)})\oplus \mathrm{M}_k:  f(0)=\kappa_0(c)\andeqn f(1)=\kappa_1(c)\}.
$$
and set
$$
C_0\otimes Q = C_1.
$$
Note that $C_0\in {\cal C}_0^0$  and  $C_1$ is an inductive limit
of  Razak algebras $C_0\otimes \mathrm{M}_{n!}.$  Moreover
$\mathrm{K}_0(C_1)=\mathrm{K}_1(C_1)=\{0\}.$
Put $p_0=\sum_{i=1}^m1_Q\otimes e_i.$
Define  ${\bar \kappa}_0: Q\to p_0(Q\otimes Q)p_0$ to be the unital \hm\,
defined by ${\bar \kappa}_0(a)=a\otimes \sum_{i=1}^me_i$ and
${\bar \kappa}_1(a)=a\otimes 1_Q$ for all $a\in Q.$

Then
one may write
$$
C_1=\{(f, c)\in \mathrm{C}([0,1], Q)\oplus Q: f(0)={\bar {\kappa}}_0(c)\andeqn f(1)={\bar \kappa}_1(c)\}.
$$
Note that ${\bar \kappa}_0\circ \psi(b)=\Psi_0(b)$ for all $b\in A_1$ and
${\bar \kappa}_1\circ \psi(b)=\Psi_1(b)$ for all $b\in A_1.$  Thus
one can define  $\Phi': A_1\to C_1$ by
$\Phi'(b)=(L_1(b), \psi(b))$ for all $b\in A_1.$

Then  $\Phi'$ is a ${\cal F}_1$-$\ep/16$-multiplicative \cpc\,
such that
\beq\label{Talwtr-17}
|\mathrm{tr}(\pi_t\circ \Phi'(h))-\mathrm{tr}\circ \psi(h)|<\ep'/4 \rforal h\in {\cal H}.
\eneq
Let $\mu$ denote 
Lebesgue measure on $[0,1].$
There is a \hm\, $\Gamma: \mathrm{Cu}^{\sim}(C_1)\to \mathrm{Cu}^{\sim}({\cal W})$ (see \cite{Robert-Cu}) such that
$\Gamma(f)(\tau_\mathcal W)=(\mu\otimes \mathrm{tr})(f)$ for all $f\in \mathrm{Aff}(\mathrm{T}(C_1)),$
where $\tau_\mathcal W$ is the unique tracial state of ${\cal W}.$
By \cite{Robert-Cu}, there exists a \hm\,
$\lambda: C_1\to {\cal W}$
such that
\beq\label{Talwtr-18}
\tau_\mathcal W\circ \lambda((f,c))=\int_0^1 \mathrm{tr}(f(t))d t\rforal (f, c)\in C_1.
\eneq
Finally, let
$\Phi=\lambda\circ \Phi'.$
Then $\Phi$ is a ${\cal F}_1$-$\ep$-multiplicative \cpc\, from
$A_1$ into ${\cal W}.$  Moreover, one computes
that
\beq\label{Talwtr-19}
|\tau_\mathcal W\circ \Phi(h)-\tau(h)|<\ep'\rforal h\in {\cal H},
\eneq
as desired.
\end{proof}

\begin{thm}\label{clas-thm}
Let $A$ and $B$ be non-unital separable simple (finite) \CA s with finite nuclear dimension and with non-zero traces.
Suppose that both $A$ and $B$ are $KK$-contractible. 
Then  $A\cong B$ if and only if
there is an isomorphism (scale preserving affine homeomorphism) $\Gamma: ({\widetilde{\mathrm T}}(B), \Sigma_B)\cong ({\widetilde{\mathrm T}}(A), \Sigma_A).$ 

Moreover,  there is an isomorphism $\phi: A\to B$ such that
$\phi$ induces $\Gamma.$ 

\end{thm}
\begin{proof}

Let 
$\Gamma:({\widetilde {\mathrm T}}(B), \Sigma_B)\to ({\widetilde {\mathrm T}}(A), \Sigma_A)$  
be
an isomorphism.
By \ref{Ctsang}, we may assume that $B\in {\cal M}_0$ 
{{(an inductive limit of Razak algebras).}}
Recall that $A$ is ${\cal Z}$-stable (by \cite{T-0-Z}). 
Let $a\in \mathrm{Ped}(A)_+$ with $\|a\|=1$ such that $A_0=\overline{aAa}$ has 
continuous scale (see  5.2 of \cite{eglnp}).
Then $\mathrm{T}(A_0)$ is a  metrizable Choquet simplex and is a base 
for the cone 
${{\widetilde{\mathrm T}}}(A).$ 
Let $b\in B_+$ be such that
$$
\mathrm{d}_{\Gamma(\tau)}(b)=d_\tau(a)\rforal \tau\in {{\widetilde{\mathrm T}}}(A).
$$
Set $\overline{bBb}=B_0.$ Then
$\Gamma$ gives an affine homeomorphism from $\mathrm{T}(A_0)$ onto $\mathrm{T}(B_0).$
It follows from \ref{TalWtrace} that every tracial state of $A_0$ or $B_0$ is a ${\cal W}$-trace. 
By \ref{Z-stable}, $A_0$ and $B_0$ are tracially approximately divisible.  It follows from 
\ref{TMW} that $A_0, B_0\in {\cal D}_0.$ Then, by \ref{TTMW}, there is an isomorphism
$\phi: A_0\to B_0$ such that
$\phi_T$ gives $\Gamma|_{T(B_0)}$ and by \   By \cite{Br1}, this induces
an isomorphism ${\tilde \phi}: A\otimes {\cal K}\to B\otimes {\cal K}.$
Fix a strictly positive element $a_0\in A$ with $\|a_0\|=1$ such
that
$$
d_\tau(a_0)=\Sigma_A(\tau),\quad \tau\in {\widetilde{\mathrm T}}(A).
$$
Set ${\tilde \phi}(a_0) = b_0.$
Then ${\tilde \phi}$ gives an isomorphism from $A$ to $B_1:=\overline{b_0(B\otimes {\cal K})b_0}.$
Let $b_1\in B$ be a strictly positive element, so that
$$
d_\tau(b_1)=\Sigma_B(\tau),\quad \tau\in {\widetilde{\mathrm T}}(B).
$$
Then
\vspace{-0.1in}$$
d_\tau(b_1)=d_\tau(b_0),\quad \tau\in {\widetilde{\mathrm T}}(B).
$$
Since $B$ is a  separable simple \CA\, with stable rank one,  this implies
there exists  an isomorphism $\phi_1: B_1\to B$ such that $(\phi_1)_\mathrm{T}={\mathrm{id}}_{{\widetilde{\mathrm T}}(A)}$
{{(see Theorem 3 of \cite{CEI-CuntzSG}, this also follows from \cite{Razak-W} as $B$ is an inductive limit of Razak algebra---see also \cite{Robert-Cu}.)}}
Then the composition $\phi_1\circ {\tilde \phi}|_A$ gives the required isomorphism.
\end{proof}

\begin{cor}\label{clas-hom}
Let $A, B$ be simple separable KK-contractible  finite \CA s with finite nuclear dimension. If there is a homomorphism $\xi: ({\widetilde{\mathrm T}}(B), \Sigma_B) \to ({\widetilde{\mathrm T}}(A), \Sigma_A)$, then there is a \CA\, homomorphism $\phi: A {\to} B$ such that $\phi_*=\xi$.
\end{cor}
\begin{proof}
In 
view of \ref{clas-thm} and \ref{Ctsang}, 
this follows from 
the classification of limits of Razak algebras (\cite{Razak-W}).

\end{proof}

\vspace{0.4in}

\appendix

\section{}


\vspace{0.2in}

This section,   mainly contributed by 
Huaxin Lin,  
 removes  the necessity  of assuming 
that $A$ has stable rank one in 
Theorem \ref{clas-thm} and \ref{clas-hom}.

 The main purpose of this appendix  is to  prove Theorem \ref{Tappendix} below. 
The existence of  a map as stated in   \ref{Tappendix} was proved in   \cite{Robert-Cu} under the additional assumption 
that $A$ has stable rank one. It is needed in the proof of \ref{TMW}.

 Subsection A.1 of the appendix also contains some 
results of independent interest.
In particular, Corollary \ref{A1Cinv}, together with the construction of models in \cite{point-line}, establishes
the 
range of the
Elliott invariant for Jiang-Su stable separable simple  exact \CA s.

\subsection{Strict comparison in ${\widetilde A}$}

\begin{defn}
Let $B$ be a  \CA\, with ${\mathrm{T}}(B)\not=\emptyset$ and  let $S\subseteq {\mathrm{T}}(B)$ be a subset.
Suppose that $a\in (B\otimes {\cal K})_+$ 
is such that $d_\tau(a)<+\infty$ for 
all $\tau\in {\mathrm{T}}(B).$
Define
\beq
\omega_S(a)=\inf\{\sup\{d_\tau(a)-\tau(c): \tau\in S\}: c\in \overline{a(B\otimes {\cal K})a},\,\,\, 0\le c\le 1\}.
\eneq
{{Let us note that, {{when}} $S$ is compact,  $\omega_S(a)=0$ if and only if $d_\tau(a)$ is continuous on $S.$}}
{{Also}} note that if $a, b\in (B\otimes {\cal K})_+,$ $0\le a, b\le 1$ and $a\lesssim  b,$ 
then there exists a sequence $x_n\in B\otimes {\cal K}$ such that 
$x_n^*x_n\to a$ and $x_nx_n^*\in \overline{b(B\otimes {\cal K})b}.$
It follows, for any $1>\dt>0,$ $f_\dt(x_n^*x_n)\to f_\dt(a).$ 
Note that  $\tau(f_\dt(x_n^*x_n))=\tau(f_\dt(x_nx_n^*))$
for any $\tau\in {\mathrm{T}}(B). $   
We  conclude that, if $a\lesssim b,$ 
there exists a sequence $\{c_k\}$ in $\overline{b(B\otimes {\cal K})b}_+$ with $0\le c_k\le 1$
such that 
$$
\lim_{k\to\infty}\sup\{|\tau(c_k)-\tau(f_{1/k}(a))|:
\tau\in {\mathrm{T}}(B)\}=0.
$$
Consequently, if we further assume $d_\tau(a)=d_\tau(b)$, then $\omega_S(a)\geq \omega_S(b)$. Note that, if $a\sim b,$ then $d_\tau(a)=d_\tau(b).$ 
Hence, when $a\sim b,$ we have $\omega_S(a)=\omega_S(b).$

Now let $A$ be a \CA\, with ${\mathrm{T}}(A)\not=\emptyset$ and with compact ${\mathrm{T}}(A).$ 
Let $a\in ({\widetilde A}\otimes {\cal K})_+$ be such that $d_\tau(a)<+\infty$ for all $\tau\in {\mathrm{T}}(A).$ 
We will write $\omega(a)$ for $\omega_{{\mathrm{T}}(A)}(a),$ namely,
\beq
\omega(a)=\inf\{\sup\{d_\tau(a)-\tau(c): \tau\in {\mathrm{T}}(A)\}: c\in \overline{a({\widetilde A}\otimes {\cal K})a},\, 0\le c\le 1\}.
\eneq

As mentioned above, if $b\in ({\widetilde A}\otimes {\cal K})_+,$ $0\le b\le 1$ and $a\sim b,$ in ${\widetilde A}\otimes {\cal K}, $ 
then $\omega(a)=\omega(b).$

\end{defn}

\begin{lem}\label{Ldecomp}
Let $A$ be a separable  stably projectionless simple  \CA\, such that ${{{\mathrm{M}}_r(A)}}$ almost has stable rank one
for every integer $r\ge 1$ and  $\mathrm{QT}(A)=\mathrm{T}(A)$ which has 
strict comparison for positive elements and 
has  continuous scale. Suppose  that $\mathrm{Cu}(A)=\mathrm{LAff}_+({\mathrm{T}}(A)).$
Suppose also that $a\in {{{\mathrm{M}}_r({\widetilde A})}}$  with $0\le a \le 1$  for some integer $r\ge 1$
and  $0<\la \pi(a)\ra,$ where 
$\pi: {\widetilde A}\to \C$ is the quotient map. 
Suppose further that
\beq
{{\sigma_0}}:=\inf\{d_\tau(a): \tau\in {\mathrm{T}}(A)\}>4\omega(a).
\eneq
Then, for any $2\omega(a)<d<{{\sigma_0}}$ and $\omega(a)/2>\ep_0>0,$ there is $b\in {\mathrm{M}}_r(A)_+$ with 
$b\le a$ such that
\beq
2\omega(a)<d_\tau(b)<d\tforal \tau\in {\mathrm{T}}(A)
\eneq
and,  for  any 
$0<\ep<\inf\{d_\tau(b): \tau\in {\mathrm{T}}(A)\},$
 there is also $a_1\in {\mathrm{M}}_r({\widetilde A})_+$ 
such that  
$$ \pi(a_1)=\pi(a'),\,\,\,
b\oplus a_1\le a',
$$
 with $\la a'\ra=\la a\ra ,$ 
$d_\tau(a_1)>d_\tau(a)-d$ for all $\tau\in {\mathrm{T}}(A),$  and
$a_1$ also has the following property: if $\{c_n\}
\in {\mathrm{M}}_r({\widetilde A})_+$  is an increasing sequence such that 
$c_n\in \overline{a_1({\widetilde A}\otimes {\cal K})a_1}$ and $\tau(c_n)\nearrow d_\tau(a_1),$ 
then, for some $n_0\ge 1,$ 
\beq
d_\tau(a_1)-\tau(c_n)<\omega(a)+\ep_0+\ep\tforal \tau\in {\mathrm{T}}(A)\tand\hspace{-0.08in} \tforal n\ge n_0.
\eneq

\end{lem}

\begin{proof}
We first consider the case that $\la a \ra$ is not represented by a projection.
There exists an invertible matrix $y\in {\mathrm{M}}_r(\C)_+$ such that 
{{$y^{1/2}\pi(a)y^{1/2}=p$}} is a projection. Let $Y\in {\mathrm{M}}_r(\C\cdot 1_{\widetilde A})_+$ be 
the same invertible scalar matrix. Then $\pi(Y^{1/2}aY^{1/2})=p.$ 
It is clear that $\la a\ra=\la Y^{1/2}aY^{1/2}\ra $ and we may replace $a$ by
$Y^{1/2}aY^{1/2}.$ 
So we assume that $\pi(a)=p.$

Choose $\eta_0>0$ such that, for $0<\eta<\eta_0,$
\beq
d_\tau(a)-\tau(f_\eta(a))<\omega(a)+\ep_0\rforal \tau\in {\mathrm{T}}(A).
\eneq

Let $(e_n)$ be an approximate identity for $\overline{a{\mathrm{M}}_r(A)a}$ such that
$e_{n+1}e_n=e_n,$ $n=1,2,....$ 

There exists $n_0\ge 1$ (recall $T(A)$ is compact)  such that
\beq\label{Ldoc1-1}
\tau(e_n)>d\andeqn d_\tau(a)-\tau(e_n)<\omega(a)+\ep_0\rforal \tau\in {\mathrm{T}}(A) \andeqn \rforal n\ge n_0.
\eneq
By a standard 
compactness
argument, 
for a fixed $n_0+1,$ there exists $\eta_0>\eta_1>0$
such that 
\beq
\tau(f_{\eta_1}(a))>\tau(e_{n_0+3})>d\rforal \tau\in {\mathrm{T}}(A).
\eneq
Let $(e_{1,n})$ be an approximate identity for $\overline{f_{\eta_1}(a){\mathrm{M}}_r(A)f_{\eta_1}(a)}$
with $e_{1,n}e_{1,n+1}=e_{1,n},$ $n=1,2,....$ 
By the same compactness argument, we have $e_{1, n}$ for some $n$ such that
\beq
\tau(e_{1,n})>\tau(e_{n_0+3})\rforal \tau\in {\mathrm{T}}(A).
\eneq
It follows that $e_{n_0+2}\lesssim e_{1, n+1}.$ Since $A$ has almost stable rank one, one has
\beq
w^*e_{n_0+2}w\le f_{\eta_1/2}(a)
\eneq
for some unitary $w\in {\mathrm{M}}_r(A)^\sim$  {{(see  the last part of Lemma 3.2 of \cite{eglnp}).}}
Choose {{a strictly positive function}} ${{g}}_{1,\eta_1}\in C_0((0,1])_+$ such that ${{g}}_{1,\eta_1}(t)=1,$ if $t\ge \eta_1/4,$ 
and
${{g}}_{1, \eta_1}(t)$ is linear
on $[0, \eta_1/4).$
In particular, $f_{\eta_1/2}{{g}}_{1, \eta_1}=f_{\eta_1/2}.$
Put $a'={{g}}_{1, \eta_1}(a).$ 
Note $\pi(a')=\pi(a)$ and $\la a'\ra =\la a\ra.$
Note   also that 
\beq
0\le w^*e_{n_0}w\le w^*e_{n_0+1}w\le w^*e_{n_0+2}w\le f_{\eta_1/2}(a)\le a'\andeqn\\
\label{Lnote1-10}
\tau(w^*e_{n_0}w)>d\andeqn d_\tau(a')-\tau(w^*e_{n_0}w)<\omega(a)+\ep_0\rforal \tau\in {\mathrm{T}}(A).
\eneq
In particular, $d_\tau(w^*e_{n_0}w)> 5\omega/2
\rforal \tau\in {\mathrm{T}}(A).$
There exists $b_0\in {\mathrm{M}}_r(A)_+$ with 
$$d_\tau(b_0)=2\omega(a)+\min\{(3/4)(d-2\omega(a)), (3/4)(\tau(w^*e_{n_0}w)- 2\omega(a))\}
=2\omega(a)+(3/4)(d-2\omega(a))$$ for all 
$\tau\in {\mathrm{T}}(A).$ Note that $d_\tau(b_0)\in \Aff({\mathrm{T}}(A)).$
We have 
\beq
d_\tau(b_0)<d_\tau(w^*e_{n_0}w)\rforal \tau\in {\mathrm{T}}(A).
\eneq
Since ${\mathrm{M}}_r(A)$ almost has  stable rank one,  by 3.2 of \cite{eglnp}, one concludes that
there exists\\ $b'\in \overline{w^*e_{n_0}w{\mathrm{M}}_r(A)w^*e_{n_0}w}$
such that $d_\tau(b')=d_\tau(b_0)$ for all $\tau\in {\mathrm{T}}(A).$
Note that $b'a'=b'.$ 
Let $\ep>0.$  Since $d_\tau(b_0)$ is continuous on ${\mathrm{T}}(A)$ and ${\mathrm{T}}(A)$ is compact, 
there exists $\dt_0>0$ such that
\beq
\tau(f_\dt(b_0))>d_\tau(b_0)-\min\{(d-2\omega(a))/2, \ep/4\}\rforal \tau\in {\mathrm{T}}(A) 
\eneq
and $ 0<\dt\le \dt_0.$

Put $b=f_{\dt_0}(b'),$ $b_1=f_{\dt/2}(b')$ and $b_2=f_{\dt/4}(b').$
Note that $b\le b_1\le b_2\le w^*e_{n_0+1}w.$
Note also that
\beq
2\omega(a)<d_\tau(b)<d\andeqn 0<d_\tau(b_2)-\tau(b)<\ep/4\rforal \tau\in {\mathrm{T}}(A).
\eneq
{{So (A.4) holds.}}	
Put $a_1=a'-b_1.$ {{Note that $a'b=b,$   $a_1\perp b,$ and
$a_1+ b\le a'.$}}
{{Since}} $\pi(a_1)=\pi(a'),$ $\la \pi(a_1)\ra =\la \pi(a)\ra .$
 Let $p_{a}$ be the open projection corresponding to $a,$ $p_{a_1}$ the open projection 
corresponding to $a_1$ and
$p_{b'}$ be the open projection corresponding to  $b'$ in ${\mathrm{M}}_r({\widetilde A})^{**}.$  
Note that $p_a$ is the same as the open projection corresponding to $a'.$ Then
$p_a\ge p_{a_1}\ge p_{a}-p_{b'},$ 
\beq\label{Ldoc1-2}
&&d_\tau(a_1)=\tau(p_{a_1})\ge \tau(p_{a'}-p_{b'})= \tau(p_{a'}{{)}}-\tau(p_{b'})\\
&&= \tau(p_{a'})-d_\tau(b')
> d_\tau(a)-d\andeqn\\
&&d_\tau(a_1)=\tau(p_{a_1})
<\tau(p_{a}-b)
<\tau(p_{a}-p_{b'})+\ep/4\\\label{Ldoc1-2+}
&&\hspace{0.2in}<\tau(p_{a})-\tau(p_{b'})+\ep/4=d_\tau(a)-d_\tau(b')+\ep/4 \rforal \tau\in {\mathrm{T}}(A).
\eneq
If $c_n$ is as stated, then $\tau(c_n)\nearrow \tau(p_{a_1}).$ Therefore, on ${\mathrm{T}}(A),$ which is compact, 
by a standard compactness argument, there is $n_1\ge 1$ such that
\beq\label{Lnote1-11}
\tau(w^*e_{n_0+1}w)-d_\tau(b')<\tau(c_n)\le \tau(p_{a_1})=d_\tau(a_1)
\eneq
for all  $\tau\in {\mathrm{T}}(A)$ and  for all  $n\ge n_1.$  It follows from   \eqref{Ldoc1-2+}, \eqref{Lnote1-10} and 
\eqref{Lnote1-11} that
\beq
&&d_\tau(a_1)-\tau(c_n)<d_\tau(a')-\tau(c_n)\\
&&<(\tau(w^*e_{n_0+1}w)+\omega(a)+\ep_0)-d_\tau(b')+\ep/4)-(\tau(w^*e_{n_0+1}w)-d_\tau(b'))\\
&&=\omega(a)+\ep_0+\ep/4\rforal \tau\in {\mathrm{T}}(A).
\eneq

Now we consider the case that $\la a\ra$ is represented by a projection $p\in M_m({\widetilde A}).$
We may write $p=a_1+b_1,$ where $a_1\in M_m(A)$ and $b_1\in M_m(\C\cdot 1_{\widetilde A})$
is a scalar matrix. In particular, 
$d_\tau(p)$ is continuous on ${\mathrm{T}}(A).$
Therefore $\omega(p)=0.$ 
Let $d>0.$  We may assume that $d<1/2.$ Since $\mathrm{Cu}(A)=\mathrm{LAff}_+({\mathrm{T}}(A)),$ choose an element $b_0\le A$ 
such that $d_\tau(b_0)=d/4$ for all $\tau\in {\mathrm{T}}(A).$   Note that $pb_0=b_0$ and $d_\tau(b_0)$ is continuous.
Now with $\omega(p)=0,$ with $a=p,$ and with this new $b_0,$ the rest of the proof  above (beginning 
with $b_0$ as constructed) applies.

\end{proof}

\begin{lem}\label{LNote1}
Let $A$ be a separable  stably projectionless simple  \CA\, such that ${\mathrm{M}}_n(A)$ has almost stable rank one
for all integers $n\ge 1$ and  ${\mathrm{QT}}(A)={\mathrm{T}}(A)$  which has 
strict comparison for positive elements and 
has continuous scale. Suppose also that ${\mathrm{Cu}}(A)=\mathrm{LAff}_+({\mathrm{T}}(A)).$
Let $a, b\in {\mathrm{M}}_r({\widetilde A})_+.$ 
%
Suppose 
that $\la \pi(a)\ra \le \la \pi(b)\ra(<\infty),$ where $\pi: {\widetilde A}\to \C$ is the quotient map, and 
\beq\label{Notefn1}
d_\tau(a)+4\omega(b)<d_\tau(b)\rforal \tau\in {\mathrm{T}}(A).
\eneq
Then $a\lesssim b.$
\end{lem}

\begin{proof}
Since, $f_\eta(a)\lesssim b$ for all 
$1/2>\eta>0$ implies $a\lesssim b.$
We may replace $a$ by $f_\eta(a).$
If $\la a\ra$ is represented by a projection, 
then   $d_\tau(a)$ is continuous. So
\beq
\inf\{d_\tau(b)-d_\tau(a): \tau\in {\mathrm{T}}(A)\}>4\omega(b).
\eneq

Otherwise, 
fix $1/2>\eta>0.$ By applying 7.1 of \cite{eglnp}, there exists 
$\eta>\eta_1>0$ and a continuous function $f: {\mathrm{T}}(A) \to \R^+$ such that
\beq
d_\tau((a-\eta)_+)<f(\tau)<d_\tau((a-\eta_1)_+)<d_\tau(b)\rforal \tau\in {\mathrm{T}}(A).
\eneq
Then 
\beq
\inf\{d_\tau(b)-d_\tau(f_{\eta}(a)): \tau\in {\mathrm{T}}(A)\}> 4\omega(b).
\eneq
Thus, in both cases (after replacing $a$ by $f_\eta(a)$), 
\beq
d=\inf\{d_\tau(b)-d_\tau(a): \tau\in {\mathrm{T}}(A)\}>4\omega(b).
\eneq

By applying \ref{Ldecomp}, one obtains non-zero and mutually orthogonal elements 
$b_0\in {\mathrm{M}}_r(A)_+$ and $b_1, {{b'}}\in {\mathrm{M}}_r({\widetilde A})_+$ such that 
\beq
&&b_0+b_1\le b',\,\,\, \la b'\ra =\la b\ra, \pi(b_1)=\pi(b'),\\\label{Ldecomp-n1}
&&2\omega(b)<d_\tau(b_0)<d/2,\,\,\, d_\tau(b_1)>d_\tau(b)-d/2\rforal \tau\in {\mathrm{T}}(A).
\eneq
and, for any $c_n\in {\mathrm{M}}_r(A)_+$ with $c_n\in \overline{b_1({\widetilde A}\otimes {\cal K})b_1}$ and $d_\tau(c_n)\nearrow d_\tau(b_1)$ on ${\mathrm{T}}(A),$
there exists $n_0\ge 1$ such that
\beq\label{Notef1-10}
d_\tau(b_1)-d_\tau(c_n)<\omega(b)+(1/64)\inf\{\tau(b_0):\tau\in {\mathrm{T}}(A)\} \rforal \tau\in {\mathrm{T}}(A).
\eneq
Moreover, $\la \pi(b_1)\ra =\la \pi(b)\ra.$
Replacing $b$ by $b',$ \wilog, we may assume that $b_0+b_1\le b.$ 

Put $d_0=\inf\{\tau(b_0):\tau\in {\mathrm{T}}(A)\}.$

There exists an invertible matrix $y_1\in {\mathrm{M}}_r(\C)_+$
such that $y_1^{1/2}\pi(b_1)y_1^{1/2}=p_1$ is a projection.
Let $Y_1\in {\mathrm{M}}_r({\widetilde A})$ denote the scalar matrix such that $\pi(Y_1)=y_1.$
Note that $\la Y_1^{1/2}b_1Y^{1/2}\ra =\la b_1\ra,$ 
$Y_1^{1/2}c_nY^{1/2}\le Y_1^{1/2}b_1Y^{1/2}$ and $d_\tau(Y^{1/2}c_nY^{1/2})=d_\tau(c_n).$
So, 
replacing $b_1$ by $Y_1^{1/2}b_1Y^{1/2},$ we may assume that 
$\pi(b_1)=p_1.$ 
Similarly, we may assume that $\pi(a)=p_2$ is also a projection. 
There is a scalar matrix $U\in {\mathrm{M}}_r({\widetilde A})$ such that 
$\pi(U^*aU)\le p_2.$ \Wlog, we may assume that $p_2\le p_1.$ 

We may further assume that there are integers $0\le m_2\le m_1$ such that 
\beq
p_i=\diag(\overbrace{1,1,..,1}^{m_i},0,...,0),\,\,\, i=1,2.
\eneq
Let $P_i=\diag(\overbrace{1_{\widetilde A}, 1_{\widetilde A},...,1_{\widetilde A}}^{m_i},0,..,0)\in {\mathrm{M}}_r({\widetilde A})$ so that $\pi(P_i)=p_i,$
$i=1,2.$ 

Note $(b_1-1/n)_+\le b_1$ and $d_\tau((b_1-1/n)_+)\nearrow d_\tau(b_1),$ 
so
by \eqref{Notef1-10},
for some $\dt_1>0,$
\beq\label{Notef1-11+}
d_\tau(b_1)-d_\tau(f_{\dt}(b_1))<\omega(b)+d_0/64\rforal \tau\in {\mathrm{T}}(A)
\eneq
and 
all
$0<\dt<\dt_1.$ 

Let $(e_n)$ be an approximate identity for $A$
such that $e_ne_{n+1}=e_{n+1}e_n=e_n,$ $n=1,2,....$
Put
\beq
E_n=\diag(e_n, e_n,...,e_n)\in {\mathrm{M}}_r(A),\,\,\, n=1,2,....
\eneq
Then $(E_n)$ is an approximate identity for ${\mathrm{M}}_r(A)$ and 
$P_iE_n=E_nP_i,$  $i=1,2,$ and $n=1,2,....$ 

We have $b_1^{1/2}E_nb_1^{1/2}\nearrow b_1$ (in the strict topology). Let $c_n= E_n^{1/2}b_1E_n^{1/2},$ $n=1,2,....$ 
It follows that $d_\tau(c_n)\nearrow d_\tau(b_1)$ on ${\mathrm{T}}(A).$
By the construction of $b_1,$ there exists $n_0\ge 1$ such that
\beq\label{Notef1-8n}
d_\tau(b_1)-d_\tau(b_1^{1/2}E_n^2b_1^{1/2})=d_\tau(b_1)-d_\tau(c_n)<\omega(b)+d_0/64
\eneq
for all $\tau\in {\mathrm{T}}(A)$ and for all $n\ge n_0.$

One then computes, by \eqref{Ldecomp-n1} and  \eqref{Notef1-8n},  that, for $n\ge n_0,$ 
\beq\label{Notef1-8}
d_\tau(a)<d_\tau(c_n)\rforal \tau\in {\mathrm{T}}(A).
\eneq

On the other hand,  since $\pi(b_1)=\pi(P_1)$ and $\pi(a)=\pi(P_2),$ 
\beq\label{A200117-n1}
&&\lim_{k\to\infty}\|(E_k^{1/2}b_1E_k^{1/2}+(1-E_k)^{1/2}P_1(1-E_k)^{1/2})-b_1\|=0\andeqn\\
&&\lim_{k\to\infty}\|(E_k^{1/2}aE_k^{1/2}+(1-E_k)^{1/2}P_2(1-E_k)^{1/2})-a\|=0.
\eneq
Put $x_k=E_k^{1/2}b_1E_k^{1/2}+(1-E_k)^{1/2}P_1(1-E_k)^{1/2}$ and $y_k=E_k^{1/2}aE_k^{1/2}+(1-E_k)^{1/2}P_2(1-E_k)^{1/2},$
$k=1,2,....$
Since
\beq
\lim_{k\to\infty}\|f_{\dt_1/2}(x_k)-f_{\dt_1/2}(b_1)\|=0,
\eneq
we may assume, \wilog,  for all $k\ge 1,$  that 
\beq
\tau(f_{\dt_1/2}(x_k))\ge \tau(f_{\dt_1/2}(b_1))-d_0/64 \rforal  \tau\in {\mathrm{T}}(A)\}.
\eneq
It follows 
by
 \eqref{Notef1-11+} (with $\dt=\dt_1/2$) that
\beq\label{Notef1-19}
\tau(f_{\dt_1/2}(x_k))>d_\tau(b_1)-\omega(b)-3d_0/64\rforal \tau\in {\mathrm{T}}(A).
\eneq
Since $A$ has continuous scale, there is $k_0\ge n_0$ such that
\beq\label{Notef1-20}
d_\tau(1-E_n))\le \tau(1-E_{n-1})<{{d_0/64}} \rforal \tau\in {\mathrm{T}}(A)\andeqn \rforal n\ge k_0.
\eneq
It follows that, for $k\ge k_0,$ 
\beq\label{Notef1-21}
&&\tau(f_{\dt_1/2}(x_k))\le d_\tau(x_k)\le d_\tau(c_k)+d_0/64\\
&&=d_\tau(b_1^{1/2}E_k^2b_1^{1/2})+d_0/64\le d_\tau(b_1)+d_0/64 \rforal \tau\in {\mathrm{T}}(A).
\eneq
Let $g_{\dt_1}\in C_0((0,1])_+$ with $1\ge g(t)>0$ for all $t\in (0, \dt_1/4),$
$g_{\dt_1}(t)\ge t$ for $t\in (0, \dt_1/16),$ $g_{\dt_1}(t)=1$ for $t\in (\dt_1/16, \dt_1/8)$ 
and $g_{\dt_1}(t)=0$ if $t\ge \dt_1/4.$ 

Since $g_{\dt_1}(x_k)f_{\dt_1/2}(x_k)=0,$ by 
\eqref{Notef1-21},
we conclude that, for $k\ge k_0,$ 
\beq\label{Notef1-22}
d_\tau(g_{\dt_1}(x_k))+\tau(f_{\dt_1/2}(x_k))\le d_\tau(x_k)\le d_\tau(b_1)+d_0/64\rforal \tau\in {\mathrm{T}}(A).
\eneq
Then, by \eqref{Notef1-19}, 
\beq
d_\tau(g_{\dt_1}(x_k))&\le&(d_\tau(b_1)-\tau(f_{\dt_1/2}(x_k)))+d_0/64\\\label{Notef1-22+}
&\le &  \omega(b)+3d_0/64+d_0/64=\omega(b)+d_0/16
\eneq
for all $ \tau\in {\mathrm{T}}(A)$
and for all $k\ge k_0.$ Moreover, since $\pi(x_k)=\pi((1-E_n)^{1/2}P_1(1-E_n)^{1/2})=p_1$ for 
all $n,$
\beq\label{Notef1-19+}
g_{\dt_1}(x_k)\in {\mathrm{M}}_r(A).
\eneq

It should be noted and will be used later that, for any $0\le x\le 1,$ 
\beq\label{Notef1-5}
x\le f_{\dt}(x)+g_{\dt_1}(x)\rforal 0<\dt<\dt_1/8.
\eneq

Fix an $\eta>0.$ 
Then there exists $k_1\ge k_0+2$ such that, since $\lim_{k\to\infty}\|y_k-a\|=0,$
\beq\label{Notef1-26}
(a-\eta)_+\lesssim y_k=E_k^{1/2}aE_k^{1/2}+(1-E_k)^{1/2}P_2(1-E_k)^{1/2}.
\eneq
Note that this holds regardless  of whether $\la a\ra$  represented by a projection or not.
Fix any $n\ge k_0\ge n_0,$.
By
\eqref{Notef1-8},
\beq
d_\tau(E_k^{1/2}aE_k^{1/2})&=&d_\tau(a^{1/2}E_ka^{1/2})\le d_\tau(a)
<d_\tau(c_n)\rforal \tau\in {\mathrm{T}}(A)
\eneq
{{and}} for any $k.$ 
Since $A$ has strict comparison,
\beq
E_k^{1/2}aE_k^{1/2}{{\lesssim}} c_n
\eneq
for any $n\ge k_0$ and any $k.$
Choose $k\ge \max\{k_1, n\}+2.$ In particular, $E_n$ and $(1-E_k)$ are mutually 
orthogonal. 
Then  
\beq\label{Notef1-30}
(a-\eta)_+ &\lesssim  &y_k \lesssim  E_k^{1/2}aE_k^{1/2}+ (1-E_k)^{1/2}P_2(1-E_k)^{1/2}\\
                 &\lesssim & c_n+(1-E_k)^{1/2}P_2(1-E_k)^{1/2}   \le c_n+(1-E_k)^{1/2}P_1(1-E_k)^{1/2} \\
                 &=& c_n+P_1(1-E_k)P_1\le c_n +P_1(1-E_n)P_1\\
                 &=& c_n+(1-E_n)^{1/2}P_1(1-E_n)^{1/2}=x_n.
                  \eneq
In other words, 
\beq\label{Notef1-31}
(a-\eta)_+\le x_n\rforal n\ge k_0.
\eneq                  
Choose $n\ge k_0$ such that (note that, by \eqref{A200117-n1}, $x_n\to b_1$ as $n\to\infty$)
$f_{\dt_1/{{8}}}(x_n){{\lesssim}} b_1.$
By \eqref{Notef1-5},
\beq\label{Notef1-33}
\la x_n\ra  &\le & \la f_{\dt_1/{{8}}}(x_n)+g_{\dt_1}(x_n)\ra \le \la f_{\dt_1/{{8}}}(x_n)\ra +\la g_{\dt_1}(x_n)\ra\\
&\le& \la b_1\ra +\la g_{\dt_1}(x_n)\ra.
\eneq
By \eqref{Notef1-22+}  and \eqref{Ldecomp-n1} and the strict comparison of $A,$ 
\beq\label{Notef1-34}
\la g_{\dt_1}(x_n)\ra \le b_0.
\eneq
Combining \eqref{Notef1-31}, \eqref{Notef1-33} and \eqref{Notef1-34}
\beq
\la (a-\eta)_+\ra \le \la b_1\ra +\la b_0\ra =\la b_1+b_0\ra \le \la b\ra.
\eneq
Since this holds for any $\eta>0,$ we conclude that
\beq
a \lesssim b.
\eneq

\end{proof}

\begin{cor}\label{Tnote3}
Let $A$ and $a$ be as in \ref{LNote1}.
Suppose that $b\in {\mathrm{M}}_r({\widetilde A})$ is such that $d_\tau(b)$ is continuous on ${\mathrm{T}}(A)$ and suppose that  
$\la \pi(a)\ra \le \la \pi(b)\ra$ and 
\beq
d_\tau(a)< d_\tau(b)\tforal \tau\in {\mathrm{T}}(A).
\eneq
Then $a\lesssim b.$

\end{cor}

\begin{defn}\label{DefS}
Let $A$ be a separable simple stably projectionless \CA\, 
with continuous scale and with strict comparison.  
Suppose also that $M_m(A)$ has almost stable rank one for all $m\ge 1,$ $\mathrm{QT}(A)=\mathrm{T}(A)$ 
and ${\mathrm{Cu}}(A)=\mathrm{LAff}_+({\mathrm{T}}(A)).$
In what follows we will continue to denote by $\pi$ the quotient map from ${\widetilde A}\to \C$ and 
its extension from $M_m({\widetilde A})\to M_m$ for all $m\ge 1,$  as well as from 
${\widetilde A}\otimes {\cal K}\to {\cal K}.$ 
 Let $S({\widetilde A})$ be the sub-semigroup of ${\mathrm{Cu}}({\widetilde A})$ generated by 
$\la a \ra \in {\mathrm{Cu}}(A)$ and those $x\in {\mathrm{Cu}}({\widetilde A})$ which is equal to 
the  supremum of an increasing sequence $(\la a_n\ra),$ where $d_\tau(a_n)\in \mathrm{Aff}_+({\mathrm{T}}(A))$ and $\la \pi(a_n)\ra <+\infty,$ and 
$\la x\ra$ is not represented by  a projection.
If $\la a\ra \in {\mathrm{Cu}}({\widetilde A}),$ we will write $\la a\ra\, \hat{}$ for the 
function $d_\tau(a)$ on ${\mathrm{T}}(A).$ 

For each $\la a\ra \in S({\widetilde A}),$ note that  $\la \pi(a) \ra =j(a)$ is either  an integer or $\infty.$
Let 
$$
L({\widetilde A})=\{(f, n): f\in {\mathrm{LAff}}_+({\mathrm{T}}(A)),\,\,\, n\in \N\cup\{0\}\cup\{\infty\}\}.
$$
We also define $(f, n)\le (g, m)$ if 
$f\le g$ and $n\le m.$ 
Define $\Gamma_0(\la a\ra ): S({\widetilde A})\to  L({\widetilde A})$ by $\Gamma_0(\la a \ra)=(\la a\ra \,\hat{}, j(a)).$

For any \CA\, $B,$ as a tradition, we use $V(B)$ for the semigroup of  Murray-von Neumann equivalence classes of projections 
in $B\otimes {\cal K}.$
\end{defn}

\begin{thm}\label{Tnote2}
Let $A$ be a stably projectionless simple  \CA\,such that ${\mathrm{M}}_r(A)$  has almost stable rank one
for all $r\ge 1,$ $\mathrm{QT}(A)=\mathrm{T}(A),$  $A$ has 
continuous scale,  and 
${\mathrm{Cu}}(A)= \mathrm{LAff}_+({\mathrm{T}}(A)).$
Then  $\Gamma_0: S({\widetilde A})\to L({\widetilde A})$ is an ordered semigroup isomorphism. 

For any $x=\la a\ra , y=\la b\ra \in V({\widetilde A})\sqcup S({\widetilde A}),$ 
if $\hat{x}< \hat{y}$ for all $\tau\in {\mathrm{T}}(A)$ and $\la \pi(a)\ra \le \la \pi(b)\ra,$  then 
$x\le y.$ Moreover, if $x$ is not represented by 
a projection, then $\hat{x}\le \hat{y}$  and $\la \pi(a)\ra \le \la \pi(b)\ra$ imply that $x\le y.$

Furthermore, if $\la a\ra \, \hat{}\le \la b\ra \, \hat{}$ and 
$\la \pi(a)\ra \le \la \pi(b)\ra,$
and if  $\la b\ra \in S({\widetilde A})$ and $\la a\ra \in {\mathrm{Cu}}({\widetilde A})$ is any element which 
is not represented by a projection,  
then $x\le y.$
\end{thm}

\begin{proof}
We will leave the additive part to the reader.
We first note that $\Gamma_0|_{{\mathrm{Cu}}(A)}$ is an ordered semigroup isomorphism 
to $\{(f, 0): f\in \mathrm{LAff}_+({\mathrm{T}}(A))\}$ (Note that we also use the fact that  $A$ is stably projectionless). 
It is then also clear that $\Gamma_0$ is order preserving. 

Claim 1: If $\la a \ra\in {\mathrm{Cu}}({\widetilde A}),$ $ \la b\ra \in S({\widetilde A})$ and $\la b\ra\, \hat{}\in \mathrm{Aff}_+({\mathrm{T}}(A))$ 
{{(i.e., $\la b\ra\, \hat{}$ is continuous),}}
and if  $\Gamma_0(\la a\ra)\le  \Gamma_0(\la b\ra ),$  then $\la a\ra \le \la b\ra,$ provided that 
$\la a\ra $ is not represented by a projection.

If $\la a \ra \in {\mathrm{Cu}}({{\widetilde{A}}}),$ then, since $A$ is stably projectionless, for any $\ep>0,$
\beq
\la (a-\ep)_+\ra\,\hat{}< \la b\ra \,\hat{}.
\eneq
Note that $\la \pi(b)\ra <\infty.$ 
Note also $\pi((a-\ep)_+)\le \pi(a).$ 
It follows from \ref{LNote1} {{(and \ref{Tnote3})}} that 
\beq
(a-\ep)_+\lesssim b.
\eneq
Therefore $a\lesssim b.$
This proves Claim 1.

Claim 2: If $\la a\ra, \la b\ra \in S({\widetilde A}),$ $\Gamma_0(\la a\ra)=\Gamma_0(\la b\ra)$ and  
$\la b\ra \, \hat{}\in \Aff_+({\mathrm{T}}(A)),$ Then $\la a\ra =\la b\ra.$

Note that $\la a\ra \,\hat{}=\la b\ra \, \hat{}$ {{(so both continuous).}} If $j(a)=j(b)=0,$ 
then this follows from the fact that $\Gamma_0|_{{\mathrm{Cu}}(A)}$ is an isomorphism. 
So we assume $j(a)=j(b)\not=0.$ 
By Claim 1, $\la a\ra \le \la b\ra \le \la a\ra.$ So $\la a\ra =\la b\ra.$

Now  assume that $a\in ({\widetilde A}\otimes {\cal K})_+,$ 
 $\la a\ra $ is not represented by a projection, $\la b\ra \in S({\widetilde A})$ and 
\beq
\la a\ra \, \hat{}\le \la b\ra \, \hat{}\andeqn \la \pi(a)\ra \le \la \pi(b)\ra.
\eneq
Write 
$\la b_n\ra \le \la b_{n+1}\ra$ and $b=\sup\{\la b_n\ra\},$ where
$\la b_n\ra\, \hat{}$ are continuous and $\la \pi(b_n)\ra <\infty.$
Then
\beq
\la (a-\ep)_+\ra \, \hat{}\le \la b\ra \, \hat{}\andeqn \la \pi((a-\ep)_+)\ra <\infty 
\eneq
{{(for any $\ep>0$).}}
Since $\la a\ra$ is not represented by projections, for any sufficiently small $\ep>0,$
\beq
\la (a-\ep)_+\ra\, \hat{} <\la (a-\ep/2)_+\ra\hat{}  <\la b\ra \, \hat{}.
\eneq
On the compact set ${\mathrm{T}}(A),$ one finds an integer  $k\ge 1$ such that 
\beq
\la (a-\ep)_+\ra\, \hat{}< \la b_{k}\ra \, \hat{} \andeqn \la \pi((a-\ep)_+)\ra \le \la \pi(b_k)\ra.
\eneq
It then follows from \ref{LNote1} that 
\beq
\la (a-\ep)_+\ra< \la b_{k}\ra \le \la b\ra.
\eneq
Therefore 
\beq
\la a\ra \le \la b\ra.
\eneq

This also implies that if  $\Gamma_0(a)=\Gamma_0(b)$ then $\la a\ra =\la b\ra.$ 
In particular, $\Gamma_0$ is the injective and the inverse restricted to the image is also order 
preserving.

To complete the proof of the first part the statement, it remains to show that the map is surjective. 
Note that  ${\mathrm{Cu}}(A)=\mathrm{LAff}_+({\mathrm{T}}(A)).$ 
Therefore elements with the form $(f, 0)$ are in the image of $\Gamma_0.$

Let $f\in \mathrm{Aff}_+({\mathrm{T}}(A))$ and $m\in \N.$
Choose $m_0\ge 1$ such that $f(t)-m_0<0$ for all $t\in {\mathrm{T}}(A).$
Put $\gamma=m_0-f(t)\in \mathrm{Aff}_+({\mathrm{T}}(A)).$

We then borrow the proof of surjectivity  in 6.2.3 of \cite{Robert-Cu} but we also use \ref{LNote1}
with possibly nonzero $\omega(b).$

Choose $a_1\in M_{m_1}(A)$ such that $a_1=2\gamma.$
For each large $n\ge 2,$ $\gamma\ll (1+1/n)\gamma.$  Thus there exist exists 
$a_n\in \mathrm{M}_{m_n}(A)_+$ such that, for some $\dt_n>0,$
\beq
\gamma< \la(a_n-\dt_n)_+\ra \,\hat{} <\la a_n\ra \, \hat{}< (1+1/n)\gamma.
\eneq
Note that $A$ has strict comparison as ${\mathrm{Cu}}(A)=\mathrm{LAff}_+({\mathrm{T}}(A)).$
Therefore we may assume that $a_n\le a_1$ ($\la a_1\ra \, \hat{}=\gamma$). 
In particular, we may assume that $m_n=m_1,$ $n=1,2,....$
We may also assume that $m_1\ge m_0.$
We may further assume that $\|a_n\|=1,$ $n=1,2,...$ 

Since $a_n\in M_{m_1}(A)_+$ and $A$ is stably projectionless, we may assume 
that ${\mathrm sp}(a_n)=[0,1].$
Consider the commutative \SCA\, generated by $a_n$ and $1_{M_{m_1}}.$
Then it is isomorphic to $C([0,1]).$ Denote by $c_{\dt_n}$  a function in the \SCA\, 
which is zero at 1,
 strictly positive on $[0, \dt_n/2),$ zero elsewhere  and $\|c_{\dt_n}\|=1.$
Note $c_{\dt_n}\in M_{m_1}({\widetilde A})$ and $\pi(c_{\dt_n})=1_{M_{m_1}}$ (in $M_{m_1}(\C)$). 
 Let $g_n$ be also in the \SCA\, which is given by 
a  non-zero positive continuous function with support in $(\dt_n/2, \dt_n).$ Note that $g_n\not=0.$ 
We may assume that $\|g_n\|\le 1.$ 

Then 
\beq
\la c_{\dt_n}\ra +\la g_n\ra +\la (a-\dt_n)_+\ra \le m_1\la 1_{\widetilde A}\ra \le \la c_{\dt_n}\ra +\la a_n\ra.
\eneq
We compute that
\beq
-(1+1/n)\gamma < ( \la c_{\dt_n}\ra -m_1 \la 1_{\widetilde A}\ra)\,\hat{}<( \la c_{\dt_n}\ra -m_1 \la 1_{\widetilde A}\ra)\,\hat{}+\la g_n\ra\,\hat{}< -\gamma.
\eneq
Therefore
\beq
m_1 \la 1_{\widetilde A}\ra)\,\hat{}-(1+1/n)\gamma < \la c_{\dt_n}\ra \,\hat{}<\la c_{\dt_n}\ra \,\hat{}+\la g_n\ra\,\hat{}<m_1 \la 1_{\widetilde A}\ra)\,\hat{} -\gamma.
\eneq
Note that, for each $n,$  $\omega(c_{\dt_n})\le \gamma/n,$ since both $\gamma$ and $\la 1_{\widetilde A}\ra$ are continuous.
For each $n_k$ there exists $n_{k+1}>n_k$ such that
$7\gamma/{n_{k+1}}<\la g_{n_k}\ra \,\hat{}.$ 
Hence
\beq
\la c_{\dt_{n_k}}\ra <m_1 \la 1_{\widetilde A}\ra)\,\hat{} -\gamma-7\gamma/{n_{k+1}}.
\eneq

Therefore, there exists a subsequence $\{n_k\}$ such that
\beq
\la c_{\dt_{n_k}}\ra \, \hat{}+6\omega(c_{\dt_{n_{k+1}}})
&<&\la c_{\dt_{n_k}}\ra \,\hat{}+6\gamma/n_{k+1}\\
&<&m_1 \la 1_{\widetilde A}\ra)\,\hat{} -\gamma-\gamma/{n_{k+1}}
<\la c_{\dt_{n_{k+1}}}\ra \,\hat{},\hspace{0.7in}
\,\,\,k=1,2,....
\eneq

It follows from \ref{LNote1} that  $\la c_{\dt_{n_k}}\ra \le \la c_{\dt_{n_{k+1}}}\ra,$ $k=1,2,....$ 
Let $c\in {\mathrm{Cu}}({\widetilde A})$ such that $c=\sup\{c_{\dt_{n_k}}\}.$  
Since $c_{\dt_{n_k}}\in M_{m_1}({\widetilde A})$ and $\pi(c_{\dt_{n_k}})=1_{M_{m_1}},$
we conclude that $c\le 1_{M_{m_1}}$ and $\la \pi(c)\ra=m_1.$
We also have 
\beq
\la c\ra\, \hat{}= m_1 \la 1_{\widetilde A}\ra \,\hat{}-\gamma=(m_1-m_0)\la 1_{\widetilde A}\ra\, \hat{} +f.
\eneq
Note 
that
\beq
\Gamma_0(\la c\ra)=((m_1-m_0)\la 1_{\widetilde A}\ra\, \hat{} +f, m_1).
\eneq
If $m_0-m>0,$ then there exists $a_{00}\in M_l(A)_+$ for some $l\ge 1$ such that
$\la a_{00}\ra \, \hat{}=(m_0-m)\la 1_{\widetilde A}\ra \, \hat{}.$
Put $c_1=c_0\oplus a_{00}.$ If $m=m_0,$ keep $c=c_1.$ 
Then 
\beq
\Gamma_0(\la c_1\ra )=(f+(m_1-m)\la 1_{\widetilde A}\ra\, \hat{}, m_1).
\eneq
If $m_1=m,$ then $\Gamma_0(\la c_1\ra)=(f, m).$  
If $m_1-m>0,$ we have 
\beq
(m_1-m)\la 1_{\widetilde A}\ra \, \hat{}< (m_1-m)\la 1_{\widetilde A}\ra\, \hat{} +f.
\eneq
Since  $(m_1-m)\la 1_{\widetilde A}\ra\, \hat{} +f\in \Aff_+({\mathrm{T}}(A)),$ by \ref{Tnote3}, we conclude 
that 
\beq
(m_1-m)\la  1_{\widetilde A}\ra \le \la c_1\ra.
\eneq 
Since $(m_1-m)\la  1_{\widetilde A}\ra $ is represented by a projection, 
one has $c_2\in M_{m_1}({\widetilde A})_+$ such that
\beq
(m_1-m)\la  1_{\widetilde A}\ra+\la c_2\ra =\la c_1\ra.
\eneq
It follows that $\la \pi(c_2)\ra =m.$ Note 
that 
$\Gamma_0(\la c_2\ra)=(f, m).$ 
To see that we can choose $c_2$ so that it is not represented by a projection, 
choose an integer $k\ge 1$ such that $f>1/k$ on ${\mathrm{T}}(A).$ 
Choose $c_2'$ so that $\la c_2'\ra=(f-1/k, m)$ and $c_{2,0}\in (A\otimes {\cal K})_+$
such that $d_\tau(c_{2,0})=1/k.$ Now $c_2=c_2'\oplus c_2''$ cannot be represented by 
a projection but $\Gamma_0(c_2)=(f,m).$

Now let $f\in \mathrm{LAff}_+({\mathrm{T}}(A))$ and $m\in \N\cup\{\infty\}.$
Choose a sequence  $(f_n)$ in $\mathrm{LAff}_+({\mathrm{T}}(A))$ with $f_n\nearrow f$ and $m_n\nearrow m,$
where $m_n<\infty,$ $n=1,2,....$
As in the previous paragraph, choose  $x_n\in S({\widetilde A})$ with $\Gamma_0(x_n)=(f_n, m_n)$
such that they are not represented by projections. 
By what  has been proved,  $x_n\le x_{n+1},$ $n=1,2,....$  Put $x=\sup\{x_n\}.$
Then it is easy to check that $\Gamma_0(x)=(f, m).$ 

This shows that $\Gamma_0$ is surjective. 

For the last part of the statement, let $\la \pi(a)\ra \le \la \pi(b)\ra.$ 
Suppose that that $y=\la b\ra $ is  represented by 
a projection $p$ and $x=\la a\ra$ is not represented by  a projection,
and $\la a\ra\, \hat{} \le \la b\ra\, \hat{}.$ 
Then, for any $\ep>0,$ 
\beq
\la (a-\ep)_+\ra \, \hat{}<\hat{y}
\eneq
Then since $\hat{y}$ is now continuous, by \ref{Tnote3}, 
\beq
\la (a-\ep)_+\ra \le y.
\eneq
It follows that $x\le y.$

Now suppose that $x$ is represented by a projection and $\hat{x}<\hat{y}.$
If $y$ is also represented by a projection, then by \ref{Tnote3},  $x\le y.$

It remains  to check the case that $\la a\ra$ is represented by a projection and 
$y$ is not, and $\la a\ra \,\hat{}<\hat{y},$ as well as 
$\la \pi(a)\ra \le {\mathrm{Cu}}(\pi)(y).$  Note that since $\la a\ra $ is represented by a projection, 
$\la \pi(a)\ra <\infty.$ In this case, there exists an increasing sequence 
$(\la b_n\ra)$ in  $ \Aff_+({\mathrm{T}}(A))$ such that
$y=\sup\{\la  b_n\ra\}.$
Since $\la a\ra\, \hat{}$ is continuous, one finds $b_n$ such that 
$\la a\ra \, \hat{}<\la b_n\ra \, \hat{}$ for some large $n.$  We may also assume that $\la \pi(b_n)\ra 
\ge \la a\ra \, \hat{}.$
Now $\la b_n\ra\, \hat{}\in \Aff_+({\mathrm{T}}(A)).$ From what has been proved, $\la a\ra \le y.$ 
This completes the proof.

\end{proof}

\begin{cor}\label{Cwup}
Let $A$ be a stably projectionless simple \CA\, such that ${\mathrm{M}}_r(A)$  almost has  stable rank one
for all $r\ge 1, $   $\mathrm{QT}(A)=\mathrm{T}(A),$ $A$ has strict comparison for positive elements and has 
continuous scale,   and 
${\mathrm{Cu}}(A)= \mathrm{LAff}_+({\mathrm{T}}(A)).$
Then $K_0(A)$ has the following 
property: for
any $x\in K_0(A),$ there exists $\tau\in {\mathrm{T}}(A)$ such that 
$\rho_A(x)(\tau)=0.$

\end{cor}

\begin{proof}
Since $A$ is stably projectionless, by \cite{BC}, $A$ is  stably finite (see also Theorem 1.2 of \cite{LZ}). 
It follows from \cite{BR} that ${\mathrm{T}}(A)\not=\O.$
Let  $x=[p]-[q],$  
where $p, q\in {\mathrm{M}}_r({\widetilde A})$ are projections such that $[ \pi(p)]=[\pi(q)],$
where $\pi: {\mathrm{M}}_r({\widetilde A})\to {\mathrm{M}}_r(\C)$ is the quotient map.
Suppose that $\rho_A(x)(\tau)>0$ for all $\tau\in {\mathrm{T}}(A).$ Then, 
\beq
\tau(p)
>\tau(q)
\rforal \tau\in {\mathrm{T}}(A).
\eneq
By Theorem \ref{Tnote2},
\beq
q
\lesssim  
p.
\eneq
Thus, there is a projection $p'\le p$ such that $[p']=[q].$ 
Put $P=p-p'.$  Then $P$ is a non-zero projection in $M_{r}({\widetilde A}),$  as $\tau(P)>0$
for all $\tau\in {\mathrm{T}}(A).$ 
Since $\pi(p')\le \pi(p)$ and $[\pi(p')]=[\pi(q)]=[\pi(p)],$ \wilog,  we may assume 
that 
\beq
\pi(P)=0.
\eneq
This implies that $P\in M_{r}(A),$ which is impossible.  {{By considering $-x,$ we  conclude 
that it is also impossible to have $\rho_A(x)(\tau)<0$ for $\tau\in T(A).$}} 

{{If there were  no $\tau$ such that $\rho_A(x)(\tau)=0,$ then there would be $\tau_1, \tau_2\in \mathrm{T}(A)$  such that $\rho_A(x)(\tau_1)=t_1>0,$  
$\rho_A(x)(\tau_2)=t_2<0.$  
Then $0<\af:=t_2/(t_2-t_1)<1.$ 
Put $\tau=\af \tau_1+(1-\af)\tau_2\in \mathrm{T}(A).$
Then $\rho_A(x)(\tau)=0.$  This implies there is $\tau\in  \mathrm{T}(A)$ such that $\rho_A(x)(\tau)=0.$}} 
\end{proof}

\begin{cor}\label{A1Cinv}
Let $A$ be a separable,  exact, ${\cal Z}$-stable  simple \CA, 
where ${\cal Z}$ is the Jiang-Su algebra.  Suppose that $x\in K_0(A)$ is such that 
$\tau(x)>0$ for all non-zero traces $\tau$ of $A.$ Then $x$ is represented by a projection $p\in A\otimes {\cal K}.$
\end{cor}

\begin{proof}
Since $A$ is assumed to be exact, $\mathrm{QT}(A)={\mathrm{T}}(A).$  Also, since $A$ is ${\cal Z}$-stable, by Lemma 6.5 of \cite{ESR-Cuntz},  $\mathrm{Cu}(A)=\text{LAff}_+({\tilde T}(A)).$ It follows from \cite{Rob-0} that ${\mathrm{M}}_n(A)$ almost has  stable rank one
as ${\mathrm{M}}_n(A)$ is ${\cal Z}$-stable.  Moreover, there is a non-zero $a\in {\mathrm{ Ped}}(A)_+$ (see 5.2 of \cite{eglnp}) such that 
$C=\overline{aAa}$ has continuous scale.  By  Brown's theorem (\cite{Br1}), $C\otimes {\cal K}\cong A\otimes {\cal K}.$
 It follows from \ref{Cwup} that we may assume that 
 $A\otimes {\cal K}$ has a non-zero projection $e.$ Then, by Brown's theorem (\cite{Br1}) again,  $A\otimes {\cal K}\cong B\otimes {\cal K},$
where $B$ is the hereditary \SCA\, generated by $e.$  Now since $B$ is unital and $B\otimes {\cal K}$ is ${\cal Z}$-stable,
by \cite{GJS} (see also 4.6 of \cite{Ror-Z-stable}), $K_0(B)$ is weakly unperforated. Thus $x>0$ and it is represented by a projection.

\end{proof}

\begin{cor}\label{Cfulllcomparison}
Let $A$ be a stably projectionless simple \CA\, such that ${\mathrm{M}}_r(A)$  has almost stable rank one
for all $r\ge 1,$   $\mathrm{QT}(A)=\mathrm{T}(A),$ $A$ has 
finitely many extremal traces, 
 $\mathrm{Ped}(A)=A,$ and 
${\mathrm{Cu}}(A)= \mathrm{LAff}_+({\mathrm{T}}(A)).$
Then ${\mathrm{Cu}}({\widetilde A})=V({\widetilde A})\sqcup L({\widetilde A}).$

\end{cor}

\begin{proof}
Since $\mathrm{Ped}(A)=A$ and $A$ has finitely many extremal traces,
$d_\tau(e_A)$ is continuous  on ${\mathrm{T}}(A)$ for any strictly positive element $e_A\in A.$  Since 
$A$ has strict comparison, $A$ has continuous scale (see the proof 5.4 of \cite{eglnp}, for example). 
Since ${\mathrm{T}}(A)$ has only finitely many extreme points, any finite affine function
on ${\mathrm{T}}(A)$ is continuous, and so, for any integer $r\ge 1, $ and any 
$a\in {\mathrm{M}}_r({\widetilde A}),$ $\la a\ra\, \hat{} \in \Aff_+({\mathrm{T}}(A)).$ Therefore, if $x\in {\mathrm{Cu}}({\widetilde A})$ 
is not represented by a projection, then $x\in S({\widetilde A}).$ In other words, 
${\mathrm{Cu}}({\widetilde A})=V({\widetilde A})\sqcup S({\widetilde A}).$ Thus, the corollary  follows from \ref{Tnote2}.
\end{proof}

\begin{cor}\label{Tnote2rk1}
Let $A$ be a stably projectionless simple \CA\,  with continuous scale
and with stable rank one such that  $\mathrm{QT}(A)=\mathrm{T}(A),$ 
and 
${\mathrm{Cu}}(A)= \mathrm{LAff}_+({\mathrm{T}}(A)).$
Then ${\mathrm{Cu}}({\widetilde A})_+=S({\widetilde A})$ and $\Gamma_0$ is an ordered semigroup 
isomorphism from ${\mathrm{Cu}}_+({\widetilde A})$ onto $L({\widetilde A}).$ 
Moreover, ${\mathrm{Cu}}({\widetilde A})=V({\widetilde A})\sqcup L({\widetilde A}).$
\end{cor}

\begin{rem}\label{RRR}

Let $A$ be a stably projectionless simple \CA\, which is ${\cal Z}$-stable such that
$\mathrm{QT}(A)={\mathrm{T}}(A).$ 
Then, by \cite{Rob-0}, ${\mathrm{M}}_r(A)$ (for all $r\ge 1$) almost has stable rank one.
A combination of \cite{Ror-Z-stable},
\cite{RW}, and \cite{ESR-Cuntz}
shows that $A$ also satisfies the rest of 
the
 conditions of \ref{Tnote2}.

 There are several other  immediate consequences   of \ref{Tnote2} and related facts about $\mathrm{Cu}^\sim$ (see 
 \cite{Robert-Cu}).
Let $A$ be as \ref{Tnote2}. 

(i) Then the canonical map $\iota_{0,A}: {\mathrm{Cu}}(A)\to \mathrm{Cu}^\sim (A)$ 
is injective. To see this, let $\la a\ra, \la b\ra \in \mathrm{Cu}(A).$ 
If $\la a \ra+k[1_{\tilde A}]=\la b\ra +k[1_{\tilde A}],$ then $\la a\ra\hat{}=\la b\ra\hat{}.$  Since $\mathrm{Cu}(A)=\text{Laff}_+({\mathrm{T}}(A)),$
 $\la a\ra=\la b\ra.$

(ii) Let $x_n\in S({\widetilde{A}})$ with $x_n\le x_{n+1},$ $n=1,2,....$ Then 
$\sup_n x_n\in S(\widetilde{A}).$ This follows from the definition immediately. 

(iii) As indicated in the proof of \ref{Tnote2}, if $\la p\ra\in V(\widetilde{A})$ and $x\in S(\widetilde{A})\setminus \{0\},$ then 
$\la p\ra +x\in S(\widetilde{A}).$

(iv) Denote by $S^\sim(A)=\{\la a\ra -\la\pi(a)\ra\cdot [1_{\tilde A}]: \la a\ra \in S({\widetilde{A}}),\,\la \pi(a)\ra <\infty\}$ as a sub-semigroup 
of $\mathrm{Cu}^\sim(A)$ (see \cite{Robert-Cu}). Then, by \ref{Tnote2}, $\mathrm{Cu}(A)\subset S^\sim(A).$ 

Let $x=\la a\ra, y=\la b\ra \in S(\widetilde{A})$ such that $\la \pi(a)\ra =n$ and $\la \pi(b)\ra=m,$ 
where $n,m$ are nonnegative integers. 
Suppose that $\hat{x}-n=\hat{y}-m.$ Then $\hat{x}+m=\hat{y}+n$ and 
$\la \pi(a)\ra +m=n+m=\la \pi(y)\ra +n.$ If $x=0,$ by (iii), $y=0.$ Let us  assume neither are zero. 
It follows from (iii) that  $x+m\la 1\ra$ and $y+n\la 1\ra$ are not represented 
by projections. By \ref{Tnote2}, $x+m\la 1\ra=y+n\la 1\ra.$ It follows that 
$x-n\la 1\ra=y-m\la 1\ra$ in $\mathrm{Cu}^\sim(A).$ 
Therefore  we may write $S^\sim(A)=\text{LAff}_+^\sim({\mathrm{T}}(A))=\{f-g: f\in \text{LAff}_+({\mathrm{T}}(A)), g\in \text{Aff}_+({\mathrm{T}}(A))\}$ 
(see \cite{Robert-Cu} for the notation).

(v) Let $\iota_A^\sim: {\mathrm{Cu}}^\sim (A)\to {\mathrm{Cu}}^\sim ({\widetilde{A}})$ be the natural 
map. Then, by \ref{Tnote2},  $\iota_A^\sim$ is injective on $S^\sim (A).$ 
In fact, let $x-n\la 1\ra, y-m\la 1\ra \in S^\sim(A),$ 
such that, for some integer $k\ge 0,$ 
\beq
x+m\la 1\ra +k\la 1\ra =y+n\la 1\ra +k \la 1\ra\in S(\widetilde{A})\subset \mathrm{Cu}({\widetilde{A}}).
\eneq
Then
\beq
 \hat{x}+m+k =\hat{y} +n+k\andeqn \la \pi(a)\ra+m+k=\la \pi(b)\ra +n+k.
\eneq
It follows that 
\beq
 \hat{x}+m=\hat{y} +n\andeqn \la \pi(a)\ra+m=\la \pi(b)\ra +n.
\eneq
As (iv), we may assume neither $x$ nor $y$ are zero.
Since $x+m\la 1\ra$ and $y+n\la 1\ra$ are not represented by projections, by  \ref{Tnote2},  $x+n\la 1\ra=y+m\la 1\ra.$ Thus
$x-n\la 1\ra=y-m\la 1\ra$ in $S^\sim(A)\subset \mathrm{Cu}^\sim(\widetilde{A}).$

(vi) 
Exactly the same argument shows that $S(\widetilde{A})$ maps to $\mathrm{Cu}^\sim(\widetilde{A})$ injectively.
Let us denote this map by $\iota_S^\sim.$  
Let us also set $S^\sim(\widetilde{A})=\{x-n\la 1_{\tilde A}\ra: x\in S(\widetilde{A}), n\in \N\cup\{0\}\}\subset \mathrm{Cu}^\sim(\widetilde{A}).$  So $S(\widetilde{A})\subset S^\sim(\widetilde{A}).$

(vii) Note that $V({\widetilde A})\sqcup S(\widetilde{A})$ maps into $K_0({\widetilde{A}})\sqcup S^\sim({\widetilde{A}}).$
Let us identify $\N\subset V(\widetilde{A})$ with $\N\cdot \la 1_{\tilde A}.\ra$ 
The map above maps $\N\sqcup S(\widetilde{A})$ into $\Z\sqcup S^\sim({\widetilde{A}})$ injectively.

\end{rem}


\subsection{An existence theorem and some uniqueness theorems}
\vspace{0.2in}

$$
\hspace{-2in}\text{The following is a variation of  a result of Pedersen and  R\o rdam (\cite{Pedjot87})}.
$$
\begin{lem}\label{LRordamu}
Let $A$ be a non-unital \CA\, and let $x\in A$ 
and $1>\dt>\bt>\gamma>0.$ 
Suppose that there exists $y\in GL({\widetilde A})$ such that
$\|x-y\|<\gamma.$ 
Then there is a unitary $u\in {\widetilde A}$ with the form $u=1+z,$
where $z\in A,$ such that
\beq
uf_{\dt}(|x|)=vf_{\dt}(|x|),
\eneq
where $x=v|x|$ is the polar decomposition of $x$ in $A^{**}.$ 
\end{lem}

\begin{proof}
This is a  modification of the proof of Pedersen in \cite{Pedjot87}.
We will follow the proof  and keep the notation of  \cite{Pedjot87} and point out where to make the changes.

In Lemma 1 of \cite{Pedjot87}, write $A=\lambda+A',$ where $A'\in {\mathfrak A}$ for some 
$\lambda\not=0.$  Let $\pi: {\tilde{\mathfrak A}}\to \C$ denote the quotient map with ${\mathrm ker}\pi={\mathfrak A}.$
If $T\in {\mathfrak A}$ and $\|T-A\|<\gamma,$ then $\|\pi(A)\|<\gamma.$ 
It follows that $|\lambda|<\gamma.$ 
There is a continuous path $\{g_1(t): t\in [|\lambda|, \gamma]\}$
such that $g_1(|\lambda|)=\lambda^{*-1},$ $g_1(\gamma)=\gamma^{-1}$
and  $|g_1(t)|\ge \gamma^{-1}.$
We define a complex valued function 
$g'\in C([0,\|A\|])$ as follows:
\beq
g'(t)=\begin{cases} \lambda^{*-1} & \text{if}\, \,t\in [0, |\lambda|];\\
                             g_1(t) & \text{if} \,\, t\in (\lambda, \gamma];\\
                             t^{-1} & \text{if}\,\,t\in (\gamma, \infty).\end{cases}
                             \eneq
This $g'$ will replace the function $g$ in the proof of Lemma 1 of \cite{Pedjot87}.
Put
\beq
B=(g')^{-1}(|T^*|)A^{*-1}(1-f)(|T|)+Vf(|T|)
\eneq
with $f$  as described in \cite{Pedjot87}.
Note that $fg'(t)=0$ if $t\in [0,\gamma],$ $fg'(t)=t^{-1}$ if $t\in [\bt, \infty),$
and $tf(t)g'(t)=f(t)$ if $t\in [0, \infty).$ Exactly as in the proof of \cite{Pedjot87},  one has $BE_\bt=VE_\bt.$ 
 Set $C=f(|T|)- A^* Vg' (|T|) f(|T|).$    
 We still have 
 $f(|T|)=|T|V^*Vg'(|T|)f(|T|).$ Therefore 
 \beq
 C=f(|T|)- A^* Vg' (|T|) f(|T|)=(T^*-A^*)V(fg')(|T|).
 \eneq
 
   The same estimate yields
 \beq
 \|C\|\le \|T^*-A^*\|\|fg'\|_{\infty}\le \|T-A\|\gamma^{-1}<1.
 \eneq
As in the proof of Lemma 1 of \cite{{Pedjot87}}, this implies that $B$ defined above is invertible
and $BE\bt =VE_\bt.$ 
Note that $\pi(B)=\lambda^* \lambda^{*-1}=1.$ In other words $B=1+z'$ for some $z\in {\mathfrak A}.$ 
As in \cite{Pedjot87}, $B$ also satisfies the conclusion of Lemma 2 of \cite{Pedjot87}, i.e., 
$F_\bt B^{*-1}=F_\bt V.$    

Define $h(t)=(t-\bt)\vee 0.$  Then $Bh(|T))=BE_\bt h(|T|)=VE_\bt h(|T|)=Vh(|T|).$
Let $A_0$ be as in Lemma 3 of \cite{Pedjot87} with $B$ defined above. Then  the
conclusion of Lemma 3 of \cite{Pedjot87} holds. 

Then, as in Lemma 4 of  \cite{Pedjot87}, one obtains $B_0\in {\tilde {\mathfrak A}}$ with $\pi(B_0)=1$ defined as $B$ defined 
with $g_0'$ instead of $g_0$ as we demonstrated above.    The same computation provides
\beq
B_0-Vf_0(|T|)&=&(g_0')^{-1}(|T^*|)A_0^{*-1}(1-f_0)(|T|)\\
&=&(g_0')^{-1}(|T^*|)B^{*-1}(h+\ep)^{-1}(|T|)(1-f_0)(|T|).
\eneq
Exactly as in the proof of Lemma 4 of \cite{Pedjot87}, we  have 
$B_0E_\dt =F_\dt B_0=F_\dt V= VE_\dt.$ 
Since   $B_0$ is invertible, we have the polar decomposition $B_0=U|B_0|$ 
in ${\tilde {\mathfrak A}}.$
Note that $\pi(U)=1$ since $\pi(B_0)=1.$ 
{{Hence}}
$U=1+z$ for some $z\in {\mathfrak A}.$
As in Theorem 5 of \cite{Pedjot87}, $UE_\dt=VE_\dt.$      
Then $Uf_\dt(|T|)=Vf_\dt(|T|).$ The lemma follows.

\end{proof}

\begin{cor}\label{PeduC}
Let $A$ be a non-unital \CA\, which almost has stable rank one.
Then, for any $x\in A$ and any $\ep>0,$ there is a unitary $u\in {\widetilde A}$ 
with form $u=1+y$ for some $y\in A$ such that
\beq
\|u|x|-x\|<\ep.
\eneq
\end{cor}

\begin{proof}
We have 
$\|f_{\ep/8}(|x|)x-x\|<\ep/4.$ 
Since $A$ almost has  stable rank one, by \ref{LRordamu}, there exists a unitary $u\in {\widetilde A}$ with the form 
$u=1+z$ for some $z\in A$ such that
\beq
\|uf_{\ep/8}(|x|)-vf_{\ep/8}(|x|)\|<\ep/4,
\eneq
where $x=v|x|$ is the polar decomposition in $A^{**}.$
It follows that
\beq
\|u|x|-x\|&<&\|u|x|-uf_{\ep/8}(|x|)|x|\|+\|uf_{\ep/8}(|x|)|x|-vf_{\ep/8}(|x|)|x|\|\\
&&+\|vf_{\ep/8}(|x|)|x|-v|x|\|<\ep.
\eneq

\end{proof}

\begin{lem}{\rm (Theorem 3.3.1 of \cite{Robert-Cu})}\label{LRobert}
Let 
$B$ be a  simple \CA\,  
which has almost stable rank one. 
Then, for any finite subset ${\cal F}$ and $\ep>0,$  there exists a finite subset ${\cal G}\subseteq 
{\mathrm{Cu}}(C)$  such that, for any two \hm s $\phi_1, \phi_2: C:=C_0((0,1])\to B,$ 
if 
\beq\label{LRa-1}
\hspace{0.3in}{\mathrm{Cu}}(\phi_1)(f)\le {\mathrm{Cu}}(\phi_2)(g) \andeqn\, {\mathrm{Cu}}(\phi_2)(f)\le {\mathrm{Cu}}(\phi_1)(g)
\tforal f, g\in {\cal G}\,\,\,\text{with}\,\,\, f\ll g,
\eneq
there exists a unitary $u\in {\widetilde B}$ such that
\beq
\|u^*\phi_2(f)u-\phi_1(f)\|<\ep\tforal f\in {\cal F}.
\eneq
\end{lem}

\begin{proof}
The lemma is based on the fact that $\phi$ and $\psi$ are approximately unitarily equivalent 
if ${\mathrm{Cu}}(\phi)={\mathrm{Cu}}(\psi)$ (see the proof of ``(iii) implies (i)" in ``Proof of Theorem 1.3" in \cite{Rob-0}).

The actual proof is almost the same as that of 3.3.1 of \cite{Robert-Cu}.
Let us present the details.

Let us point out what is the difference.
In the proof of 3.3.1 of \cite{Robert-Cu},  consider $(b_G)\in \prod_G B_G$ and let $\ep>0.$ 
Since each $B_G$ almost  has stable rank one, by \ref{PeduC}, there is $u_G\in {\widetilde B_G}$  
 such that $u_G=1+z_G$ for some $z_G\in B_G$ with $\|z_G\|\le 2$ and 
\beq
\|u_G|b_G|-b_G\|<\ep.
\eneq
Note that $(u_G)\in  1+\prod_GB_G.$ Since elements with polar decomposition, in the sense of being 
the (non-unique) product of a unitary and a positive element, are in the closure 
of the invertible elements, this implies that (with notation 
as
in the proof of 3.3.1 of \cite{Robert-Cu}) both $\prod_G B_G$ and 
$B$ almost  have stable rank one. 
The rest of the proof then can proceed  just as 
in the  proof 
of
3.3.1 of \cite{Robert-Cu}
(note that we only compute ${\mathrm{Cu}}(\phi)$ and ${\mathrm{Cu}}(\psi)$ which is easier).
\end{proof}

\begin{rem}
A direct proof of the 
lemma
above could also be obtained using  \cite{Pedjot87} 
directly.
\end{rem}

\begin{cor}\label{Cuniqcone}
Let $C=C_0((0,1])$ and  let
$\Delta: C^{q,{\mathbf 1}}\setminus \{0\}\to (0,1)$ be an order preserving map.
Then, for any $\ep>0$ and any finite subset
${\cal F}\subseteq C,$ there 
exist
a finite subset ${\cal H}_1\subseteq C_+^{\mathbf 1}\setminus \{0\},$ a finite subset ${\cal H}_2\subseteq C_{s.a.},$
and
$\gamma>0$  satisfying the following condition:
for any
 two \hm s $\phi_1, \phi_2: C\to A$ for some $A$
which is separable, simple, exact, stably projectionless, 
and
has 
continuous scale, almost stable rank one,
and the property that the map ${\mathrm{Cu}}(A)\to {\mathrm LAff}_+(\mathrm{T}(A))$ is an 
ordered semigroup isomorphism  such that
\beq\label{Lrluniq-1}
&&\tau(\phi_i)(a)\ge \Delta(\hat{a})\tforal a\in {\cal H}_1\tand \tforal \tau\in \mathrm{T}(A)\tand\\
&&|\tau(\phi_1(b))-\tau(\phi_2(b))|<\gamma\tforal b\in {\cal H}_2\tand \tforal \tau\in \mathrm{T}(A),
\eneq
there exists a unitary $u\in {\widetilde A}$ such that
$$
\|u^*\phi_2(f)u-\phi_1(f)\|<\ep\rforal f\in {\cal F}.
$$

\end{cor}

\begin{proof}
The proof of this is the combination of \ref{LRobert} and the proof of 7.8 of \cite{eglnp}.

\end{proof}

\begin{rem}\label{R1}
Let $\af, \bt: {\mathrm{Cu}}(C_0((0,1]))\to {\mathrm{Cu}}(A)$ be two morphisms in ${\mathbf{Cu}}.$ 
Recall  the pseudo-metric ${\mathrm{d}}_w$ introduced in \cite{CE}:
\beq 
\hspace{0.3in}{\mathrm{d}}_w(\af, \bt)=\inf \{r\in \R^+: \af(\la e_{t+r}\ra )\le \bt(\la  e_t\ra )\andeqn
\bt(\la e_{t+r}\ra )\le \af( \la e_t\ra ),\,t\in \R^+\},
\eneq
where $e_t(x)=(x-t)_+$ is a function on $(0,1].$ 

If $\phi, \psi: C_0((0,1])\to A$ are two \hm s, define ${\mathrm{d}}_w(\phi, \psi)={\mathrm{d}}_w({\mathrm{Cu}}(\phi), {\mathrm{Cu}}(\psi)).$
  Let $J\subseteq (0,1]$ be any  relatively open  interval $(\af, \bt)\cap (0,1].$ 
Define, for each $r>0,$ $J_r=\{t\in (0,1]: {\mathrm{dist}}(J, t)<r\}.$
For each $J$ fix a positive function $e_J$ which is strictly positive on $J$ and zero elsewhere. 
To be more symmetric than the definition of ${\mathrm{d}}_w,$ one can also defines the following metric:
\beq
\hspace{0.2in}D_w(\phi, \psi)=\inf\{r\in \R^+: \phi(e_{J})\lesssim \psi(e_{J_r}),\,\,\, \psi(e_J)\lesssim \phi(e_{J_r}),\,\, J\subseteq (0,1]\}.
\eneq
(see some related discussion in \cite{HL}).   Then $D_w$ is a metric (see the proof of Proposition 2 of \cite{RS}). 
If $\mathrm{Cu}(A)$ has the weak cancellation,  $d_w$  is a metric (Proposition 2 of \cite{RS}), and $d_w$ and $D_w$ are equivalent.

 Another version of \ref{LRobert} can be stated as follows:

({\bf A}): For any $\ep>0$ and any finite subset ${\cal F}\subseteq C,$ there exists $\dt>0$ with the following property:
if
$D_w(\phi, \psi)<\dt,$ then there exists a unitary $u\in {\widetilde A}$ such that
\beq
\|u^*\phi(f)u-\psi(f)\|<\ep\rforal f\in {\cal F},
\eneq
and, if, furthermore, $\mathrm{Cu}(A)$ has weak cancellation, $D_w(\phi, \psi)<\dt$ can be replaced by
$d_w(\phi, \psi)<\dt$ (with possibly a different $\dt$). 

Suppose that $A$  is  a stably projectionless simple \CA\, with ${\mathrm{T}}(A)\not=\O.$
Consider any  $x+z\ll y+z$ for $x, y, z\in {\mathrm{Cu}}(A),$
where $x\not=y.$    Suppose that $b\in (A\otimes {\cal K})_+$ is  such 
that $\la b\ra =y+z$ and $0\le b\le 1.$  Then, for any $1/2>\dt>0, $  $f_{\dt/2}(b)-f_\dt(b)>0.$ 
Therfore $d_\tau(x+z)<d_\tau(y+z)$ for all $\tau\in {\mathrm{QT}}(A).$
{{Thus}} $d_\tau(x)<d_\tau(y)$ for all $\tau\in {\mathrm{QT}}(A).$ 
If 
 $A$ is also assumed to have strict comparison, then 
 $x\le y.$ This implies that ${\mathrm{Cu}}(A)$ has 
weak cancellation. 
 As shown in Proposition 2 of \cite{RS}, ${\mathrm{d}}_w$ is 
then
a metric.

\end{rem}

\begin{prop}\label{PappI}
Let $C=C_0((0,1]).$ 
Then, for any $\ep>0,$ any $\sigma>0,$ and any finite subset ${\cal F}\subseteq C,$ there exists $\dt>0$
satisfying the following
condition:
Suppose that  $A$ is a  stably projectionless simple \CA\, with continuous scale which 
almost has stable rank one and suppose that  $\phi, \psi: C\to {\widetilde A}$ are  \hm s.
If ${\mathrm{d}}_w(\phi, \psi)<\dt,$ then there exists $\psi': C\to {\widetilde A}$ such that
$\pi\circ \psi'=\pi\circ \phi,$ 
\beq
\|\psi'(f)-\psi(f)\|<\ep\tand {\mathrm{d}}_w(\phi, \psi')<\sigma,
\eneq
where $\pi: {\widetilde A}\to \C$ is the quotient map.
\end{prop}

\begin{proof}
Let $\iota: (0,1]\to (0,1]$ denote the identity map which we view as a generating element of $C_0((0,1]).$ 
Fix $0<\eta<\ep/2$ and a finite subset ${\cal G}\subseteq C_0((0,1]).$ There exists $\dt'>0$ such that,
if $|t-t'|<\dt',$  then
\beq
\|g(t)-g(t')\|<\eta\rforal \, g\in {\cal G}.
\eneq

If ${\mathrm{d}}_w(\phi, \psi)<\dt',$ then it is easy to see that $\|\pi(\phi(\iota))-\pi(\psi(\iota))\|<\dt'.$
Let $\lambda_1, \lambda_2\in (0,1]$ such that 
$\pi(\phi(\iota))=\lambda_1$ and $\pi(\psi(\iota))=\lambda_2.$ Then $|\lambda_1-\lambda_2|<\dt'.$
There exists a continuous map $j: (0,1]\to (0,1]$ such that
$$
|j(t)-t|<\dt'\andeqn j(\lambda_2)=\lambda_1.
$$
Define $\psi': C_0((0,1])\to {\widetilde A}$ by $\psi'(f)=\psi(f\circ j).$ 
Then $\pi(\psi'(\iota))=\lambda_1=\pi(\phi(\iota)).$ Moreover
\beq
\|\psi(g)-\psi'(g)\|=\|\psi(g-g\circ j)\|<\eta\rforal g\in {\cal G}.
\eneq
If $\sigma>0$ is given one can choose large ${\cal G}$ and sufficiently small $\eta$ so that
\beq
{\mathrm{d}}_w(\psi(f), \psi'(f))<\sigma/2.
\eneq
We also can choose $\dt=\min\{\dt', \sigma/2\}.$
\end{proof}

\begin{rem}\label{RR2}
Let $I=(\af, \bt]$ (or $I=[\af, \bt)$).  
Let $A$ be  a stably projectionless simple 
\CA\, 
which almost has stable rank one. 
Fix a homeomorphism $h_I: (\af, \bt]\to (0,1]$  given by $h(t)={t-\af\over{\bt-\af}}$ for $t\in (\af, \bt],$
or ($h_I(t)= {\bt-t\over{\bt-\af}}$ for $t\in [\af, \bt)$.)
If $\phi: C_0(I)\to A$ is a \hm,  denote by $\phi\circ h_I^*: C_0((0,1])\to A$ the \hm\, defined by $\phi\circ h_I^*(f)=\phi(f\circ h)$
for all $f\in C_0((0,1]).$ 

Suppose now there are two \hm s $\phi, \psi: C_0(I)\to A.$ 
Define
\beq
{\mathrm{D}}_{w, I}(\phi,\psi)=D_w(\phi\circ h_I^*, \psi\circ h_I^*)\andeqn
\mathrm{d}_{w,I}(\phi, \psi)=d_w(\phi\circ h_I^*, \psi\circ h_I^*).
\eneq
Put  $\iota_I'(t)=t-\af,$ if $I=(\af,, \bt]$ and $\iota_I'(t)=\bt-t,$ if $I=[\af,\bt).$
Now assume that $0<\bt-\af\le1.$ 
Put $f_I(t)=(\bt-\af)t\in C_0((0,1]).$ 
Then $\iota'_I=f_I\circ h.$
Let ${\cal F}\subset C_0(0,1]).$ 
 Then $g\circ \iota'_I=g\circ f_I\circ h$ for each $g\in {\cal F}.$
Let $\lambda\in (0,1].$ Define $f_\lambda(t)=\lambda t$ for $t\in (0,1].$

For any $\ep>0,$ there is a finite subset $K\in (0,1]$ such that, for any $0<\bt-\af\le 1,$ 
with $I=(\af, \bt]$ or $I=[\af, \bt),$  for each $g\in {\cal F},$ 
$\|g\circ \iota'_I-g\circ f_\lambda\circ h\|<\ep/2$ for some $\lambda\in K.$ 
Let ${\cal G}_{\cal F,\ep}=\{g\circ f_\lambda: g\in {\cal F},\,\, \lambda\in K\}.$ 
Then, by ({\bf A}) of \ref{R1}, we have the following:

({\bf B}): Let $\ep>0,$  let ${\cal F}\subset C_0((0,1])$ be a finite subset, 
Let $\dt>0$ be given by ({\bf A}) in \ref{R1} for $\ep/2$ and ${\cal G}_{\cal F, \ep}.$ 
Suppose $I=(\af, \bt]$ or $I=[\af, \bt)$ with $0<\bt-\af<1.$ 
Then, for any \hm s $\phi, \psi: C_0(I)\to A$ such that $\mathrm{D}_{w, I}(\phi, \psi)<\dt,$
there exists a unitary $u\in {\widetilde A}$ such that
\beq
\|u^*\psi(f)u-\phi(f)\|<\ep\rforal f\in \{\iota'_I, g\circ \iota'_I: g\in {\cal F}\}.
\eneq
Furthermore, if $\mathrm{Cu}(A)$ has weak cancellation, $\mathrm{D}_{w,I}$ above could be replaced by $\mathrm{d}_{w,I}.$

\end{rem}

\begin{lem}\label{Lconeuniq}
Let 
$A$ be 
a
stably projectionless simple \CA\, with continuous scale which 
 almost has stable rank one.  Suppose also that $A$ has strict comparison for 
positive elements and that $\mathrm{QT}(A)=\mathrm{T}(A).$
For any $\ep>0$ and any finite subset ${\cal F}\subseteq C,$ there exists $\dt>0$  satisfying the following condition:
If  $\phi, \psi: C_0((0,1])\to {\widetilde A}$ are two \hm s such that
\beq\label{28a}
{\mathrm{d}}_w(\phi, \psi)<\dt,
\eneq
then there exists a unitary $u\in {\widetilde A}$ such that
\beq\label{28ab}
\|u^*\psi(f)u-\phi(f)\|<\ep\rforal f\in {\cal F}.
\eneq

Moreover, 
\eqref{28ab} 
also holds, without assuming $A$ has strict comparison, on  replacing 
\eqref{28a} by 
\beq\label{28b}
D_w(\phi, \psi)<\dt.
\eneq
\end{lem}

\begin{proof}
By 
\ref{PappI}, \wilog, we may assume that 
$\pi\circ \phi=\pi\circ \psi, $ where $\pi: {\widetilde A}\to \C$ is the quotient map.
Let $\iota: (0,1]\to (0,1]$ denote the identity map which we view as a generator of $C_0((0,1]).$
Fix $\ep>0.$
Put $a=\phi(\iota)$ and $b=\psi(\iota).$ It suffices to establish the case that ${\cal F}=\{\iota\}.$

 Choose $\sigma_0=\ep/256.$ 
 We will prove the  last part of the statement first. 
 The first part will follows, since, by \ref{R1}, with the assumption that $A$ has strict comparison, 
 $\mathrm{Cu}(A)$ has weak cancellation.  It follows that $\mathrm{d}_w$ and $\mathrm{D}_w$ are 
 equivalent. 
Let ${\cal F}_0'=\{\iota, f_{\sigma_0}\}.$

By ({\bf B}) of \ref{RR2}, we obtain $1/2>\dt_0>0$ 
with
the following property: for any 
interval $I=[\af, \bt)$ or $I=(\af, \bt]$  with $0<\bt-\af\le 1,$
and if $\phi', \psi': C_0(I)\to A$ are two \hm s such that
${\mathrm{D}}_w(\phi',\psi')<\dt_0,$
then  there exists a unitary $u'\in {\widetilde A}$ such that 
\beq
\|(u')^*\psi'(f)u'-\phi'(f)\|<\ep/64\rforal f\in {\cal F}_0,
\eneq
where ${\cal F}_0=\{\iota'_I, f_{\sigma_0/2}(\iota'_I)\}$ and where  $\iota_I'$ is the function
defined in \ref{RR2},

Put $\dt=(\ep/4)\dt_0>0.$ Let $\phi, \psi: C_0((0,1])\to {\widetilde{A}}$ be two \hm s which satisfy  \eqref{28a} 
for $\dt.$

Now suppose that  $\pi(a)=\pi(b)=\lambda$
for some $\lambda\in (0,1].$ Let $x=a-\lambda\cdot 1_{\widetilde A}$ and $y=b-\lambda \cdot 1_{\widetilde A}.$
Note that $sp(x), sp(y)\subset[-\lambda, 1-\lambda].$
If $f\in C_0([-\lambda, 0))\oplus C_0((0, 1-\lambda])\subset C([-\lambda, 1-\lambda]),$ 
then $f(0)=0.$ Therefore $\pi(f(x))=\pi(f(y))=0.$ 
Define  two \hm s $\phi_1, \psi_1
: C_0([-\lambda, 0))\oplus C_0((0, 1-\lambda])\to A$ by $\phi_1(f)=f(x)$ and $\psi_1(f)=f(y)$ 
for all $f\in C_0([-\lambda, 0))\oplus C_0((0, 1-\lambda]).$

Define $c_-\in C_0([-\lambda, 0)\oplus C_0((0, 1-\lambda])$ by $c_-(t)=\max(-t, 0)$ ($c_-(t)=-t$ 
in $ [-\lambda, 0)$ and $c_-(t)=0$ in $(0,1-\lambda]$), and  $c_+(t)=\max(t,0).$
 Then $\phi_1(c_+(x))=(a-\lambda)_+,$ 
 $\phi_1(c_-(x))=(a-\lambda)_-,$ and $\psi_1(c_+)=(b-\lambda)_+$ and $\psi_1(c_-(x))=(b-\lambda)_-.$
Let $\phi_{1+}={\phi_1}|_{C_0((0,1-\lambda])},$  $\psi_{1,+}={\psi_1}|_{C_0((0,1-\lambda])},$ 
$\phi_{1-}={\phi_1}|_{ C_0([-\lambda, 0))},$ and $\psi_{1,-}={\psi_1}|_{ C_0([-\lambda, 0))}.$


Note that 
\beq\label{Lconeuniq-2}
&&\|f_\sigma(c_-)c_- -c_-\|<\ep/64,\,\, \|c_-f_\sigma(c_-)-c_-\|<\ep/64,\andeqn\\
&&\|f_\sigma(c_+)c_+-c_+\|<\ep/64,\,\, \|c_+f_\sigma(c_+)-c_+\|<\ep/64
\eneq
for all $0<\sigma\le \sigma_0.$
Let us also assume that, {{for all $0<\sigma\le \sigma_0,$}}
\beq\label{Lconeuniq-3}
&&\|f_\sigma(c_-)^{1/2}c_--c_-\|<\ep/64,\,\, \|c_-f_\sigma(c_-)^{1/2}-c_-\|<\ep/64,\andeqn\\
&&\|f_\sigma(c_+)^{1/2}c_+-c_+\|<\ep/64,\,\, \|c_+f_\sigma(c_+)^{1/2}-c_+\|<\ep/64.
\eneq

Let us consider the case $\lambda\ge \ep/2$ and $\lambda<1-\ep/2$  first.
Let $I_+=(0, 1-\lambda]$ and $I_-=[-\lambda, 0).$
Recall that $h_{I_+}(t)={t\over{1-\lambda}}$ and $h_{I_-}(t)={-t\over{\lambda}}$ (see \ref{RR2}).
Therefore the  condition that ${\mathrm{D}}_w(\phi, \psi)<\dt=(\ep/4)\dt_0$ implies (see \ref{RR2}) that 
\beq\label{28c}
{\mathrm{D}}_{w, I_+}(\phi_{1,+}, \psi_{1,+})=D_w(\phi_+\circ h_{I_+}^*, \phi_+\circ h_{I_+}^*)<\dt_0.
\eneq
Note this holds in ${\mathrm{Cu}}(A).$ 
We also have that
\beq\label{28d}
{\mathrm{D}}_{w, I_-}(\phi_{1,-}, \psi_{1,-})<\dt_0. 
\eneq
Put ${\cal F}_1'=\{f_{\sigma_0/2}(c_-), c_-\}$ and 
${\cal F}_2'=\{f_{\sigma_0/2}(c_+), c_+\}.$
By the choice of $\dt_0,$ there are unitaries 
$u_1, u_2\in {\widetilde A}$ 
such that
\beq\label{Lconeuniq-1}
\|u_i^*\phi_1(f)u_i-\psi_1(f)\|<\ep/64\rforal f\in {\cal F}_i',\,\,\,i=1,2.
\eneq
By replacing $u_i$ by $\af_i u_i$ for some $\af_i\in \T,$ we may assume 
that $\pi(u_i)=1,$ $i=1,2.$
Put $z=\phi_1(f_{\sigma_0/2}(c_-)^{1/2})u_1\psi_1(f_{\sigma_0}(c_-)^{1/2})+\phi_1(f_{\sigma_0/2}(c_+)^{1/2})u_2\psi_1(f_{\sigma_0}(c_+)^{1/2})\in A.$
Keep in mind that $\phi_1(c_+)\phi_1(c_-)=0$ and $\psi_1(c_+)\psi_1(c_-)=0.$
Then we have 
\beq\label{Lconeuniq-5}
\|z^*\phi_1(f)z-\psi_1(f)\|<\ep/16\rforal f\in {\cal F}_0'.
\eneq
We also have, by \eqref{Lconeuniq-1}
\beq\label{Lconeuniq-6}
&&\psi_1(f_{\sigma_0}(c_-)^{1/2})u_1^*\phi_1(f_{\sigma_0/2}(c_-))u_1\psi_1(f_{\sigma_0}(c_-)^{1/2})\\
&&\approx_{\ep/64}
\psi_1(f_{\sigma_0}(c_-)^{1/2})\psi_1(f_{\sigma_0/2}(c_-))\psi_1(f_{\sigma_0}(c_-)^{1/2})=\psi_1(f_{\sigma_0}(c_-)).
\eneq
Similarly, 
\beq\label{Lconeuniq-7}
\psi_1(f_{\sigma_0}(c_+)^{1/2})u_1^*\phi_1(f_{\sigma_0/2}(c_+))u_1\psi_1(f_{\sigma_0}(c_+)^{1/2})\approx_{\ep/64}
\psi_1(f_{\sigma_0}(c_+)).
\eneq

It follows from Lemma 5 of \cite{Pedjot87} (see also \ref{PeduC}  here for convenience) that 
there exist a unitary $u\in {\widetilde A}$ such 
that
\beq
\|u|z|-z\|<\ep/64.
\eneq
Combining  this with \eqref{Lconeuniq-5}, \eqref{Lconeuniq-6}  and \eqref{Lconeuniq-7}, we estimate that
\beq
\|u^*\phi_1(f)u-\psi_1(f)\|<\ep\rforal f\in {\cal F}.
\eneq

\noindent
Thus, there exists a unitary $w\in {\widetilde A}$ such that
$\|w^*xw-y\|<\ep.$
Then 
\beq
w^*aw &=&w^*(x+\lambda 1_{\widetilde A})w\approx_{\ep} y+w^*(\lambda 1_{\widetilde A})w=b.
\eneq

For the case $\lambda>1-\ep/2,$ note that we have reduced the general  case
(of this)
 to the case   ${\cal F}=\{\iota\}.$
Choose a 1-1 continuous function $h: (0,1]\to (0,\lambda]$ such that
\beq
\|h-\iota\|<\ep/2\andeqn \|\iota-h\circ \iota\|<\ep/2.
\eneq
Consider the composed  maps $\phi_1=\phi\circ h$ and $\psi_2=\psi\circ h.$ 
This reduces the problem to the case $\lambda=1.$ So there exists only one interval
$[-1, 0).$ 
{{In}} this case we can choose $\dt$
depending only on $\ep/2$ not on $\lambda.$ 

In the case $\lambda<\ep/2,$ one has a unitary $u\in {\widetilde A}$
such that
$\|u^*\phi_1(c_+)u-\psi_1(c_+)\|<\ep/64.$
Then 
\beq
u^*\phi(\iota)u\approx_{\ep/2} u^*\phi_1(c_+)u\approx_{\ep/64} \psi_1(c_+)\approx_{\ep/2} \psi(\iota).
\eneq


\end{proof}

\begin{cor}\label{Tconeuniq}
Let 
$C=C_0([0,1))\otimes {\cal K}$ and  let
$A$ be a stably projectionless simple \CA\, such that
$M_m(A)$ almost has stable rank one for each $m\ge 1$ and suppose that $A$ has strict comparison
for positive elements and that $\mathrm{QT}(A)=\mathrm{T}(A).$
Then, for any $\ep>0$ and any finite subset ${\cal F}\subseteq C$ there exists $\dt>0$ satisfying the following:

Let $\phi, \psi\hspace{-0.03in}: C\hspace{-0.02in}\to {\widetilde A}\otimes {\cal K}$ be two \hm s such that
${\mathrm{d}}_w(\phi, \psi)<\dt.$
Let $C_0=C_0([0,1))=C_0([0,1))\otimes e_{11}\subseteq C$ denote the 1-1 corner. 
If $\phi(C_0), \psi(C_0)\subseteq {\widetilde A}\otimes e_{11},$ then 
for any $\ep>0$ and any finite subset ${\cal F}\subseteq C,$ there exists a unitary 
$U\in (A\otimes {\cal K})^\sim$ such that
\beq
\|U^*\phi(c)U-\psi(c)\|<\ep\rforal c\in {\cal F},
\eneq
where  $U={\mathrm{diag}}(\overbrace{u,u,...,u}^n,1,1,...),$ where $u\in U({\widetilde{A}})$ for some $n\ge 1.$
\end{cor}

\begin{proof}
We will write ${\mathrm{M}}_n(C_0([0,1)))$ as a \SCA\, of $C,$ ${\mathrm{M}}_n({\widetilde A})$ as a \SCA\, of ${\widetilde A}\otimes {\cal K}$
and ${\mathrm{M}}_n(A)$ as a \SCA\, of $A\otimes {\cal K}$ for all integers $n\ge 1.$
Let $\ep>0$ and ${\cal F}\subseteq C$ be  a finite subset. 
\Wlog, we may assume that ${\cal F}\subseteq {\mathrm{M}}_n(C).$
Furthermore we may write ${\cal F}=\{(c_{i,j})_{n\times n}: c_{i,j}\in {\cal G}\},$
where ${\cal G}\subseteq C_0$ is a finite subset.
We will apply \ref{Lconeuniq} with $\ep/n^2$ in place of $\ep$ and ${\cal G}$ in place of ${\cal F}.$ 
Choose $\dt$  
as provided
for $\ep/n^2$ and ${\cal G}$ (in place of ${\cal F}$) in \ref{Lconeuniq}.
Suppose 
that ${\mathrm{d}}_w(\phi, \psi)<\dt.$

Define $\phi_1=\phi|_{C_0}$ and $\psi_1=\psi|_{C_0}.$
By 
\ref{Lconeuniq}, there exists a unitary 
$u\in {\widetilde A}$ such that
\beq
\|u^*\phi_1(a)u-\psi_1(a)\|<\ep/n^2\rforal a\in {\cal G}.
\eneq
We may assume that $\pi(u)=1,$ where $\pi: {\widetilde A}\to \C$ is the quotient map.
Define 
$$
U=\diag(\overbrace{u,u,...,u}^n, 1,1,...).
$$
Then $U\in (A\otimes {\cal K})^\sim.$ Moreover,
\beq
\|U^*\phi(c)U-\psi(c)\|<\ep\rforal c\in {\cal F}.
\eneq
\end{proof}

\begin{cor}\label{CCC0uniq}
Let $C=C_0((0,1])$  with a strictly positive element $e_c$ and 
$A$ be a stably projectionless simple \CA\, with continuous scale such that $M_m(A)$  
 almost has stable rank one ($m\ge 1$).  Suppose also that $A$ has strict comparison for 
positive elements and that $\mathrm{QT}(A)=\mathrm{T}(A).$
 
 (a) Then, for any $\gamma: {\mathrm{Cu}}(C)\to {\mathrm{Cu}}({\widetilde{A}})$ which is an ordered semigroup \hm\, in ${\mathbf{Cu}}$ 
 such that $\la e_c\ra \le \la 1_{\widetilde{A}}\ra,$  there exists a \hm\, 
 $\phi: C\to {\widetilde{A}}$ such that ${\mathrm{Cu}}(\phi)=\gamma.$
 
 (b)  Let $\phi, \psi: C\to {\widetilde{A}}$ be two  unital \hm s such that ${\mathrm{Cu}}(\phi)={\mathrm{Cu}}(\psi).$
 Then ${\mathrm{Cu}}^\sim (\phi)={\mathrm{Cu}}^\sim (\psi).$
\end{cor}

\begin{proof}

For part (a), for any integer $n\ge 1,$  by Theorem 4 of \cite{RS}, there exists \hm\,  $\psi_n: C\to {\widetilde{A}}$ such that
$d_w(\mathrm{Cu}(\psi_n), \gamma)<1/2^n.$  By the uniqueness theorem \ref{Tconeuniq}, one obtains a sequence
of \hm s $\phi_k: C\to {\widetilde{A}}$ such that $d_w(\mathrm{Cu}(\phi_k), \gamma)\to 0$ and 
$(\phi_k(c))_{k=1}^{\infty}$ is a Cauchy sequence for all $c\in C.$ Let $\phi$ be the limit \hm. 
Then ${\mathrm{Cu}}(\phi)=\gamma.$

For part (b) follows from the uniqueness theorem \ref{Tconeuniq} immediately. 

\end{proof}

\begin{defn}\label{DfDD}
Let $R=R_{1,n}$ be the Razak algebra as below:
\beq\label{dfRn}
\hspace{0.3in}R=R_{1,n}=\{f\in {\mathrm{M}}_n(C([0,1])): f(0)=\af\cdot 1_{M_{n-1}}\andeqn
f(1)=\af\cdot 1_{M_n},\,\af\in \C\}.
\eneq

Put
$$
D=\{f\in {\mathrm{M}}_n(C_0([0,1))): f(0)=\begin{pmatrix} 0_{n-1} & 0\\
                                                                              0 & \af\end{pmatrix}, \af\in \C\}.
$$
Then ${\widetilde{R}}$ is the unitization of $D.$
Denote by $\phi^\sim, \psi^\sim: {\widetilde D}\to {\widetilde A}$ the unital extension of 
$\phi$ and $\psi.$ 
Consider 
\beq
C_0=\{f\in {\mathrm{M}}_n(C_0([0,1)): f(t)=\begin{pmatrix} 0_{n-1} & 0\\
                                                                               0 &  a(t)\end{pmatrix},\, a(t)\in C_0([0,1))\}.
                                                                               \eneq
 Then $C_0\cong C_0([0,1)).$ Note also that $C_0\subset D$ is a  full hereditary \SCA\, 
 Let $j_0: C_0([0,1))\to C_0\subset D$ be the embedding. 
 By 
Brown's theorem
 (see \cite{Br1}), there is an isomorphism 
 $s:C_0\otimes {\cal K}\cong D\otimes {\cal K}.$ Note, from the construction in \cite{Br1}, the isomorphism $s$ 
(given by partial isometry in $M_2(M(D\otimes {\cal K}))$) has the property 
 that  ${\mathrm{Cu}}(s)={\mathrm{Cu}}(j_0).$
  This was discussed in  4.3 of \cite{Robert-Cu}.
 Let $e_C$ be a strictly positive element of $s(C_0([0,1))\otimes e_{1,1})$ and $e_{C_0}$ be a 
 strictly positive element of $C_0\otimes e_{1,1}\subset D\otimes e_{1,1}\subset  D\otimes {\cal K}.$ 
 Then $\la e_C\ra=\la e_{C_0}\ra$ in $D\otimes {\cal K}.$ 
 Since $D\otimes {\cal K}$ has stable rank one,  by  \cite{CEI-CuntzSG} (see also  1.7 of \cite{Lncuntz}), there is 
  a partial isometry $w\in (D\otimes {\cal K})^{**}$ such that $wa, aw^*\in D\otimes {\cal K}$ 
  for all $a\in s(C_0([0,1))\otimes e_{1,1})$ and $w^*aw\in C_0\otimes e_{1,1}.$ 
  Denote by $s(e_{1,1})$ the range projection of $s(C_0([0,1))\otimes e_{1,1}).$ Then $s(e_{1,1})\in M(s(C_0([0,1))\otimes {\cal K}).$
  Also $w^*s(e_{11})w={\bar e}_{1,1},$ where ${\bar e_{1,1}}$ is the range projection of $C_0\subset D.$ 
  Clearly ${\bar e}_{1,1}\in M(D)\subset M(D\otimes {\cal K}).$  Denote by $p_D$ the range projection of $D\otimes e_{1,1}.$ 
  Let $P=1-p_D$ in $M(D\otimes {\cal K}).$ 
  Then  we may write $1-{\bar e}_{11}=((1_D-{\bar e}_{1,1})\otimes 1)\oplus ({\bar e}_{1,1}\otimes P)$ which 
  is Murray-Von Neumann equivalent  to $((1_D-{\bar e}_{1,1})\otimes 1)\oplus ({\bar e}_{1,1}\otimes 1)=1_D\otimes 1$
  in $M(D\otimes {\cal K}).$  
  Note also $1-s(e_{1,1})$ is also Murray-Von Neumann equivalent  to $1,$ as $s(C_0\otimes {\cal K})=D\otimes {\cal K}.$ 
  It follows that there is a partial isometry 
  $W_1\in M(D\otimes {\cal K})$ such that $W_1^*W_1=(1-s(e_{1,1}))$ and $W_1W_1^*=1-{\bar e}_{1,1}$ (see also  Lemma 2.5 of \cite{Br1}).   
 Define $W=W_1\oplus w.$ Then $W\in M(D\otimes {\cal K})$ is a unitary.  Set 
 $j={\mathrm{Ad}}\, W\circ s.$ Note ${\mathrm{Cu}}(j)={\mathrm{Cu}}(\id_{C_0}).$ 
 The additional feature is that $j(C_0([0,1))\subset D.$ 
 
 For any \hm\, $\phi: {\widetilde{D}}\to B$ (for some \CA\, $B$), denote by $\phi$ again 
 for the extension from ${\widetilde{D}}\otimes {\cal K}\to B\otimes {\cal K}.$
 Define $\phi_{C_0}=\phi\circ j: C_0([0,1))\otimes {\cal K}\to B\otimes {\cal K}.$
 If $\psi: R \to A$ (for any \CA\,) is a \hm\, let $\psi^\sim: \widetilde{D}=\widetilde{R}\to \widetilde{A}$ be the extension. 
 We will  use $\psi_{C_0}:=\psi^\sim\circ j: C_0([0,1))\otimes {\cal K}\to \widetilde{A}\otimes {\cal K}.$

\end{defn}

With the definition above, we present the following lemma:

\begin{lem}\label{CCDD}
Let $D$ be as in \ref{DfDD}    and let $A$ be   
a stably projectionless simple \CA\,  such that $M_m(A)$  almost has  stable rank one
for every $m\ge 1,$ has strict comparison and  
$\mathrm{QT}(A)=\mathrm{T}(A).$
 Then, for any $\eta>0$ and any 
finite subset ${\cal S}\subseteq {\widetilde{D}},$ there exists $\dt_0>0$ satisfying the following condition:

For any two unital \hm s $\phi, \psi: {\tilde D}\to {\widetilde{A}},$ if 
\beq
{\mathrm{d}}_w(\phi_{C_0}, \psi_{C_0})<\dt_0,
\eneq
then there exists a unitary $u\in {\widetilde A}$ such that
\beq\label{CCDD-1}
\|u^*\psi(f)u-\phi(f)\|<\eta\rforal f\in {\cal S}.
\eneq
\end{lem}

\begin{proof}
Fix $\eta>0$ and  finite subset ${\cal S}\subset {\widetilde{D}}.$
We may assume that ${\cal S}=\{g+r\cdot 1_{\widetilde{D}}: g\in {\cal G}, r\in K\},$
where ${\cal G}\subset D$ is a finite subset  and $K$ is a finite subset of $\C.$ 
Let $j: C_0([0,1))\otimes {\cal K}\to D\otimes {\cal K}$ be the isomorphism defined in  \ref{DfDD}. 
Let ${\cal F}=j^{-1}({\cal G}).$ Note $j(C_0([0,1))\otimes e_{11})\subset D.$ 
Let $\dt_0>0$ be given for $\eta$ (in place of $\ep$) and ${\cal F}$  by \ref{Tconeuniq}. 
Consider $\phi\circ j, \psi\circ j: C_0([0,1))\otimes {\cal K}\to {\widetilde{A}}\otimes {\cal K}.$
By applying \ref{Tconeuniq}, there exists a unitary $u\in \widetilde{A}$ and integer $n\ge 1$ such that
\beq
\|U^*\phi\circ j(f)U-\psi\circ j(f)\|<\eta\rforal f\in {\cal F},
\eneq
where $U={\text{diag}}(\overbrace{u,u,...u}^n,1,1,...).$
This implies that
\beq
\|U^*\phi(g)U-\psi(g)\|<\eta\rforal  g\in {\cal G}.
\eneq
Since $\phi(g),\, \phi(g)\in \widetilde{A}$ for all $g\in {\cal G},$   we actually have 
\beq
\|u^*\phi(g)u-\psi(g)\|<\eta\,\tforal g\in {\cal G}.
\eneq
Since both $\phi$ and $\psi$ are unital and $u^*r\cdot 1_{\widetilde{A}}u=r\cdot 1_{\widetilde{A}},$
we finally conclude that \eqref{CCDD-1} holds. 

\end{proof}

\begin{thm}\label{Nexthm}
Let $R$ be a Razak algebra and let $A$ be a separable stably projectionless simple \CA\, with continuous scale such 
that ${\mathrm{M}}_n(A)$ almost has stable rank one for every $n\ge 1.$
Suppose  that $A$ also has strict comparison for positive elements and
$\mathrm{QT}(A)=\mathrm{T}(A).$
Let $\gamma: {\mathrm{Cu}}({\widetilde{R}})\to  \N\sqcup S(\widetilde{A})\subset {\mathrm{Cu}}({\widetilde{A}})$ 
(see \ref{DefS} and (vii) of \ref{RRR})
be an ordered semigroup \hm\, in {\text{\bf{Cu}}} with $\gamma(\la 1_{\widetilde{R}}\ra)= \la 1_{\widetilde{A}}\ra$
such that  $\gamma|_{\mathrm{Cu}(R)}\subset \mathrm{Cu}(A)$ and $\gamma(\la a\ra)\not=0$ for 
all $\la a\ra \not=0$ in  ${\mathrm{Cu}}({\widetilde{R}}).$ 
We also assume that $\gamma$ maps elements which cannot be represented by projections 
to the sub-semigroup $S(\widetilde{A}).$ 
Then there exists a \hm\, $\phi: R\to A$ such that 
$\mathrm{Cu}(\phi)=\gamma|_{\mathrm{Cu}(R)}.$

\end{thm}

\begin{proof}
We will keep notation in \ref{DfDD}.
In what follows denote by $\pi: \widetilde{A}\to \C$  as well as $\pi: \widetilde{D}\to \C$ for the quotient maps. 
By the assumption above, we $\mathrm{Cu}(\pi)\circ \gamma|_{\mathrm{Cu}(R)}=0.$ 

Since ${\widetilde{R}}={\widetilde{D}},$ $D$ is a hereditary \SCA\, of $\widetilde{R}.$
It follows that $\mathrm{Cu}(D)$ is an ordered sub-semigroup of $\mathrm{Cu}(\widetilde{R}).$ 
Note also $D\otimes {\cal K}\cong C_0([0,1))\otimes {\cal K}.$
We specify a strictly positive element $e_D(t)=\begin{pmatrix} g(t)\cdot 1_{n-1} & 0\\ 0 & (1-t)\end{pmatrix}$ (for $t\in [0,1]$),
where $g(t)=2t$ if $t\in [0,1/2]$ and $g(t)=2(1-t)$ for $t\in (1/2, 1].$
Note $\gamma(\la e_D\ra)\le \la 1_{\mathrm{A}}\ra.$ 
By part (a) of \ref{CCC0uniq}, there is a \hm\, $\phi_C: C_0([0,1))\otimes {\cal K}\to {\mathrm{A}}\otimes {\cal K}$ 
such that $\mathrm{Cu}(\phi_C)=\gamma\circ {\mathrm{Cu}}(j).$
Then
 $\psi:=\phi_C\circ j^{-1}: D\otimes {\cal K}\to {\widetilde{A}}\otimes {\cal K}$ is a \hm\, such that 
 $\mathrm{Cu}(\phi_C\circ j^{-1})=\gamma|_{\mathrm{Cu}(D)}.$ 
 
 Define $e_n=\psi(f_{1/2^n}(e_D)),$ $n=1,2,....$ Note $\gamma(\la e_D\ra)\le \gamma(\la 1_{\widetilde{D}}\ra =\la 1_{\tilde{A}}\ra.$
 It follows from  Proposition 2.4 of \cite{Rr11} that there exists $x_n\in \widetilde{A}\otimes {\cal K}$ such that
 \beq
 e_n=x_n^*1_{\widetilde{A}}x_n,\,\,\, n=1,2,....
 \eneq
 Put $y_n=1_{\widetilde{A}}x_n.$  Let $y_n=v_n|y_n|$ be the polar decomposition 
 of $y_n$ in $(\widetilde{A})^{**}.$ Then $v_na\in \widetilde{A}\otimes {\cal K}$ for all $a\in \overline{e_n(\widetilde{A}\otimes {\cal K})e_n}$
 and $v_n^*av_n\in \widetilde{A}$ for all $a\in \overline{e_n(\widetilde{A}\otimes {\cal K})e_n},$ 
 $n=1,2,...$ Define $\psi_n: D\to \widetilde{A}$ by $\psi_n(d)=v_n^*e_n\psi(d)e_nv_n$  for all $d\in D.$
 Since $D$ is semiprojective, there exists, for each large $n,$ 
 a \hm\, $h_n: D\to \widetilde{A}$ such that 
 \beq
 \lim_{n\to\infty}\|h_n(d)-\psi_n(d)\|=0\rforal d\in D.
 \eneq
Define $h_n^\sim: \widetilde{D}\to \widetilde{A}$ by defining $h_n(1_{\widetilde{D}}+d)=1_{\widetilde{A}}+h_n(d).$
It is a unital \hm. 

Since $\lim_{n\to\infty} \|e_n^*\psi(d)e_n-\psi(d)\|=0,$ it is easy to compute that
\beq\label{Next-n1}
\lim_{n\to\infty}d_w( \mathrm{Cu}((h_n^\sim)_{C_0}), \mathrm{Cu}(\phi_{C_0}))=0 
\eneq
Let ${\cal F}_n\subset {\widetilde{D}}$ be finite subsets such that ${\cal F}_n\subset {\cal F}_{n+1}$ and 
$\cup_{n=1}^{\infty}{\cal F}_n$ is dense in $\widetilde{D}.$
It follows from  \ref{CCDD}  that there exists a 
subsequence  $\{n_k\}$ and 
a sequence of unitaries $u_k\in \widetilde{A}$ such 
that,
\beq
\|{\mathrm{Ad}}\,u_{k+1}\circ  h_{n_{k+1}}(d)-{\mathrm{Ad}}\, u_k\circ h_{n_k}(d)\|<1/2^k\rforal d\in {\cal F}_k,\,\, k=1,2,....
\eneq
It follows that $({\mathrm{Ad}}\, u_k\circ h_{n_k}(a))_{k=1}^{\infty}$ a Cauchy sequence for each $d\in \widetilde{D}.$
Let $H(a)$ be the limit. Then $H$ defines a unital \hm\, from $\widetilde{D}$ to $\widetilde{A}.$
By \eqref{Next-n1}, $\mathrm{Cu}(H_{C_0})=\mathrm{Cu}(\phi_C).$
Since $D\otimes {\cal K}\cong C_0([0,1))\otimes {\cal K},$ we then have 
$\mathrm{Cu}(H|_D)=\gamma|_{\mathrm{Cu}(D)}.$

Since $\mathrm{Cu}(\pi)\circ \gamma|_{\mathrm{Cu}(R)}=0,$ 
if  $\mathrm{Cu}(H)=\gamma,$   then
$H|_R\subset A.$  Therefore it remains to show $\mathrm{Cu}(H)=\gamma.$ To show 
that, we will apply part (b) of \ref{CCC0uniq}.

It follows  by  \cite{Tsang-W} that there is a separable simple \CA\, $B$ which is an inductive limit 
of Razak algebras with continuous scale such  that ${\mathrm{T}}(B)={\mathrm{T}}(A).$  Note that $B$ has stable rank one and 
$K_0(B)=\{0\}.$
By 6.2.3 of \cite{Robert-Cu} (see also  7.3 of \cite{eglnp}), $\mathrm{Cu}(\widetilde{B})=\N\sqcup L({\widetilde{B}}).$ 
It follows by \ref{Tnote2}  there is an ordered semigroup isomorphism $\gamma_b: L(\widetilde{B})\to S(\widetilde{A}).$ 
This extends to  an ordered semigroup isomorphism $\gamma_B :\mathrm{Cu}(\widetilde{B})\to  \N\sqcup S(\widetilde{A}).$  
It then extends an ordered semigroup isomorphism $\gamma_B^\sim: \mathrm{Cu}^\sim(\widetilde{B})
\to \Z\cup S^\sim(\widetilde{A})$ defined by 
$\gamma_B(m)=m$ and  $\gamma_B^\sim(x-k\la 1_{\tilde B}\ra)=\gamma_b(x)-k\la 1_{\tilde A}\ra.$
(Recall that $\widetilde{B}$ has stable rank one and, by 3.16 of \cite{Robert-Cu}, $\mathrm{Cu}(\widetilde{B})$ embedded into 
$\mathrm{Cu}^\sim({\widetilde{B}})$ and, by (vi) of \ref{RRR}, $S(\widetilde{A})$ embedded into $S^\sim(\widetilde{A}).$)
Moreover it induces an isomorphism $\gamma_{B\sim}: \mathrm{Cu}^\sim(B)\to S^\sim(A).$ 
Let $\gamma_b^{-1},$ $\gamma_B^{-1}$ and $\gamma_{B\sim}^{-1}$ be the inverse maps of 
$\gamma_b,$ $\gamma_B$ and $\gamma_{B\sim}.$

Define $\gamma_\sim: \mathrm{Cu}^\sim(D)\to S^\sim(A)$ by 
$\gamma_\sim(\la a\ra-n\la 1_{\widetilde{D}}\ra)=\gamma(\la a\ra)-n\la 1_{\widetilde{A}}\ra$ for all $a\in \mathrm{Cu}({\widetilde{D}})$
which are not represented by projections (recall also $V({\tilde D})=\N$ and $K_0(D)=K_0(C_0([0,1))=\{0\}.$) and 
$n=\la \pi(a)\ra<\infty.$
This extends $\gamma|_{\mathrm{Cu}(D)}.$  We can also define 
$\gamma^\sim: \mathrm{Cu}^\sim(\widetilde{D})\to \Z\sqcup S^\sim ({\tilde A})$ (see \ref{RRR}) by
$\gamma^\sim(m\la 1_{\tilde D}\ra)=m\la 1_{\tilde A}\ra$ and 
$\gamma^\sim(\la a\ra- k\la 1_{\tilde D}\ra)=\gamma(\la a\ra)-k\la 1_{\tilde D}\ra$ for all 
$\la a\ra\in \widetilde{A}$ which are not represented by projections. Note, in fact, since both 
$\tilde{D}$ and $\widetilde{A}$ are unital,  $\gamma^\sim$ is uniquely determined by $\gamma$
(see 3.1 of \cite{Robert-Cu}).
Note that $\gamma^\sim|_{\mathrm{Cu}^\sim(D)}=\gamma_\sim,$ 
$\gamma^\sim|_{\mathrm{Cu}({\tilde D})}=\gamma$ and $\gamma_\sim|_{\mathrm {Cu}(D)}=\gamma|_{\mathrm{Cu}(D)}.$
Recall that $\N\sqcup S(\widetilde{A})\cong \mathrm{Cu}(\widetilde{B}),$ 
Note that, by the assumption, $\gamma_\sim,$ $\gamma^\sim$ are ordered semigroup \hm s in {\bf{Cu}}.

Note $\widetilde{D}$ is a 1-dimensional NCCW and $\widetilde{B}$ has stable rank one. 
By Theorem 1.0.1 of \cite{Robert-Cu}, there is a unital \hm\, $\Psi_{d,b}: \widetilde{D}\to \widetilde{B}$ such that 
$\mathrm{Cu}^\sim(\Psi_{d,b})=(\gamma_B^\sim)^{-1}\circ \gamma^\sim.$

Let $D_1=H(D)$ and let  $\imath_{D_1}: D_1\to \widetilde{A}$ be the embedding. 
Denote also $\imath_{D_1}: {\widetilde{D}_1}\to \widetilde{A}$ the unital extension.
 Since $\gamma$ is strictly positive, $\imath_{D_1}\circ H|_D$
is an isomorphism.
Then there exists a \hm\, $\Psi_{D_1}: D_1\to {\widetilde{B}}$ such that 
$\mathrm{Cu}(\Psi_{D_1})=\gamma_B^{-1}\circ  \mathrm{Cu}(\imath_{D_1}).$
 Let $\Psi_{D_1}^{-1}: \Psi|_{D_1}(D_1)\to D_1$ 
be the inverse of $\Psi_{D_1}.$ 
Then $\mathrm{Cu}(\Psi_{D_1}\circ H|_D)=\gamma_B^{-1}\circ \gamma|_{\mathrm{Cu}(D)}.$ 
In particular, $\mathrm{Cu}(\Psi_{d,b}\circ j)=\gamma_B^{-1}\circ \gamma\circ \mathrm{Cu}(j).$ 
It follows from  part (b) of \ref{CCC0uniq} that 
$\mathrm{Cu}^\sim(\Psi_{d,b}\circ j)=\mathrm{Cu}^\sim(\Psi_{C_1}\circ H\circ j).$
Note that $j$ is also an isomorphism from $C_0([0,1))\otimes {\cal K}$ onto $D\otimes {\cal K}.$
It follows that  $\mathrm{Cu}^\sim (\Psi_{D_1}\circ H|_{D})=\mathrm{Cu}^\sim(\Psi_{d,b}|_D).$
In other words,  $\mathrm{Cu}^\sim (\Psi_{D_1}\circ H|_{D})=(\gamma_B^\sim)^{-1}\circ \gamma^\sim|_{\mathrm{Cu}^\sim(D)}.$ 
Since $\Psi_{D_1}^{-1}\circ \Psi_{D_1}\circ H=H,$ it follows that 
\beq
\mathrm{Cu}^\sim(H)|_{\mathrm{Cu}^\sim(D)}= \gamma^\sim|_{\mathrm{Cu}^\sim(D)}=\gamma_\sim,
\eneq
where $\mathrm{Cu}^\sim(H): \mathrm{Cu}^\sim({\widetilde{D}})\to \mathrm{Cu}^\sim(\widetilde{A})$ is the map induced by
$H.$

\hspace{-0.1in}
Now let $x\in \mathrm{Cu}(\widetilde{D})$ with $\la \pi(x)\ra=n<\infty.$
Then
\beq
\mathrm{Cu}(H)(x)&=&
\mathrm{Cu}^\sim(H)(x-n\la 1_{\tilde D}\ra)+
\mathrm{Cu}^\sim(H)(n\la 1_{\tilde D}\ra)\\
&=&\gamma_\sim(x-n\la 1_{\tilde D}\ra)+n\la 1_{\tilde A}\ra
=\gamma^\sim(x-n\la 1_{\tilde D}\ra)+\gamma^\sim(n\la 1_{\tilde D}\ra\\
&=&
\gamma^\sim(x)=\gamma(x).
\eneq
It follows that $\mathrm{Cu}(H)=\gamma.$

\end{proof}

\begin{thm}\label{TTTTappendix}
Let $A$ be a  separable  simple stably projectionless \CA\, 
with
 continuous scale such that $M_m(A)$ has almost stable rank one for all $m\ge 1.$
   Suppose that 
 $\mathrm{QT}(A)={\mathrm{T}}(A)$
 and $\mathrm{Cu}(A)=\mathrm{LAff}_+({\mathrm{T}}(A)).$ 
Suppose also that $B$  is a simple \CA\, which is an inductive limit
of Razak algebras with injective connecting maps, with continuous scale  and  with ${\mathrm{T}}(A)={\mathrm{T}}(B).$  Then 
there exists  a \hm\, $\phi: B\to A$ which maps strictly positive elements to strictly positive elements 
and which induces the identification  ${\mathrm{T}}(A)={\mathrm{T}}(B).$ 
 
\end{thm}

\begin{proof}
Let us construct the required \hm\, $\phi.$ 

Note ${\mathrm{Cu}}(A)= \mathrm{LAff}_+({\mathrm{T}}(A))$ and ${\mathrm{Cu}}(B)=\mathrm{LAff}_+({\mathrm{T}}(B)).$ 
Denote by $\Lambda: {\mathrm{T}}(A)\to {\mathrm{T}}(B)$ the affine homeomorphism. 
Then $\Lambda$ induces 
an ordered semigroup isomorphism 
 $\lambda_0: {\mathrm{Cu}}(B)\to {\mathrm{Cu}}(A)$ 
 in ${\mathbf{Cu}}.$

Fix strictly positive elements $e_B$ of $B$  and $e_A$ of $A,$ respectively.
Then $\lambda_0(\la e_B\ra)\le \la e_A\ra.$
Consider the sub-semigroup $S=S({\widetilde A})\subseteq {\mathrm{Cu}}({\widetilde A})$  defined in \ref{DefS}. 
Note that, by \ref{Tnote2rk1},  $\mathrm{Cu}(\widetilde{B})=\N\sqcup L({\widetilde{B}}).$ 
By \ref{Tnote2}, this induces an ordered semigroup isomorphism $\lambda_1: {\mathrm{Cu}}({\widetilde B})
\to \N\sqcup S({\widetilde A})\subseteq {\mathrm{Cu}}({\widetilde A})$ with $\lambda_1(\la 1_{\tilde B}\ra)=\la 1_{\tilde A}\ra$
(see also (vii) of \ref{RRR}).
Write $B=\lim_{n\to\infty} (R_n, \imath_n)$ (see \cite{Tsang-W}), where each $\imath_n$ is injective. 
Let $\gamma_n: {\mathrm{Cu}}({\widetilde R_n})\to {\mathrm{Cu}}({\widetilde A})$ be given by $\lambda_1\circ {\mathrm{Cu}}(\imath_{n, \infty}).$ 
Note that $\gamma_{n+1}\circ {\mathrm{Cu}}(\imath_n)=\gamma_n$ and $\gamma_n(\la 1_{\tilde R_n}\ra)=\la 1_{\tilde A}\ra.$
It follows from \ref{Nexthm} that there is a unital \hm\, $\phi_n: {\widetilde{R_n}}\to \widetilde{A}$ such 
that $\mathrm{Cu}(\phi_n)=\gamma_n$ and $\phi_n|_{R_n}\subset A,$ $n=1,2,....$
Note also $\mathrm{Cu}(\phi_{n+1}\circ \imath_n)=\mathrm{Cu}(\phi_n),$ $n=1,2,....$ 
We also have $\gamma_n(\mathrm{Cu}(R_n))\subset \mathrm{Cu}(A).$ If $x\in \mathrm{Cu}(\widetilde{R}_n)$
is not represented by a projection, neither $\mathrm{Cu}(\imath_{n, \infty})(x).$ It follows 
that $\gamma_n(x)\subset S(\widetilde{A}),$ $n=1,2,....$

Let  $(\ep_n)$  be a decreasing sequence of positive numbers with $\sum_{n=1}^{\infty} \ep_n<\infty.$ 
Let ${\cal F}_n\subseteq R_n$ be finite subsets such that $\imath_n({\cal F}_n)\subseteq {\cal F}_{n+1},$
$n=1,2,...,$ and we assume that $\bigcup_{n=1}^{\infty}\imath_{n, \infty}({\cal F}_n)$ is dense in $B.$
By \ref{CCDD}, there exists a sequence of uniaties $u_n\in {\widetilde{A}}$ such that
\beq
\|\mathrm{Ad}\, u_n\circ \phi_{n+1}(b)-{\mathrm{Ad}}\circ \phi_n(b)\|<1/2^n\rforal b\in {\cal F}_n,
\eneq
$n=1,2,....$ Then $({\mathrm{Ad}}\, u_n\circ \phi_n(b))_{n=1}^{\infty}$ is a Cauchy sequence in $A$ for each 
$b\in B.$
Let $\phi(b)$ be the limit (for each $b\in B$). Then $\phi$ is a \hm\, from $B$ to $A$ such 
that $\mathrm{Cu}(\phi)=\lambda_1.$ From the definition of $\lambda_1,$ we see that $\phi$ meets the requirement.

\end{proof}

\begin{cor}\label{Tappendix}
Let $A$ be a  separable  simple \CA\, 
which 
has finite nuclear dimension and  continuous scale. 
Then there exist a simple \CA\, $B$ which is an inductive limit
of Razak algebras with injective connecting maps and  with ${\mathrm{T}}(A)={\mathrm{T}}(B)$
and a \hm\, $\phi: B\to A$ which maps strictly positive elements to strictly positive elements 
and which induces the identification  ${\mathrm{T}}(A)={\mathrm{T}}(B).$ 
\end{cor}

\begin{proof}
First, the existence of such a \CA\, $B$  with ${\mathrm{T}}(B)={\mathrm{T}}(A)$ is given by \ref{R2}.
It follows from \cite{Winter-Z-stable-02}
that $A$ is ${\cal Z}$-stable and has strict comparison for positive elements.
Moreover, by \cite{Rob-0}, ${\mathrm{M}}_r(A)$ (for 
every
integer $r\ge 1$) and $A\otimes {\cal K}$  
have almost stable rank one.
By Lemma 6.5 of \cite{ESR-Cuntz},  $\mathrm{Cu}(A)=\text{LAff}_+({\widetilde{\mathrm T}}(A)).$ Thus \ref{TTTTappendix}
applies.
\end{proof}

\bibliographystyle{amsalpha}

\begin{thebibliography}{Btrace}


\bibitem{AAP} C.  A. Akemann, J.  Anderson and G. K. Pedersen, {\em Excising states of $C^*$-algebras},
 Canad. J. Math. {\bf 38} (1986), 1239--1260.

\bibitem{Alf} E. M. Alfsen, {\em Compact convex sets and boundary integrals}, Ergebnisse der Mathematik und ihrer Grenzgebiete, Band 57, Springer-Verlag, New York-Heidelberg, 1971. x+210 pp. 


\bibitem
{Bla-WEP}
B.~E. Blackadar, \emph{Weak expectations and nuclear C*-algebras},
Indiana Univ.~Math.~J. \textbf{27} (1978),  1021--1026.. 
  
  
  \bibitem
{Btrace} B. E.  Blackadar, {\em Traces on simple AF \CA s},  J. Funct. Anal.  {\bf 38} (1980),  156--168.
 \bibitem{BC}  B. E. Blackadar and J.  Cuntz, {\em The structure of stable algebraically simple $C^*$-algebras},
  Amer. J. Math. {\bf 104} (1982),  813--822.
  
\bibitem{BH} B. E. Blackadar and D. E. Handelman, {\em Dimension functions and traces on \CA s},  J. Funct. Anal. 
{\bf 45} (1982),  297--340.


\bibitem{BR} B. E.  Blackadar and M.  R\o rdam, {\em Extending states on preordered semigroups and the existence of quasitraces on $C^*$-algebras},  J. Algebra  {\bf 152} (1992),  240--247.
  
\bibitem {Br1}  L. G. Brown, {\em Stable isomorphism of hereditary subalgebras of $C^*$-algebras},  Pacific J. Math.  {\bf 71} (1977),  335--348.

\bibitem {BrPerT} N. P. Brown, F. Perera and A. S. Toms, The Cuntz semigroup, the Elliott conjecture, and dimension functions on C*-algebras, J. Reine Angew. Math. {\bf621} (2008), 191--211.
\bibitem{CE} A. Ciuperca and G. A. Elliott, {\em A remark on invariants for \CA s of stable rank one},
Internat. Math. Res. Not.  {\bf 5} (2008), Article ID rnm158, 33 pages. 

\bibitem
{CEI-CuntzSG}
K.~T. Coward, G.~A. Elliott, and C.~Ivanescu, \emph{The {C}untz semigroup as an
 invariant for {C*}-algebras}, J. Reine Angew. Math. \textbf{623} (2008),
 161--193. \MR{2458043 (2010m:46101)}


\bibitem
{DE-KK-Asy}
M.~Dadarlat and S.~Eilers, \emph{Asymptotic unitary equivalence in
  {KK}-theory}, K-Theory \textbf{23} (2001),  305--322. 

\bibitem
{DE-classification}
M.~D{\u{a}}d{\u{a}}rlat and S.~Eilers, \emph{On the classification of nuclear
  {C*}-algebras}, Proc. London Math. Soc. (3) \textbf{85} (2002), 
  168--210. 
  
\bibitem{DL} 
M.~D{\u{a}}d{\u{a}}rlat and T. Loring, \emph{A universal multicoefficient theorem for the $\textrm{Kasparov}$ groups}, Duke Math. J. \textbf{84} (1996), 355--377.
  168--210. 



\bibitem{LPcw}
S. Eilers, T. Loring, and G. K. Pedersen,
\emph{Stability of anticommutation relations: an application of noncommutative $\mbox{CW}$ complexes},
J. Reine Angew. Math.
\textbf{499} (1998), 101--143.

\bibitem
{Ell-AT-RR0}
G.~A. Elliott, \emph{On the classification of {C*}-algebras of real rank zero},
J. Reine Angew. Math.
\textbf{443} (1993), 179--219.

\bibitem
{point-line}
\bysame, \emph{An invariant for simple $\mbox{C}$*-algebras}, 
Canadian Mathematical Society. 1945--1995, Vol. 3, Canadian Math. Soc., Ottawa, ON, 1996, pp. 61--90 (English, with English and French summaries).



\bibitem
{EGLN-ASH}
G.~A. Elliott, G.~Gong, H.~Lin, and Z.~Niu, \emph{The classification of simple
  separable unital $\textrm{$\mathcal Z$}$-stable locally {ASH} algebras},
J. Funct. Anal. (2017),  5307--5359.

\bibitem
{EGLN-DR}
\bysame, \emph{On the classification of simple amenable {C*}-algebras with
  finite decomposition rank, {II}}, preprint.
 arXiv:1507.03437.
  
\bibitem
{eglnp}  
\bysame, 
{\em Simple stably projectionless \CA s 
of generalized tracial rank one},  J. Noncommutative Geometry, to appear, \, arXiv:1711.01240.
  

\bibitem
{EK-BDF}
G.~A. Elliott and D.~Kucerovsky, \emph{An abstract
  {V}oiculescu--{B}rown-{D}ouglas-{F}illmore absorption theorem}, Pacific J.
  Math. \textbf{198} (2001),  385--409. 


\bibitem
{EN-K0-Z}
G.~A. Elliott and Z.~Niu, \emph{On the classification of simple amenable
  {C*}-algebras with finite decomposition rank}, ``Operator Algebras and their
  Applications: A Tribute to Richard V.~Kadison" (R.~S. Doran and E.~Park, eds.), Contemporary Mathematics, vol. 671, Amer. Math. Soc., 2016, pp. 117--125.



\bibitem
{ESR-Cuntz}
G.~A. Elliott, L.~Robert, and L.~Santiago, \emph{The cone of lower
  semicontinuous traces on a {C*}-algebra}, Amer. J. Math \textbf{133} (2011),
   969--1005.

\bibitem
{Gabe-BDF}
J.~Gabe, \emph{A note on nonunital absorbing extensions}, Pacific J. Math.
  \textbf{284} (2016), 383--393. 

\bibitem{GJS} G. Gong, X. Jiang and H.  Su, \emph{Obstructions to ${\cal Z}$-stability for unital simple \CA s}, Canad. Math. Bull. {\bf 43} (2000),  418--426

\bibitem
{GL1} G. Gong and H.  Lin, {\em Almost multiplicative morphisms and K-theory},
 Internat. J. Math.  {\bf 11} (2000),  983--1000.

\bibitem{GLpp1}
 G. Gong and H.  Lin, {\em  On classification of non-unital simple amenable \CA s, I}, preprint,
 arXiv:1611.04440.


\bibitem{Gd} K. R. Goodearl, {$K_0$ of multiplier algebras of \CA s with real rank zero},  K-Theory {\bf 10} (1996), no. 5, 419--489


\bibitem
{GLN-TAS}
G.~Gong, H.~Lin, and Z.~Niu, \emph{Classification of finite simple amenable
  {$\mathcal Z$}-stable {C*}-algebras}, preprint, arXiv:1501.00135.

\bibitem{HR} J. Hejlmborg and M. R\o rdam, {On Stability of $C^*$-algebras}, 
J. Funct. Anal. {\bf 155} (1998), 153--170.

\bibitem{HL} S. Hu and H. Lin, {\em  Distance between unitary orbits of normal elements in simple $C^*$-algebras of real rank zero} J. Funct. Anal. {\bf 269} (2015),  355--437.



\bibitem
{Jacelon-W}
B.~Jacelon, \emph{A simple, monotracial, stably projectionless {C*}-algebra},
  J. London Math. Soc. (2) \textbf{87} (2013),  365--383. 


\bibitem
{KP0}
E.~Kirchberg and N.C.~Phillips,
\emph{Embedding of exact {C*}-algebras in the {C}untz algebra {$\mathcal O_2$}},
  J. Reine Angew. Math. \textbf{525} (2000), 17--53. 

\bibitem{KW} E.  Kirchberg
and W. Winter, {\em Covering dimension and quasidiagonality}, 
Internat. J. Math.  {\bf 15}  (2004),  63--85.

\bibitem
{Lncs1} H.   Lin, {\em Simple \CA s  with continuous scales and simple corona algebras},
 Proc. Amer. Math. Soc.  {\bf 112} (1991), 871--880.
 
  \bibitem{Lnbk} H. Lin, {\em  An introduction to the classification of amenable \CA s},  World Scientific Publishing Co., Inc., River Edge, NJ, 2001. xii+320 pp. ISBN: 981-02-4680-3.
 
 \bibitem
 {Lnsuniq} H. Lin, {\em Stable approximate unitary equivalence of homomorphisms}, J. Operator Theory
{\bf 47} (2002), 343--378.

  
\bibitem
{Lncs2} 
H.  Lin, {\em Simple corona \CA s}, Proc. Amer. Math. Soc. {\bf 132}  (2004), 3215--3224.

 \bibitem
 {Lnauct} H. Lin, {\em An approximate universal coefficient theorem}, Trans. Amer. Math. Soc. {\bf 357} (2005), 3375--3405.

\bibitem
{Ln-hmtp} H.  Lin, {\em Homotopy of unitaries in simple \CA s with tracial rank one},  J. Funct. Anal. {\bf 258} (2010),  1822--1882.

\bibitem{Lncuntz} H. Lin, {\em  Cuntz semigroups of \CA s of stable rank one and projective Hilbert modules}, preprint,
(2010), \, arXiv:1001.4558. 


\bibitem
{Lncbms} H. Lin, {\em  From the Basic Homotopy Lemma to the Classification of C*-Algebras},
     a CBMS Lectures Notes, preprint.
\bibitem{LZ} H.  Lin and S. Zhang, {\em On infinite simple \CA s},  J. Funct. Anal. {\bf 100}, (1991),  221--231.

 \bibitem{Pedjot87} G. K. Pedersen, {\em Unitary extensions and polar decompositions in a \CA},  J. Operator Theory {\bf 17} (1987),  357--364.


\bibitem
{Ph1}
N.C. Phillips, \emph{A classification theorem for nuclear purely infinite simple {C*}-algebras}, Doc. Math. \textbf{5} (2000), 49--114.

\bibitem{Pbook}  G. K. Pedersen, {\em $C^*$-Algebras and Their Automorphism Groups},
 London Mathematical Society Monographs, 14, Academic Press, Inc.,
 London-New York, 1979. ix+416 pp. ISBN: 0-12-549450-5.

\bibitem
{Razak-W}
S.~Razak, \emph{On the classification of simple stably projectionless
  {C*}-algebras}, Canad. J. Math. \textbf{54} (2002), 138--224.

\bibitem
{Rn} J. R. ~Ringrose, {\em Exponential length and exponential rank in \CA s}, Proc. Royal
Soc. Edinburgh, Sect. A, {\bf 121} (1992), 55--71.

\bibitem
{Robert-Cu}
L.~Robert, \emph{Classification of inductive limits of 1-dimensional {NCCW}
  complexes}, Adv. Math. \textbf{231} (2012),  2802--2836. 

\bibitem
{Rob-0}
\bysame, \emph{Remarks on {$\mathcal{Z}$}-stable projectionless {C*}-algebras},
  Glasgow. Math. J. \textbf{58} (2016),  273--277. 

\bibitem{RS} L. Robert and L. Santiago, {\em Classification of $C^*$-homomorphisms from $C_0(0,1]$ to a $C^*$-algebra}, 
J. Funct. Anal. {\bf 258} (2010), 869--892.

\bibitem{Rr11} M.   R\o rdam, {\em On the structure of simple \CA s tensored with a UHF-algebra. II},  J. Funct. Anal. {\bf 107} (1992),  255--269.


\bibitem
{Ror-Z-stable}
M.~R{\o}rdam, \emph{The stable and the real rank of {$\mathcal Z$}-absorbing
  {C*}-algebras}, Internat. J. Math. \textbf{15} (2004), no.~10, 1065--1084.

\bibitem{RW} M.  R\o rdam and W.  Winter, {\em The Jiang-Su algebra revisited} . J. Reine Angew. Math. 
{\bf 642} (2010), 129--155. 

\bibitem
{T-0-Z}
A.~Tikuisis, \emph{Nuclear dimension, {$\mathcal {Z}$}-stability, and algebraic
  simplicity for stably projectionless {C*}-algebras}, Math. Ann. \textbf{358}
  (2014), nos.~3-4, 729--778. 
  
\bibitem
{TWW-QD}
A~Tikuisis, S.~White, and W.~Winter, \emph{Quasidiagonality of nuclear
  {C*}-algebras},  Ann. of Math. (2) {\bf 185} (2017),  229--284.
  

{
\bibitem{Toms-ASH}
A.S. Toms, \emph{Comparison theory and smooth minimal C*-Dynamics}, Commun.~Math.~Phys.~\textbf{289}, 401--433 (2009)
}

\bibitem
{TW-D}
A.~S. Toms and W.~Winter, \emph{Strongly self-absorbing {C*}-algebras}, Trans.
  Amer. Math. Soc. \textbf{359} (2007),  3999--4029. 
  
  
  
\bibitem
{Tsang-W}
K.~Tsang, \emph{On the positive tracial cones of simple stably projectionless
  {C*}-algebras}, J. Funct. Anal. \textbf{227} (2005), no.~1, 188--199.
  


\bibitem
{Winter-Z-stable-02}
W.~Winter, \emph{Nuclear dimension and {$\mathcal{Z}$}-stability of pure
  {C*}-algebras}, Invent. Math. \textbf{187} (2012), no.~2, 259--342.

\bibitem
{Winter-TA}
W. ~Winter, \emph{Classifying crossed product {C*}-algebras}, Amer. J. Math.
  \textbf{138} (2016), no.~3, 793--820.

\bibitem
{WZ-OR0}
W.~Winter and J.~Zacharias, \emph{Completely positive maps of order zero}, M\"{u}nster J.~Math. \textbf{2} (2009), 311-324. 




\bibitem
{WZ-ndim}
W.~Winter and J.~Zacharias, \emph{The nuclear dimension of {C*}-algebras}, Adv.
  Math. \textbf{224} (2010), no.~2, 461--498. 

\end{thebibliography}

\providecommand{\bysame}{\leavevmode\hbox to3em{\hrulefill}\thinspace}
\providecommand{\MR}{\relax\ifhmode\unskip\space\fi MR }
\providecommand{\MRhref}[2]{%
  \href{http://www.ams.org/mathscinet-getitem?mr=#1}{#2}
}
\providecommand{\href}[2]{#2}

\end{document}